\documentclass[11pt, a4]{article}
\usepackage{pict2e}
\usepackage{latexsym, amssymb, amsfonts, amsthm, amsmath, xspace,
 pdfsync, a4wide}

\newtheorem{thm}{Theorem}[section]
\newtheorem{prop}[thm]{Proposition}
\newtheorem{lemma}[thm]{Lemma}
\newtheorem{cor}[thm]{Corollary}

\theoremstyle{definition}
\newtheorem{proc}[thm]{Procedure}
\newtheorem{alg}[thm]{Algorithm}
\newtheorem{defn}[thm]{Definition}
\newtheorem{example}[thm]{Example}
\newtheorem{remark}[thm]{Remark}

\newcommand{\GAP}{{\sf GAP}\xspace}
\newcommand{\MAGMA}{{\sc MAGMA}\xspace}

\def\Z{{\mathbb Z}}
\def\N{{\mathbb N}}
\def\R{{\mathbb R}}

\newcommand{\valid}{green-rich\xspace}
\newcommand{\Valid}{Green-rich\xspace}

\newcommand{\ip}{\mathcal{I}(\mathcal{P})}
\newcommand{\cR}{{\cal R}}
\newcommand{\cK}{{\cal K}}
\newcommand{\cS}{\mathcal{S}}
\newcommand{\cI}{{\cal I}}
\newcommand{\cG}{{\cal G}}
\newcommand{\mybar}[1]{\overline{\raisebox{1.2ex}{}#1}}
\newcommand{\cP}{{\cal P}} 
\newcommand{\cQ}{{\cal Q}} 
\newcommand{\cV}{{\cal V}} 

\newcommand{\calD}{{\cal D}}
\newcommand{\cL}{{\cal L}}
\newcommand{\PD}{\mathrm{PD}}
\newcommand{\De}{\mathrm{D}}
\newcommand{\RSym}{\texttt{RSym}\xspace}
\newcommand{\ComputeRSym}{\texttt{ComputeRSym}\xspace}
\newcommand{\RSymVerify}{\texttt{RSymVerify}\xspace}
\newcommand{\RSymVerifyAtPlace}{\texttt{RSymVerifyAtPlace}\xspace}
\newcommand{\RSymIntVerify}{\texttt{RSymIntVerify}\xspace}

\newcommand{\RSymSolve}{\texttt{RSymSolve}\xspace}
\newcommand{\VerifySolver}{\texttt{VerifySolver}\xspace}
\newcommand{\VerifySolverAtPlace}{\texttt{VerifySolverAtPlace}\xspace}
\newcommand{\VerifySolverTrivInt}{\texttt{VerifySolverTrivInt}\xspace}
\newcommand{\RSymSolveTrivInt}{\texttt{RSymSolveTrivInt}\xspace}
\newcommand{\FindEdges}{\texttt{FindEdges}\xspace}
\newcommand{\KBMAG}{{\sf KBMAG}\xspace}
\newcommand{\sfR}{{\sf R}}
\newcommand{\sfG}{{\sf G}}
\newcommand{\bP}{\mathbf{P}}
\newcommand{\bQ}{\mathbf{Q}}
\newcommand{\true}{\texttt{true}\xspace}
\newcommand{\false}{\texttt{false}\xspace}
\newcommand{\fail}{\texttt{fail}\xspace}

\newcommand{\Vertex}{\texttt{Vertex}\xspace}
\newcommand{\OneStep}{\texttt{OneStep}\xspace}
\newcommand{\Blob}{\texttt{Blob}\xspace}

\newcommand{\Area}{\mathrm{Area}\xspace}
\newcommand{\CArea}{\mathrm{CArea}\xspace}
\newcommand{\Post}{\mathrm{Post}\xspace}
\newcommand{\Pre}{\mathrm{Pre}\xspace}

\newenvironment{mylist}{\begin{list}{}{
\setlength{\parskip}{0mm}
\setlength{\topsep}{2mm}
\setlength{\parsep}{0mm}
\setlength{\itemsep}{0.5mm}
\setlength{\labelwidth}{5mm}
\setlength{\labelsep}{2mm}
\setlength{\itemindent}{0mm}
\setlength{\leftmargin}{7mm}
\setlength{\listparindent}{6mm}
}}{\end{list}}

\newenvironment{mylist2}{\begin{list}{}{
\setlength{\parskip}{0mm}
\setlength{\topsep}{2mm}
\setlength{\parsep}{0mm}
\setlength{\itemsep}{0.5mm}
\setlength{\labelwidth}{100pt}
\setlength{\labelsep}{2mm}
\setlength{\itemindent}{0mm}
\setlength{\leftmargin}{12.5mm}
\setlength{\listparindent}{12mm}
}}{\end{list}}

\title{Polynomial-time proofs that groups are hyperbolic}
\author{Derek Holt, Stephen Linton, Max Neunh\"offer, Richard Parker,\\
  Markus Pfeiffer and Colva M. Roney-Dougal$^\ast$\\
\footnotesize{ $\ast =$ corresponding author}}

\date{July 22, 2020}

\begin{document}
\maketitle

\abstract{
It is undecidable in general whether a given finitely presented group
is word hyperbolic. We use the concept of \emph{pregroups},
introduced by Stallings in \cite{Stallings}, to define a new class of
van Kampen diagrams, which represent groups as quotients of 
\emph{virtually} free groups. We then present a
polynomial-time procedure that analyses these diagrams, and either
returns an explicit linear Dehn function for the presentation, or returns
\fail, together with its reasons for failure. 
Furthermore, if our procedure succeeds we are often able to produce
in polynomial time a word problem solver for the presentation that runs in
linear time. Our algorithms have been implemented, and when successful
they are many orders of magnitude faster than \KBMAG, the only comparable
publicly available software. 
}

\medskip


\section{Introduction}
The Dehn function of a
finitely presented group is linearly bounded 
if and only if the group is hyperbolic. We describe a
new, polynomial-time procedure for proving that a group defined by a
finite presentation is hyperbolic, by establishing such a linear upper bound on
its Dehn function.
 Our procedure returns a positive
answer significantly faster than other methods, and in particular 
always terminates in low degree polynomial time, although sometimes it
will terminate with \fail even when the input group is hyperbolic. 
 Our approach has the added
advantage that it can sometimes be carried out by hand, which can
enable one to prove the hyperbolicity of infinite families of groups.

A finitely generated group is (word) hyperbolic if
its Cayley graph is negatively curved as a geometric metric space; that is,
if its geodesic triangles are uniformly slim. There are several good
sources, such as~\cite{MSRInotes}, for an introduction to and development of
the basic properties of hyperbolic groups. Another useful reference for the
specific properties that we need in this paper is~\cite[Chapter 6]{HRR}.
In particular, hyperbolic groups are finitely presentable,
and they admit a Dehn algorithm. Furthermore, for
groups that are defined by a finite presentation, this last condition
 implies that the group is hyperbolic.

We use \cite[Chapter V]{LyndonSchupp} as reference for the theory of
van Kampen diagrams over group presentations, and its application to
groups defined by presentations that satisfy various small cancellation
hypotheses. The arguments used in the proofs of these results
can be formulated in terms of the assignment of
curvature to the vertices, edges and faces of reduced van Kampen diagrams.
The idea is to show that, under appropriate conditions,
the curvature in those parts of the diagrams that are not close to the
boundary is non-positive.  For example, one specific conclusion of
\cite[Theorem 4.4]{LyndonSchupp} is that groups with presentations that satisfy
$C'(\lambda)$ for $\lambda \le 1/6$, or $T(4)$ together with
$C'(\lambda)$ for $\lambda \le 1/4$, have Dehn algorithms.
It is not assumed in these results that the group presentations in
question are finite, but we shall be working only with finite presentations in
this paper, in which case these conditions imply that the group is hyperbolic.

The algorithmic methods developed in this paper involve the assignment of
curvature to van Kampen diagrams in the manner described above, such that
the total curvature of every diagram is $1$. A serious
limitation of methods that rely on small cancellation conditions is that
they are unlikely to be satisfied in the presence of short defining relators
such as powers $x^n$, where $n$ is small and $x$ is a generator.
However, such
relators are present in many of the most interesting group presentations.
Our methods use the theory of \emph{pregroups}, developed by Stallings in
\cite{Stallings}. This theory enable us to remove short relators, 
replacing them with certain other relators of length three (the
pregroup relators), which we then ignore when considering
generalisations of small cancellation. 

Our general aim is to assign the curvature in such a
way that vertices and edges have zero curvature, faces labelled by pregroup
relators have non-positive curvature, and other faces that are not close to the
boundary of the diagram have curvature that is bounded above
by some constant $-\varepsilon < 0$. Our principal theoretical result is
Theorem~\ref{thm:hypercurvature}, which states roughly that, if we can assign
curvature in this manner to all reduced diagrams over the presentation,
then the group has a Dehn function that is bounded above by a linear function
that we can specify explicitly in terms of $\varepsilon$ and various basic
parameters of the defining presentation.  So the group is hyperbolic.

Although we cannot expect such methods to work for all hyperbolic
groups, we can explore a variety of methods of assigning curvature,
which we call \emph{curvature distribution schemes}. In this paper we restrict
attention to a single such scheme, which we call \emph{\RSym}. In
Theorem~\ref{thm:rsymsucceeds}, we apply Theorem~\ref{thm:hypercurvature}
to calculate an explicit linear bound on the Dehn function that is satisfied
in the event that \RSym succeeds in assigning curvature in the
required fashion. As a simple remedy in examples in which \RSym
does not succeed and some, but not all, interior faces of a diagram end up with
zero or positive curvature, we could try to transfer some of the negative
curvature from those faces that already have it to those
that do not. This process is hard to implement in a computer algorithm, but it
can often be done by hand, which significantly increases the applicability of
the methods.

The principal algorithmic challenge is to prove that \RSym succeeds
on a sufficiently large set of reduced diagrams over the input presentation. Our main
algorithm, \RSymVerify, which is
described in  Section~\ref{sec:tester}, attempts to achieve this.
It is technically complicated and involves a detailed study of the possible
neighbourhoods of interior faces in diagrams over the presentation.

For hyperbolic groups, the word problem is
solvable in linear time by a Dehn algorithm. However, the current best
results require as a preprocessing step the computation of the
set $S$ of all words of length up to $8
\delta$ that are trivial in the group, where geodesic triangles in the Cayley graph are
$\delta$-thin. Given a linear bound $\lambda n$ on the Dehn function
$\De(n)$ of a
finitely-presented group $G$, it is therefore theoretically possible to use brute force  to
test all such words for triviality in $G$, and hence to construct the set $S$. 
However, this requires time and
space that are exponential in both $\delta$ and $\lambda$,  and so is
completely impractical. 
We instead devise an additional polynomial-time test, which, if
satisfied, enables the polynomial-time construction of a linear time
 word problem solver. This additional  test is also the basis of future joint work
by the sixth author, which will give a polynomial-time construction of
a quadratic time solver for the conjugacy problem, the second of
Dehn's classic problems. 

\bigskip

Here is a breakdown of the contents of the paper.
In Section~\ref{sec:pregroups}, we summarise the required properties
of pregroups, and define a new kind of presentation, called a
\emph{pregroup presentation}, for a group $G$. It was shown by Rimlinger in
\cite{Rimlinger} that a finitely generated group $H$ is virtually free if
and only if $H$ is the universal group $U(P)$ of a finite pregroup $P$: see
Theorem~\ref{thm:rimlinger}.  Pregroup presentations
enable  us to view the group $G$ as a quotient of a virtually free
group $U(P)$, rather than just as a quotient of a free group, and hence to 
ignore any failures of small cancellation on the defining relators of $U(P)$.

In Section~\ref{sec:diagpg} we define \emph{coloured} van Kampen diagrams 
over these
new pregroup presentations, where the relators of the virtually free
group $U(P)$ (which we collect in a set $V_P$)
are coloured red, and the additional relators (which we
collect in a set
$\mathcal{R}$) 
are green. We show in
Proposition~\ref{prop:red_bound} that, given any coloured van Kampen
diagram $\Gamma$ satisfying a certain technical condition, there
exists a coloured van Kampen diagram $\Gamma'$, with the same boundary
word as $\Gamma$, 
whose area is bounded by an explicit linear function of the
number of green faces of $\Gamma$. Hence, to prove that a group is hyperbolic
it suffices to prove a linear upper bound on the number of green faces
appearing in any reduced coloured diagram of boundary length $n$. 

In Section~\ref{sec:interleave} we show that if we replace our
presentation by a certain related presentation, then we can assume without
loss of generality that each vertex of a coloured diagram is incident
with at least \emph{two} green faces. This property will be critical
to our later curvature analysis, and is automatically
satisfied by diagrams over free groups, since for them all faces are green. 

Section~\ref{sec:curv_dist} is devoted to the definition and general
discussion of curvature distribution schemes.
These provide an overall schema for the
design of many possible methods for proving that a group given by a
finite pregroup presentation  is hyperbolic: since a pregroup
presentation is a generalisation of a standard presentation, these
methods apply to all finite presentations. As mentioned earlier, in
Theorem~\ref{thm:hypercurvature} we characterise how these schema
can produce explicit bounds on the Dehn function.

In Section~\ref{sec:rsym} we present the \RSym curvature distribution scheme
mentioned earlier.  For reasons of space and ease of comprehension,
we restrict attention to this scheme in this paper, but our
approach can be used to define many others.
Theorem~\ref{thm:rsymsucceeds} gives an explicit bound on the
Dehn function of the presentation when 
 \RSym succeeds on all coloured van Kampen diagrams of minimal
 coloured area. 

In Section~\ref{sec:tester} we prove (see Theorems~\ref{thm:rsym} and
\ref{thm:r_sym_complexity}) that under some mild and easily
testable 
assumptions on the set $\cal{R}$ of green relators,
one can test whether \RSym succeeds on \emph{all} of the (infinitely
many) coloured van Kampen diagrams of minimal area.
This test is carried out by our procedure \RSymVerify (Procedure~\ref{proc:RSymVerify}), which runs in time
$O(|X|^5 + r^3 |X|^4|\mathcal{R}|^2)$, where $X$ is the 
set of generators, and $r$
is the length of the longest green relator. 
Our assumptions hold, for example, for all
groups given as quotients of free products of free and finite groups.
We also prove that without these assumptions, one can test whether 
\RSym succeeds on 
all minimal diagrams in
time polynomial in $|X|$ and $r|\cal{R}|$: our procedure to do this is
called 
\RSymIntVerify (Procedure~\ref{proc:RSymIntVerify}). 

In Section~\ref{sec:wp} we go on to consider the word problem. Whilst
a successful run of \RSymVerify or \RSymIntVerify proves an explicit linear bound
on the Dehn function, it is rarely practical to construct a set of Dehn rewrites. We present a low degree
polynomial-time method to construct a word problem solver:
see Theorem~\ref{thm:verify_solver}.
The construction of the solver succeeds in many but not all examples
in 
which \RSym 
succeeds, and the solver itself runs in linear time: see
Theorem~\ref{thm:solver} and Proposition~\ref{prop:better_dehn}. 

In Section~\ref{sec:examples} we consider a variety of
examples of finite group presentations, and show how \RSym 
can be used by hand to prove that
the groups are hyperbolic. In particular, we prove that \RSym
succeeds on groups satisfying any of a wide variety of small
cancellation conditions, we use \RSym to analyse
two infinite families of presentations, and we discuss a range of
possible future applications of \RSym to problems concerning the
hyperbolicity of finitely-presented groups. 

Our procedures have been released as part of both the 
\GAP \cite{GAP} and \MAGMA \cite{MAGMA} computer algebra systems,
and in Section~\ref{sec:implementation} we present runtimes on a variety of
 examples, including some with very large numbers of
generators and relations. Almost none of these examples could have been analysed
using previously existing methods, due to the size of the presentations. 

Since we have introduced many new terms and much new notation, we
conclude with an Appendix containing lists of all new terms, notation, and
procedures. 

\bigskip

As far as we know, the only other publicly available software that can
prove hyperbolicity of a group defined by an  arbitrary finite presentation is
the first author's \KBMAG\ package \cite{KBMAG} for computing automatic
structures.  Hyperbolicity is verified by proving that geodesic bigons in the
Cayley graph are uniformly thin, as described in \cite[Section 5]{Holt96}.
It was proved by Papasoglu in \cite{Papasoglu} that this property implies
hyperbolicity, but it does not provide a useful bound on the Dehn
function. 
An algorithm for computing the ``thinness'' constant for geodesic triangles in
the Cayley graph of a hyperbolic group is described in \cite{EpsteinHolt},
but this is of limited applicability in practice, on account of its high memory
requirements.
Even on the simplest examples, the \KBMAG\ programs involve far too many
computational steps for them to be carried out by hand, and they can
only be applied to individual presentations. The automatic structure
does however provide a fast method (at worst quadratic time) of reducing words
to normal form and hence solving the word problem in the group.

Shortly before submitting this paper, we became aware of a paper by
Lysenok \cite{Lysenok}, which explores similar concepts of
redistributing curvature to prove hyperbolicity to those presented in
Section 5 of this paper. His main
theorem is similar to our Theorem~\ref{thm:hypercurvature}, but the ideas
are less fully developed.

\section{Pregroup presentations}\label{sec:pregroups}

In this section we introduce pregroups, establish some of their
elementary properties, and show that any quotient of a virtually free
group by finitely many relators can be defined by a finite
pregroup presentation.
Pregroups were first defined by Stallings in \cite{Stallings}.

\begin{defn}\label{def:pregroup}
A \emph{pregroup} is a set $P$, with a distinguished element
$1$, equipped with a partial multiplication $(x, y) \rightarrow
xy$ which is defined for $(x, y) \in D(P) \subseteq P \times P$, and with
an involution $\sigma: x \rightarrow x^{\sigma}$, satisfying the following
axioms, for all $x, y, z, t \in P$:
\begin{enumerate}
\item[(P1)] $(1, x), (x, 1) \in D(P)$ and $1 x = x 1 = x$;
\item[(P2)] $(x, x^{\sigma}), (x^{\sigma}, x) \in D(P)$ and
  $x x^{\sigma} = x^{\sigma}x = 1$;
\item[(P3)] if $(x, y) \in D(P)$ then $(y^{\sigma}, x^{\sigma}) \in D(P)$ and
  $(xy)^{\sigma} = y^{\sigma} x^{\sigma}$;
\item[(P4)] if $(x, y), (y, z) \in D(P)$ then $(xy, z) \in D(P)$ if and only
  if $(x, yz) \in D(P)$, in which case $(xy)z = x(yz)$;
\item[(P5)] if $(x, y), (y, z), (z, t) \in D(P)$ then at least one of $(xy,
  z), (yz, t) \in D(P)$.
\end{enumerate}
Since 
we will often be working with words over $P$, if we wish 
to emphasise that two (or more) consecutive letters, say  $x$ and $y$,  of a word $w$
are to be multiplied we shall write $[xy]$. 
\end{defn}

Note that (P2) implies that $1^\sigma=1$, and that (P1), (P2) and (P4)
imply that inverses are unique: if $xy
= 1$ then $y = x^\sigma$.
It was shown in \cite{Hoare}  that (P3) follows from (P1), (P2) and (P4), 
but we include it to
keep our numbering consistent with the literature. 

\begin{defn}\label{def:up}
Let $P$ be a pregroup. We define $X  = P \setminus \{1\}$, and let
$\sigma$ be the involution on $X$. 
 We write $X^\sigma$ to denote $X$, equipped with this involution, but
will sometimes omit the $\sigma$, when the meaning is clear.
We shall write $F(X^\sigma)$ to denote the
group defined by the
presentation $\langle X \mid xx^\sigma : x \in X \rangle$. 
If $\sigma$ has cycle structure $1^k2^l$ on $X$, then $F(X^\sigma)$
is the free product of $k$ copies of $C_2$ and $l$ copies of $\mathbb{Z}$. 

Let $V_P$ 
be the set of all length three relators over $X$ of the form
$\{xy[xy]^{\sigma} : x, y \in X,  (x, y) \in D(P), x \neq
y^{\sigma}\}$. 
The \emph{universal group} $U(P)$ of $P$ is the group given
by
$$
\langle X \mid \{xx^\sigma : x \in X\} \cup   V_P \rangle =
F(X^\sigma)/\langle \langle V_P \rangle \rangle, $$
where $\langle \langle V_P \rangle \rangle$ denotes the normal closure of
$V_P$ in $F(X^\sigma)$.
\end{defn}

Since this presentation of $U(P)$ is on a set of monoid generators that is closed
under inversion, we can and shall write the elements of $U(P)$ as words over $X$,
and use $x^\sigma$ rather than $x^{-1}$ to denote the inverse of $x$.
More generally, if $w = x_1 \ldots x_n  \in F(X^\sigma)$, then
$w^{-1} =_{F(X^\sigma)}  x_n^\sigma x_{n-1}^{\sigma} \cdots x_1^{\sigma}$.

Stallings in \cite{Stallings} defines $U(P)$ as the universal group of $P$ in a
categorical sense: every morphism from $P$ to a group $G$ factors
through $U(P)$. 

\begin{remark}
It is an easy exercise using the pregroup axioms
 to show that if $(x, y) \in D(P)$ and $z  =
[xy]$ then $(y, z^\sigma), (z^\sigma, x) \in D(P)$, so that if
$xyz^\sigma \in V_P$ then all products of cyclic pairs of letters are
defined in $P$.
\end{remark}

\begin{example}\label{ex:free_group}
A pregroup $P$ such that $U(P) = F(X^\sigma)$ is free of rank $n$ can be made
by letting $X$ have $2n$  elements, defining $\sigma$ to be fixed-point-free
on $X$, and letting the only products be $x  x^\sigma = 1$,  $1x = x1 = x$,
and $1\cdot 1 = 1$, for all $x \in X$.
\end{example} 

\begin{example}\label{ex:free_prod}
A pregroup $P$ such that $U(P)$ is the free product of finite groups
$G$ and $H$ can be made as follows. We let $P$ have elements the disjoint
union of $\{1\}$, $G\setminus \{1\}$ and $H \setminus \{1\}$.
 We define $\sigma$ to be
the inversion map on both $G$ and $H$, and to fix $1$. 
We let $D(P) = (G \times G) \cup (H \times H)$, and define all 
products as in the parent groups. 

More generally, if the finite groups $G$ and $H$ intersect in a subgroup $I$,
and again $P = G \cup H$ with inversion and $D(P)$ defined as
before, then $U(P)$ is the amalgamated free product $G *_I H$. 
\end{example}

\begin{defn} \label{def:reduced}
We define a word $w \in X^\ast$ to be \emph{$\sigma$-reduced} if $w$ contains no
consecutive pairs $xx^{\sigma}$ of letters: this is a slight generalisation of
free reduction. We define \emph{cyclically $\sigma$-reduced}
similarly.

The word $w = x_1 \cdots x_n \in X^\ast$ is \emph{$P$-reduced} if either
$n \leq  1$, or $n > 1$ and no pair $(x_i, x_{i+1})$ lies in
$D(P)$. The word $w$
is \emph{cyclically $P$-reduced} if either (i) $n \leq  1$; or (ii) $w$ is
$P$-reduced,  $n > 1$, and $(x_n, x_1) \not\in D(P)$.
\end{defn}

Stallings defines a relation $\approx$ on the set of $P$-reduced words in
$X^\ast$ as follows.
\begin{defn}\label{def:interleave}
Let $v = x_1\cdots x_n \in X^\ast$ be $P$-reduced 
and let  $w = y_1\cdots y_m$ be any word in $X^\ast$.
Then we write $v \approx w$ if $n = m$ and there exist $s_0 = 1, s_1,
\ldots, s_{n-1}, s_n = 1\in P$ such that $(s_{i-1}^{\sigma}, x_i),  (x_i, s_i), 
([s_{i-1}^{\sigma}x_i], s_i) \in D(P)$ for all $i$, and
$y_i = [s_{i-1}^{\sigma}x_is_i]$. We say that $w$ is an \emph{interleave} of $v$. 
In the case when $s_i \ne 1$ for a single value of $i$, we call the
transformation from $v$ to $w$ a {\em single rewrite}.
\end{defn}

Notice that if $(s_{i-1}^{\sigma}, x_i),  (x_i, s_i),  
([s_{i-1}^{\sigma}x_i], s_i) \in D(P)$ then it follows
from (P4) that $(s_{i-1}^{\sigma}, [x_is_i]) \in D(P)$, and that $y_i
= (s_{i-1}^\sigma x_i)s_i = s_{i-1}^\sigma (x_i s_i)$. 

\begin{example}\label{ex:interleave}
Let $G = \langle a \rangle$ and $H = \langle b \rangle$ be cyclic
of order 6 with $a^2=b^2$ and $I := G \cap H = \langle a^2 \rangle$,
and let $P = G \cup H = \{1, a_1, a_3, a_5, b_1, b_3, b_5, i_2, i_4\}$
as in Example~\ref{ex:free_prod} above, with the interpretation $a_j =
a^i$, $b_j = b^i$ and $i_j = a^j = b^j \in I$.  
Let $v = a_1b_1a_3$. Then, by choosing
$s_1 = i_2$ and $s_2 = i_4$, we obtain the interleave $w=a_3 b_3 a_5$. 
\end{example}

\begin{thm}[{\cite[3.A.2.7, 3.A.2.11,  3.A.4.5 \& 3.A.4.6]{Stallings}}]\label{thm:up}
Let $P$ be a pregroup, let $X = P \setminus \{1\}$, and let $v, w \in
U(P)$, with $v$ a $P$-reduced word. Then 
\begin{enumerate}
\item[(i)]   if $v \approx w$ then $w$ is $P$-reduced; 
\item[(ii)]  interleaving is
an equivalence relation on the set of $P$-reduced words over $X$;
\item[(iii)]  each element $g \in U(P)$ can be represented by
a $P$-reduced word in $X^\ast$;
\item[(iv)]  if $w$ is $P$-reduced, then $v$ and $w$ represent the same element of $U(P)$ if and only if $u \approx v$;
in particular, $P$ embeds into $U(P)$.
\end{enumerate}
\end{thm}

\begin{cor}\label{cor:wpuplinear}
Let $P$ be a finite pregroup. Then the word problem in $U(P)$ is soluble in linear time.
\end{cor}

\begin{proof}
The only $P$-reduced word representing $1_{U(P)}$ is the empty
word, so we can solve the word problem in $U(P)$  by reducing words
using the products in $D(P)$. This process is a Dehn algorithm,
which by \cite{DomanskiAnshel}  requires time linear in the length of the input word.
\end{proof}

\begin{defn}\label{def:pregroup_pres}
We now define a new type of presentation, which we shall call a
\emph{pregroup presentation}. 
Let $P$ be a pregroup, let $X = P \setminus \{1\}$, let $\sigma$ be
the involution giving inverses in $X$, and let $\cR
\subset X^\ast$ be a set of cyclically $P$-reduced words.  We write
$$\cP = \langle X^\sigma  \ | \ V_P \ | \ \cR \rangle$$ to define a
group presentation $\cP = \langle X \ | \ \{xx^\sigma \, : \,x \in X\}
\cup  V_P\cup \cR \rangle$ 
on the set $X$ of monoid
generators.
\end{defn}

\begin{example}\label{ex:pregroup_pres}
We construct two pregroup presentations for $G = \langle x, y \ | \
x^\ell, y^m, (xy)^n \rangle$ (with $\ell, m \geq 2$). 
For the first, let $P$ be the pregroup with universal group $C_\ell \ast C_m$,
as in Example~\ref{ex:free_prod}, so that $$P = \{1, x=x_1,x_2,\ldots,x_{\ell-1},
y = y_1,y_2,\ldots,y_{m-1}\}$$ with each $x_i =_{P} x^i$ and
$y_i=_{P} y^i$. Then $\sigma$ fixes $1$, maps each $x_i$ to
$x_{\ell - i}$ and maps each $y_j$ to $y_{m -j}$. The set $V_p$ consists of
all triples $x_i x_j x^{\sigma}_{i+j}$ and $y_i y_j y^{\sigma}_{i+j}$, where $+$
represents modular arithmetic (and no subscript is equal to zero). 
Let $\cR = \{(xy)^n\}$. 
Then $\cP = \langle (P \setminus \{1\})^{\sigma} \mid V_P \mid  \cR \rangle$ is
a pregroup presentation for $G$.

For the second, assume for convenience that $\ell, m \geq 3$, and let
$Q = \{1, x, X, y, Y\}$ be a pregroup with $1$ as the identity, $\sigma$ exchanging
$x$ with $X$ and $y$ with $Y$, and no other products defined. Notice
that $U(Q)$ is a free group of rank $2$.  Let
$\mathcal{S} = \{x^\ell, y^m, (xy)^n\}$. Then $\cQ =
\langle \{x, X, y, Y\}^\sigma  \mid \emptyset \mid \cS \rangle$ is also a pregroup
presentation for $G$. 
\end{example}

\begin{remark}\label{rem:reduced}
We shall assume throughout the rest of the paper that there are no 
relators of the form $x^2$ for $x \in X$ in $\cR$ and that, instead, we
have chosen a pregroup $P$ such that $x^\sigma = x$. This can always
be achieved by, for example, choosing $P$ such that $U(P) =
F(X^\sigma)$.  Notice also that
$\cR \cap V_P = \emptyset$, since each element of $\cR$ is cyclically $P$-reduced.
\end{remark}

We finish this section by considering the applicability of these
presentations. 

\begin{thm}[{\cite[Corollary to Theorem B]{Rimlinger}}]\label{thm:rimlinger}
A finitely generated group $G$ is virtually free if and only if $G$ is the
universal group of a finite pregroup.
\end{thm}

The class of virtually free groups includes amalgamated free products of
finite groups, and HNN extensions with finite base groups, which is the
source of many of the pregroups that are useful in the algorithmic applications
to proving hyperbolicity that are described in this paper. More generally,
a group is virtually free if and only if it is the fundamental group of a
finite graph of groups with finite vertex groups \cite[Proposition 11]{Serre}.

For the remainder of the paper, we shall be working with groups given
by finite pregroup presentations. The following immediate 
corollary shows that this includes all
quotients of virtually free groups by finitely many additional
relators. 

\begin{cor}\label{cor:pgpres}
Let a group $G$  have pregroup presentation $\cP = \langle X^\sigma \,
| \, V_P \, | \, \cR \rangle$, as in
Definition~\ref{def:pregroup_pres}.
 Then $G \cong U(P)/\langle 
\langle \cR \rangle \rangle$, where $\langle \langle \cR \rangle \rangle$
denotes the normal closure of $\cR$ in $U(P)$.
Furthermore, any group that is a quotient of a virtually free group
by finitely many additional relators has a finite pregroup presentation.  
\end{cor}

\section{Diagrams over pregroups}
\label{sec:diagpg}

In this section, we introduce coloured van Kampen diagrams, 
which are a natural generalisation of van Kampen
diagrams to pregroup presentations. After completing the introductory
material, our main result is Proposition~\ref{prop:red_bound}, which 
shows that if a word of length $n$ can be written as a
product of conjugates of $k$ relators from $\cR^{\pm}$ over $U(P)$, then it
can be written as a product of conjugates of $\lambda k + n$ relators from
$\cR^{\pm} \cup V_P$ over $F(X^\sigma)$, where $\lambda$
depends only on the maximum length $r$ of the relators in $\cR$, and
$\cR^{\pm}$ denotes $\cR \cup \{R^{-1} : R \in \cR\}$. 

This section contains many new definitions, and we remind the reader
that the Appendix contains a list of all new terms and notation.

In general we follow standard terminology for van Kampen diagrams, 
as given in \cite[Chapter 5, \S 1]{LyndonSchupp} for example. For clarity, 
we record some definitions that will be useful in what follows. 

\begin{defn}\label{def:van_kampen}
We shall orient each face clockwise. We shall count all incidences with 
multiplicities, for example a vertex may be incident more than once 
with the same face.

In the present article, we shall require our diagrams to be simply
connected. There is therefore a unique  \emph{external face}, and 
its label is the \emph{external word}.  All other faces are
\emph{internal}. 
If an element $x \in X$ is self-inverse in $P$, then $x$ has order $2$
in $U(P)$ and we will identify $x$ with $x^{\sigma}$, so that an edge
may have label $x$ on both sides.
 
We will refer to a nontrivial path of maximal length that is common to two
adjacent faces of a diagram as a {\em consolidated 
edge}. 

We denote the boundary of a face $f$ or a diagram $\Gamma$ by
$\partial(f)$ and $\partial(\Gamma)$, respectively. We consider
$\partial(f)$ (and $\partial(\Gamma)$) to contain both vertices and
edges, but abuse notation and write $|\partial(f)|$ for the number of edges.
An internal face $f$ of a diagram ${\Gamma}$ is a {\em boundary face} if
$|\partial(f) \cap \partial(\Gamma)| \geq 1$. A vertex or edge of $\Gamma$ is a {\em
   boundary} vertex or edge if it is contained in $\partial(\Gamma)$. 
\end{defn} 

\begin{defn}\label{def:coloured_vkd}
A \emph{coloured van Kampen diagram} over the pregroup presentation
$\cP = \langle X^\sigma \,
| \, V_P \, | \, \cR \rangle$ is a van Kampen diagram with edge labels from $X^\sigma$, and
face labels from $V_P \cup \cR^{\pm}$, in which the faces labelled by an
element of $V_P$ are coloured red, and the faces labelled by
an element of $\cR^{\pm}$, together with the external face, are
coloured green. We shall often refer to a coloured van Kampen diagram
as a \emph{coloured diagram}. 

For $v$ a vertex in a coloured diagram $\Gamma$,
we shall write $\delta(v, \Gamma)$ for the degree of $v$,
$\delta_G(v, \Gamma)$ for the \emph{green degree} of $v$: the number of green
faces incident with $v$ in $\Gamma$, and $\delta_R(v, \Gamma)$ for the
number of red faces 
incident with $v$ in $\Gamma$. 
\end{defn}

Notice from our relator set $V_P$ that all red faces are triangles: 
we shall often refer to them as such.
If a word $w$ is a product of conjugates of exactly $k$ relators from
$\cR^{\pm}$ in $U(P)$, then there exists a coloured diagram
for $w$ with exactly $k$ internal green faces.
The proof of this is essentially the same as for standard van Kampen diagrams;
see \cite[Chapter V, Theorem 1.1]{LyndonSchupp}, for example.

\begin{defn}\label{def:area}
The area of a coloured diagram $\Gamma$ is its total number
of internal faces, both red and green, and is denoted
$\Area(\Gamma)$. 
\end{defn}

However, when comparing areas of
diagrams, it is convenient to count green faces first and then red triangles.

\begin{defn}\label{def:c_area}
Let $\Gamma$ be a coloured diagram. 
The {\em coloured area} $\CArea(\Gamma)$ of 
$\Gamma$ is an ordered pair $(a, b) \in \N \times \N$,  where $a$ is the
number of internal green faces of $\Gamma$ and $b$ is the number of red
triangles.
Let $\Delta$ be a coloured diagram with $\CArea(\Delta) = (c, d)$.
We say that  $\CArea(\Gamma) \leq \CArea(\Delta)$ 
 if $a < c$ or if $a = c$ and $b \leq d$.  A diagram has
{\em minimal coloured area} for  a word $w$ if its coloured area is minimal
over all diagrams with boundary word $w$.
\end{defn}

\begin{defn}\label{def:subdiagram}
Let $\Gamma$ be a coloured van Kampen diagram. A \emph{subdiagram} of 
$\Gamma$ is a subset of the
edges, vertices and internal faces of $\Gamma$ which, together with a new
external face coloured green, form a coloured diagram in their own right.
\end{defn}

In particular, 
 we do not allow annular subdiagrams. 

\begin{defn}\label{def:sigma_reduced}
A coloured diagram is {\em semi-$\sigma$-reduced} if no two
distinct adjacent 
faces are labelled by $w_1w_2$ and $w_2^{-1}w_1^{-1}$ for
some relator $w_1w_2 \in V_P \cup \cR^{\pm}$ 
and have a common consolidated edge labelled by $w_1$ and
$w_1^{-1}$.  It is \emph{$\sigma$-reduced} if the same also holds for a single
face adjacent to itself.
\end{defn}

Our definition of a $\sigma$-reduced coloured diagram corresponds 
to the usual definition
of a reduced diagram;
see \cite[p241]{LyndonSchupp}, for example. Unfortunately, the proof
in \cite[Chapter V, Lemma 2.1]{LyndonSchupp} that a diagram of minimal area is reduced
 breaks down in our situation for faces adjacent to themselves,
because two cyclic conjugates of a $P$-reduced word can be mutually
inverse in $F(X^\sigma)$: we shall eventually get around this problem, and
the first step in this direction is Proposition~\ref{prop:invtriv}.

First, however, we make a natural generalisation of
$\sigma$-reduction.

\begin{defn}\label{def:P_reduced}
A coloured diagram is {\em semi-$P$-reduced} if no two 
distinct adjacent green 
faces are labelled by $w_1w_2$ and $w_3^{-1}w_1^{-1}$
and have a common consolidated edge labelled by $w_1$ and $w_1^{-1}$, where
$w_2$ and $w_3$ are equal in $U(P)$ (which, by Theorem \ref{thm:up} is
equivalent to $w_2 \approx w_3$). 
\end{defn}

Notice in particular that a semi-$P$-reduced diagram is
semi-$\sigma$-reduced. 

\begin{prop}\label{prop:semired}
Let $\Gamma$ be a coloured diagram with boundary word $w$. Then
there exists a semi-$P$-reduced coloured
diagram $\Delta$ with boundary word $w$ such that $\CArea(\Delta)
\leq \CArea(\Gamma)$. If $\Gamma$ is not already
semi-$P$-reduced then this inequality is strict.
Notice in particular that the diagram $\Delta$ is semi-$\sigma$-reduced. 
\end{prop}

\begin{proof}
 If a coloured diagram $\Gamma$ fails to be semi-$P$-reduced, then $\Gamma$
contains
 two adjacent green
faces labelled by $w_1w_2$ and $w_3^{-1}w_1^{-1}$, as in
Definition~\ref{def:P_reduced}. 

Since $w_2 =_{U(P)} w_3$, we can remove the consolidated edge labelled
$w_1$, identify any consecutive edges with inverse labels, 
and fill in the resulting region, with label the cyclically
$\sigma$-reduced word resulting from $w_2w_3^{-1}$, by a number of
red triangles, yielding a diagram $\Gamma_1$.  It is possible that
$\Area(\Gamma_1) > \Area(\Gamma)$, but $\CArea(\Gamma_1) <
\CArea(\Gamma)$, since $\Gamma_1$ has two fewer green faces than
$\Gamma$. The process terminates at a semi-$P$-reduced diagram
$\Delta$. 
\end{proof}

The following result will enable us to restrict our attention to
$\sigma$-reduced diagrams later in the paper.

\begin{prop}\label{prop:invtriv}
Let $G$ have pregroup presentation $\cP$. Suppose that there exists a 
coloured diagram $\Gamma$  over $\cP$ that
contains a face $f$ that is adjacent to itself 
in such a way that $\Gamma$ is not
$\sigma$-reduced. 
Then $f$ is green.

Furthermore, there exists $t \in X$ such that
\begin{mylist}
\item[(i)] $t^\sigma = t$ and $t =_G 1$;
\item[(ii)]
all coloured  diagrams $\Delta$ of minimal coloured area with boundary label
$t$ are $\sigma$-reduced and 
semi-$P$-reduced, and satisfy $\CArea(\Delta) < \CArea(\Gamma)$.
\end{mylist}
\end{prop}

\begin{proof}
The only way that a red triangle could be adjacent to itself is if an
element of $X$ is trivial in $U(P)$, contradicting
Theorem~\ref{thm:up}(iv). Hence $f$ is green.

Let $f$ be adjacent to itself via a consolidated edge labeled
$w_1$, so that reading $\partial(f)$ from the side of the edge
labelled $w_1$ gives $w_1w_2$, and from the other side gives
$w_1^{-1}w_2^{-1}$. Then $w_2$ contains a subword $w_1^{-1}$, so there are words
$v_1$ and $v_2$ with $w_2=_{F(X^\sigma)} v_1w_1^{-1}v_2$ and $w
=_{F(X^\sigma)} w_1w_2 
=_{F(X^\sigma)}  w_1v_1w_1^{-1}v_2$.

We have assumed that $w_2$ is the inverse of the cyclic subword of $w$ that
starts just after $w_1^{-1}$ and finishes just before it, namely $v_2w_1v_1$.
So $w_2 =_{F(X^\sigma)} v_1^{-1}w_1^{-1}v_2^{-1}$ and hence, since there is no
cancellation in these products, $v_1=_{F(X^\sigma)} v_1^{-1}$
and $v_2 =_{F(X^\sigma)}  v_2^{-1}$. 
Both $v_1$ and $v_2$ are $\sigma$-reduced, so 
each is an $F(X^\sigma)$-conjugate of a self-inverse element of
$X^\sigma$.

The face $f$ in $\Gamma$ encloses regions with boundary labels $v_1$
and $v_2$, so the corresponding involutions are trivial in $G$.
In fact these regions are subdiagrams $\Gamma_1$ and $\Gamma_2$ of $\Gamma$,
and at least one of them, say $\Gamma_1$, does not contain $f$, and hence
has coloured area less than that of $\Gamma$. By identifying inverse
pairs of edges on
 $\partial(\Gamma_1)$ we obtain a diagram with the same coloured area as
$\Gamma_1$, and with boundary label a self-inverse element of
$X^\sigma$. 

Now let $\Delta$ be a diagram with smallest possible coloured area such that
its boundary label is some $t \in X$ with $t = t^{\sigma}$. Then $\Delta$ is
semi-$P$-reduced by Proposition \ref{prop:semired}, and it must be
$\sigma$-reduced, 
since otherwise we could repeat the above argument and obtain such a diagram
with smaller coloured area.  Furthermore, $\CArea(\Delta) \leq
\CArea(\Gamma_1) < \CArea(\Gamma)$.
\end{proof}

\begin{defn}\label{def:delta_Ge}
A \emph{plane} graph is a planar graph embedded in the
plane, so that the faces are determined. 
We denote the sets of all edges, vertices, and internal faces of a
coloured van Kampen diagram or plane graph
$\Gamma$ by $E(\Gamma)$, $V(\Gamma)$ and $F(\Gamma)$, respectively,
and if the faces of $\Gamma$ are coloured then we write $F(\Gamma) = F_G(\Gamma) \cup F_R(\Gamma)$, where
$F_G(\Gamma)$ and $F_R(\Gamma)$ denote the sets of internal green and
red faces of $\Gamma$.
For an edge $e$ of $\Gamma$, we define $\delta_G(e,\Gamma)$ to be the number of
green faces incident with $e$.
So $\delta_G(e,\Gamma)= 0$, $1$ or $2$.
\end{defn}

Recall from Definition~\ref{def:van_kampen} that we count incidences of faces
with edges and vertices with multiplicity: we do the same in the
corresponding plane graph. 
Recall from Definition~\ref{def:coloured_vkd} that the external face
of any coloured diagram $\Gamma$ is green, and so contributes 1 to the value of
$\delta_G(e,\Gamma)$ for each edge on its boundary. Recall also 
Definition~\ref{def:coloured_vkd}  of $\delta_G(v, \Gamma)$. 

In the proof of the next result, we use the well-known
Euler formula $|V(\Gamma)|+|F(\Gamma)|-|E(\Gamma)| = 1$
(we have $1$ rather than $2$ because $F(\Gamma)$ does not include the
external face).

\begin{prop} \label{prop:graph}
Let $\Gamma$ be a simply-connected
plane graph, where all faces have been coloured red
or green and the unique external face is green. 
 Assume that the boundary of every red face has length $3$.
Then, with the notation of Definition~\ref{def:delta_Ge}, we have
\[ \sum_{e \in E(\Gamma)} \delta_G(e,\Gamma) = |F_R(\Gamma)| + 2\left( 1 - |F_G(\Gamma)| +
  \sum_ {v \in V(\Gamma)} (\delta_G(v, \Gamma)-1) \right). \]
\end{prop}

\begin{proof}
Let $E := |E(\Gamma)|$, $V := |V(\Gamma)|$, $F := |F(\Gamma)|$,
$F_R := |F_R(\Gamma)|$ and $F_G := F_G(\Gamma)$.
The proof is by induction on $F_R$, so suppose first that $F_R=0$.
Then $F_G = F$ and
$\sum_{v \in V(\Gamma)} (\delta_G(v, \Gamma) - 1) = 2E - V$, so
the right hand side of the above equation is $2(1  - F + 2E - V)$, which 
by Euler's formula
is equal to $2E = \sum_{e \in E} \delta_G(e,\Gamma)$.  

For the inductive step, it is enough to prove that the equation remains
true when we change the colour of an internal triangle from green to
red, since any two-face-coloured graph can be made by first colouring 
all faces one colour and then changing the colour of some faces.
The triangle has three sides, so when it changes colour,
$\sum_{e \in E} \delta_G(e,\Gamma)$ decreases by $3$, the value of
$F_R$ increases by $1$ and $F_G$ decreases by $1$. The triangle is
incident with three vertices, so 
$\sum_{v \in V} (\delta_G(v, \Gamma)-1)$ decreases by $3$.
Thus, the formula still holds. \end{proof}

\begin{lemma}\label{lem:no_1}
Let $\Gamma$ be a coloured  diagram
with cyclically $\sigma$-reduced boundary word. Then 
$\Gamma$ contains no vertices of degree $1$. 
\end{lemma}

\begin{proof}
We assumed in Definition~\ref{def:pregroup_pres} that all boundary labels 
of internal green faces are cyclically $P$-reduced, and hence in particular
are cyclically $\sigma$-reduced. Furthermore,  all boundary labels of
red triangles are $\sigma$-reduced by the definition of $V_P$. We have now 
assumed that the boundary word of $\Gamma$ is cyclically $\sigma$-reduced. If 
$v$ is a vertex in $\Gamma$ of degree $1$, then the boundary of the face
containing $v$ has label containing a subword $xx^{\sigma}$, where $x$
is the label of the unique edge incident with $v$. This is a
contradiction. 
\end{proof}

A number of the results that we prove later require the assumption that
certain elements of  $X$ are not trivial in $G$, as in
Proposition~\ref{prop:invtriv}. A potential problem with this assumption
is that, on the one hand the triviality of generators is known to be undecidable
in general in finitely presented groups, but on the other hand our
algorithms need to be
able to test whether it holds on the presentations that we want to test
for hyperbolicity.

Fortunately, as we shall see in 
Theorem~\ref{thm:NoVletter1}, there is a way of avoiding these difficulties.
Although we cannot assume in our proofs that our diagrams
contain no loops, we can typically assume that,
on consideration of a minimal counterexample to whatever we are trying to
prove, if there is a such a loop in the diagram $\Gamma$ under consideration,
then $\Gamma$ has the smallest possible coloured area among all such diagrams.
The following definition formalises this assumption.

\begin{defn}\label{def:loop_min}
 A letter $x \in X$ such that $x^{\sigma} = x$ or $x$ is a
  letter of a relator in $V_P$ is called a \emph{$V^{\sigma}$-letter}. 
A coloured diagram $\Gamma$ over $\cP$ is
{\em loop-minimal} if every coloured diagram $\Delta$ over $\cP$ that contains a loop
labelled by some $V^{\sigma}$-letter $x \in X$ satisfies
$\CArea(\Delta) \geq \CArea(\Gamma)$. 
\end{defn}

\begin{example}
\begin{enumerate}
\item If no element of $X$ is trivial in $G$, then all coloured
diagrams over $\cP$ are loop-minimal. 
\item If $U(P)$ is free, and constructed as in Example~\ref{ex:free_group}, then there are no $V^\sigma$-letters,
  and so all diagrams over $\cP$ are loop-minimal. 
\item Assume that at least one $V^\sigma$-letter is trivial in $G$,
  and let $\Delta$ be a diagram of smallest coloured area with
  boundary word a single $V^\sigma$-letter. Then the set of
  loop-minimal coloured diagrams is the set of 
  diagrams with coloured area less than or equal to $\CArea(\Delta)$. 
\end{enumerate}
\end{example}

Our algorithms to prove hyperbolicity will be able to verify for
certain presentations that no $V^\sigma$-letter is trivial in $G$, so that all diagrams over
these presentations are loop-minimal: see
Theorem~\ref{thm:NoVletter1}. This approach will succeed for almost
all of the examples in 
Section~\ref{sec:examples}, and for the remainder we shall demonstrate
other methods to
prove that no element of $X$ is trivial in $G$. 

In the remainder of this section, we shall work towards a proof of
Proposition~\ref{prop:red_bound}: 
the total area of a loop-minimal coloured  diagram is
bounded above by a
linear function of the boundary length and the number of green
faces. First we shall prove two results which 
allow us to assume that every vertex $v$ in a coloured diagram
$\Gamma$ satisfies $\delta_G(v, \Gamma) \geq 1$. 

\begin{lemma}\label{lem:threereds}
Let $\Gamma$ be a coloured diagram over $\cP$ with boundary word $w$. Assume that $\Gamma$ contains a
vertex $v$ with three consecutive adjacent red triangles, and that none of the
edges in any red triangle incident with $v$ is a loop based at $v$.

 Then there exists a diagram $\Delta$ over $\cP$ with boundary word
 $w$, with the same green faces as
$\Gamma$,  satisfying $\CArea(\Delta) \leq \CArea(\Gamma)$,
 in which $v$ is incident with at least one fewer red
triangle than it is in $\Gamma$, and in which none of the edges of any of the red
triangles incident with $v$ is a loop based at $v$.
\end{lemma}
\begin{proof}
Let $f_1$, $f_2$ and $f_3$ be the three consecutive red triangles,
with edge labels as in the
left hand picture of Figure~\ref{fig:2}. We assume that the edges
labelled $\{a, c, e\}$ are pairwise distinct, and that the edges $\{c,
e, g\}$ are pairwise distinct, but we allow the possibility that the
edge labelled $a$ is equal to the edge labelled $g$ (so that $v$ has
degree $3$). Some of the other vertices and edges in
Figure~\ref{fig:2} may not be distinct, but notice that the assumption
on loops based at $v$ implies that 
none of $v_1,v_2,v_3$ and $v_4$ coincide with $v$, and hence that
$f_1$, $f_2$ and $f_2$ are pairwise distinct. 

\begin{figure}
\begin{center}
\begin{picture}(320,65)(0,-5)
\put(0,0){\circle*{3}}
\put(30,45){\circle*{3}}
\put(60,0){\circle*{3}}
\put(90,45){\circle*{3}}
\put(120,0){\circle*{3}}
\put(57,-9){$v$}
\put(-3,-9){$v_1$}
\put(27,49){$v_2$}
\put(87,49){$v_3$}
\put(117,-9){$v_4$}

\put(60,0){\vector(-1,0){30}}
\put(30,0){\line(-1,0){30}}
\put(28,-8){$a$}
\put(24,10){\small{$f_1$}}

\put(60,0){\vector(-2,3){15}}
\put(45,22.5){\line(-2,3){15}}
\put(38,18){$c$}
\put(55,26){\small{$f_2$}}

\put(60,0){\vector(2,3){15}}
\put(75,22.5){\line(2,3){15}}
\put(78,18){$e$}
\put(88,10){\small{$f_3$}}

\put(60,0){\vector(1,0){30}}
\put(90,0){\line(1,0){30}}
\put(88,-8){$g$}

\put(0,0){\vector(2,3){15}}
\put(15,22.5){\line(2,3){15}}
\put(5,18){$b$}

\put(30,45){\vector(1,0){30}}
\put(60,45){\line(1,0){30}}
\put(58,49){$d$}

\put(90,45){\vector(2,-3){15}}
\put(105,22.5){\line(2,-3){15}}
\put(112,18){$f$}

\put(180,0){\circle*{3}}
\put(210,45){\circle*{3}}
\put(240,0){\circle*{3}}
\put(270,45){\circle*{3}}
\put(300,0){\circle*{3}}
\put(237,-9){$v$}

\put(240,0){\vector(-1,0){30}}
\put(210,0){\line(-1,0){30}}
\put(208,-8){$a$}
\put(224,32){\small{$f_2'$}}

\put(180,0){\vector(2,1){45}}
\put(225,22.5){\line(2,1){45}}
\put(215,24){$x$}
\put(227,10){\small{$f_1'$}}

\put(240,0){\vector(2,3){15}}
\put(255,22.5){\line(2,3){15}}
\put(258,18){$e$}
\put(268,10){\small{$f_3$}}

\put(240,0){\vector(1,0){30}}
\put(270,0){\line(1,0){30}}
\put(268,-8){$g$}

\put(180,0){\vector(2,3){15}}
\put(195,22.5){\line(2,3){15}}
\put(185,18){$b$}

\put(210,45){\vector(1,0){30}}
\put(240,45){\line(1,0){30}}
\put(238,49){$d$}

\put(270,45){\vector(2,-3){15}}
\put(285,22.5){\line(2,-3){15}}
\put(292,18){$f$}

\end{picture}
\end{center}
\caption{Reducing the degree of $v$}
\label{fig:2}
\end{figure}
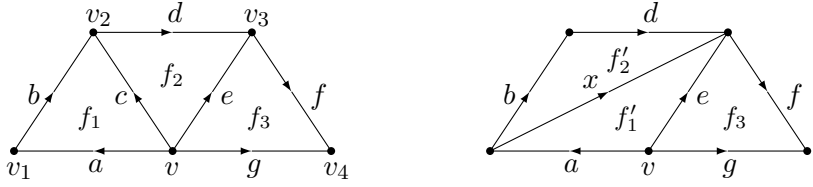

Then $a^{\sigma} =_P bc^{\sigma}$, $d =_P c^{\sigma}e$
and $g =_P ef$. Hence by Axiom (P5) at least one of  the pairs
$([bc^{\sigma}], e), ([c^{\sigma}e], f)$ is in $D(P)$. 
Suppose that $([bc^{\sigma}], e) \in D(P)$ (the other case is
similar), and let $x = [bc^{\sigma}e]$. Then
\[x =_P a^{\sigma}e =_P [bc^{\sigma}]e =_P b[c^{\sigma}e] =_P bd.\]

If $x =_P 1$ then $e = a$ and $d = b^{\sigma}$, so $\Gamma$ is not
semi-$\sigma$-reduced, and $v$ has degree at least four (since
$a^\sigma e$ is not a subword of the label of any red triangle).
 We may delete faces $f_1$ and $f_2$,
identifying $v_1$ with $v_3$, the directed edge labelled $b$ with the one
labelled $d^{\sigma}$, and the edge labelled $a$ with the one labelled
$e$, to produce a diagram $\Delta$. 
Then $\delta_R(v, \Delta) = \delta_R(v, \Gamma) -2$, 
and $\Delta$ has no loops based at $v$ (since
the only amalgamated vertices are $v_1$ and $v_3$, which do not
coincide with $v$). 

Otherwise, $x \neq 1$, and so $axe^{\sigma}, x^{\sigma}bd \in V_P$,
and there is a diagram $\Delta$ in which the triangles $f_1$ and
$f_2$ have been replaced by triangles $f_1'$ and $f_2'$ respectively, with
labels $axe^{\sigma}$ and $x^{\sigma}bd$ (as in the right hand picture
of Figure \ref{fig:2}). We have deleted an edge $\{v,
  v_2\}$, and added an edge $\{v_1, v_3\}$. Since $v_1$ and $v_3$ are
not equal to $v$, the number of
red and green faces of $\Delta$ is identical to that of $\Gamma$, but
$\delta(v, \Delta) < \delta(v, \Gamma)$. Since the only added edge is
not incident with $v$, 
the condition on loops based at $v$ still holds. 
\end{proof}

\begin{thm}\label{thm:one_green}
Let $\Gamma$ be a loop-minimal coloured diagram over $\cP$
with cyclically $\sigma$-reduced boundary word $w$. Then
$w$ is the boundary word of a loop-minimal coloured diagram $\Delta$ with 
$\CArea(\Delta) \le \CArea(\Gamma)$, the same green faces as $\Gamma$,
and 
every vertex $v$ of $\Delta$ satisfies $\delta_G(v, \Delta) \geq 1$. 
Furthermore, if $\Gamma$ contains a vertex $v$ with $\delta_G(v, \Gamma) = 0$, then
$\CArea(\Delta) < \CArea(\Gamma)$.
\end{thm}

\begin{proof}
Suppose that $v$ is a vertex of $\Gamma$  such that $\delta_G(v, \Gamma) =
0$. We shall show that it is possible to modify
$\Gamma$ to produce a diagram $\Delta$ with the same boundary word
and green faces, but
fewer red triangles. 
Since $\CArea(\Delta) < \CArea(\Gamma)$, the new diagram $\Delta$ must
be loop-minimal.
A vertex of degree $1$ does not exist in $\Gamma$ by Lemma~\ref{lem:no_1}.

We prove the result first under the assumption that the
three vertices of any red triangle are distinct.
Suppose that $\delta(v, \Gamma) = 2$. Let
the outgoing edge labels be $a$ and $b$. Then one of the incident triangles has
label $a^{\sigma}b[b^{\sigma}a]$ and the other has label
$b^{\sigma}a[a^{\sigma}b]$, and so $\Gamma$ is not
semi-$\sigma$-reduced.
Thus there is a diagram $\Delta$ with boundary
word $w$ in which $v$ and both triangles have been removed from $\Gamma$, and
in which the edges labelled $a^{\sigma}b$ and
$b^{\sigma}a$ have been identified. Then $\Delta$ has
the same green faces as $\Gamma$, but two fewer red triangles.

Suppose $\delta(v, \Gamma) \geq 3$. Since $\delta_G(v, \Gamma) = 0$, the vertex $v
\not\in \partial(\Gamma)$, so the loop-minimality of $\Gamma$ ensures
that there are no loops based at $v$. Hence
 by Lemma~\ref{lem:threereds},
we can replace $\Gamma$ by a diagram $\Gamma_1$ with 
boundary word $w$ and the same green
faces as $\Gamma$,  such that $\CArea(\Gamma_1) \leq \CArea(\Gamma)$ and
$\delta_R(v, \Gamma_1) \leq \delta_R(v, \Gamma) -
1$, and in which there are no loops based at $v$.  
By repeating this process, we eventually reduce to a diagram $\Gamma_k$
in which $\delta(v, \Gamma_k) = 2$, and the
coloured area can then be reduced, yielding $\Delta$.  
This completes the proof in the case that no red triangle of $\Gamma$
meets itself at one or more vertices.

Now we allow for the possibility that not all vertices of a red
triangle are distinct.
Figure \ref{fig:tri_meet} shows the possible
configurations of a red triangle in $\Gamma$ in which two or more of the
vertices coincide.  The red triangle is labelled \textsf{T} and other regions
of the diagram, which  may contain additional vertices, edges, and faces, are
labelled $\Theta$, $\Theta_1$ or $\Theta_2$. 

In the third of these configurations, two consecutive letters of a word
in $V_P$ $\sigma$-cancel, which contradicts the definition of $V_P$.
In the second and the fourth, there is an internal loop labelled by a
single letter from $V_P$,
contradicting our assumption of loop-minimality. 

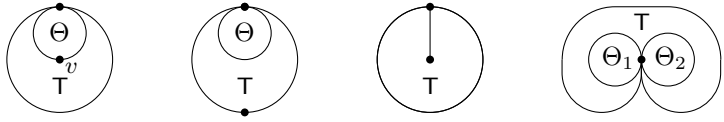
\begin{figure}
\begin{center}
\begin{picture}(260,70)(0,5)
\put(10,50){\circle*{3}}
\put(10,30){\circle*{3}}
\put(10,30){\circle{40}}
\put(10,40){\circle{20}}
\put(7,17){\footnotesize \textsf{T}}
\put(6,37){\small $\Theta$}
\put(12,25){\footnotesize $v$}
\put(80,50){\circle*{3}}
\put(80,10){\circle*{3}}
\put(80,30){\circle{40}}
\put(80,40){\circle{20}}
\put(80,50){\circle*{3}}
\put(77,17){\footnotesize \textsf{T}}
\put(76,37){\small $\Theta$}
\put(150,30){\circle*{3}}
\put(150,50){\circle*{3}}
\put(150,30){\circle{40}}
\put(150,30){\circle{40}}
\put(150,30){\line(0,1){20}}
\put(147,17){\footnotesize \textsf{T}}
\put(230,30){\circle*{3}}
\put(220,30){\circle{20}}
\put(240,30){\circle{20}}
\put(215,30){\oval(30,40)[b]}
\put(245,30){\oval(30,40)[b]}
\put(230,30){\oval(60,40)[t]}
\put(227,41){\footnotesize \textsf{T}}
\put(215,28){\small $\Theta_1$}
\put(235,28){\small $\Theta_2$}
\end{picture}
\end{center}
\caption{Triangles with two or three coincident vertices}\label{fig:tri_meet}
\end{figure}

It therefore remains only to consider the case where the whole of
$\Gamma$ has the structure of the first diagram in
Figure~\ref{fig:tri_meet}. Here, the sub-diagram labelled $\Theta$ has
no loops labelled by $V^\sigma$-letters,
and so $\Theta$ can have no red triangles with
coincident vertices. The  boundary word of $\Theta$ is $\sigma$-reduced, by
definition of $V_P$, and so, by the arguments above, we may assume that
all vertices of $\Theta$ have green degree at least $1$. But we have to
consider the possibility that the only green face of $\Theta$ that is incident
with the vertex labelled $v$ in $\Gamma$ is the external face of $\Theta$,
which is coloured red as a face of $\Gamma$. We shall now show that
this case does not occur. 

If $\delta(v, \Gamma) = 2$, then the face of $\Theta$ incident with $v$ is a red
triangle, and so $\Theta$ contains a loop labelled by a
$V^{\sigma}$-letter, 
contradicting our assumption of loop-minimality. 
If  $\delta(v, \Gamma) \geq 3$, there are three
consecutive red triangles incident with $v$.  Since $\Gamma$ has no loops based
at $v$, we can repeatedly apply Lemma~\ref{lem:threereds}  to $v$ to reduce
its degree, and hence reduce the coloured area, reaching a diagram
$\Gamma'$ in which  $\delta(v, \Gamma') =
2$. This yields the same contradiction to our assumption of
loop-minimality as at the beginning of this paragraph. 
\end{proof}

\begin{prop} \label{prop:red_bound}
Let $\Gamma$ be a loop-minimal coloured van Kampen diagram 
with cyclically $P$-reduced boundary word $w$ of length $n$ and let $r =
\max\{|R| \ : \ R \in \cR\}$.
Then there exists a loop-minimal coloured diagram $\Delta$ with
boundary word $w$, such that $\Area(\Delta) \leq (3+r)|F_G(\Delta)| + n - 2$.
If every vertex $v$ of $\Gamma$ satisfies $\delta_G(v, \Gamma) \geq 1$ 
then $\Delta$ can be taken to be $\Gamma$.
\end{prop}

\begin{proof}
By Theorem~\ref{thm:one_green}, there exists a loop-minimal coloured diagram
$\Delta$ with boundary word $w$ and the same green faces as $\Gamma$, in
which every vertex $v$ satisfies $\delta_G(v, \Delta) - 1 \geq  0$. 
Hence for $\Delta$,  Proposition~\ref{prop:graph} gives the inequality
$$|F_R(\Delta)| \leq \sum_{e \in E(\Delta)} \delta_G(e,\Delta) - 2 + 2|F_G(\Delta)|.$$
Each contribution to $\delta_G(e,\Delta)$ comes either from the boundary of the
external face (of length $n$), or from an internal green face
(of which there are $|F_G(\Delta)| = |F_G(\Gamma)|$, each of boundary length at most $r$), so we
deduce that $$|F_R(\Delta)| \le n + r|F_G(\Delta)| - 2 + 2|F_G(\Delta)|,$$
and 
$\Area(\Delta) = |F_R(\Delta)| + |F_G(\Delta)|  \le (3+r)|F_G(\Delta)| + n -2,$ as claimed.
\end{proof}

\section{Interleaving the green relators}
\label{sec:interleave}

In
Subsection~\ref{subsec:cyclic} we shall generalise 
interleaving (Definition~\ref{def:interleave}) to cyclic
interleaving, and show that this gives an equivalence relation on
cyclically $P$-reduced words. Then in
Subsection~\ref{subsec:valid} we shall prove our
main result in this section, Proposition~\ref{prop:two_greens}. This shows
that if we replace $\cR$ by the  (finite) set $\cI(\cR)$ of all cyclic interleaves of 
elements of $\cR$, then a cyclically $P$-reduced word $w$ is
equal to $1$ in $G = \langle X^\sigma \ | \ V_P \ | \ \cR \rangle$ if and
only if some cyclic interleave of $w$  is the boundary of a coloured
diagram over $\langle X^\sigma \ | \ V_P \ | \  \cI(\cR) \rangle$ in which
each vertex is incident with at least two green faces. Finally, in
Subsection~\ref{subsec:red_blobs} we shall study the regions of coloured
diagrams that are composed entirely of red triangles.

We remind the reader that all newly defined terms and notation are
listed in the Appendix. 

\subsection{Cyclic interleaving}\label{subsec:cyclic}

Recall Definition~\ref{def:interleave} of interleaving, and that by
Theorem~\ref{thm:up} this
yields an equivalence relation on $P$-reduced words, with one
equivalence class for each element of $U(P)$. 

\begin{lemma} \label{lem:approxcr}
Let $v= x_1x_2\cdots x_n \in X^\ast$ be cyclically $P$-reduced.
If $v \approx w$, then $w$ is cyclically $P$-reduced.
\end{lemma}
\begin{proof}
If $n \leq 1$ then the result is trivial, so assume that $n > 1$, and
let $s_0 = 1, s_1, \ldots, s_n = 1$ be the elements of $P$ from
Definition~\ref{def:interleave}. 

By the transitivity of $\approx$, it suffices to consider a single rewrite, so
assume that 
$s_i$ is nontrivial for a single $i$. The result is immediate unless
$i = 1$ or
$i = n-1$, so assume without loss of generality that
$s_1 \neq 1$. 

Suppose first that $n \ge 3$. Then $x_n x_1 x_2 \cdots x_{n-1}$ is
$P$-reduced by assumption, and $x_n x_1 x_2 \cdots x_{n-1}
\approx x_n [x_1 s_1] [s_1^{\sigma}x_2] \cdots x_{n-1}$. By Theorem \ref{thm:up}(i), the word 
$x_n [x_1s_1] [s_1^{\sigma}x_2] \cdots x_{n-1}$ is 
$P$-reduced, so $(x_n,[x_1s_1]) \not\in D(P)$, which proves the result.

Otherwise, $n=2$ and $w = [x_1s_1] [s_1^{\sigma}x_2]$.  We
need to show that $([s_1^{\sigma}x_2],[x_1s_1]) \not\in D(P)$, so assume the contrary.
We apply Axiom (P5) with $$(x,y,z,t) = (s_1,s_1^{\sigma}x_2,x_1s_1,s_1^{\sigma})$$ and
conclude that at least one of $(x_2,x_1s_1),\,(s_1^{\sigma}x_2,x_1) \in D(P)$.
Suppose that $(x_2,x_1s_1) \in D(P)$  (the other case is similar).
Applying (P5) again with $(x,y,z,t) = (x_2,x_1s_1,s_1^{\sigma},x_2)$
gives $(x_1, x_2) \in D(P)$ or $(x_2, x_1) \in D(P)$,
contradicting $v$ being cyclically $P$-reduced.
\end{proof}

We now define a coarser relation than  $\approx$, by allowing the elements $s_0$ and
$s_n$ from Definition~\ref{def:interleave} to be nontrivial but equal.

\begin{defn}\label{def:cyc_interleave}
Let $v = x_1 \cdots x_n \in X^\ast$ be cyclically $P$-reduced, and
let $w = y_1 \cdots y_m \in X^\ast$. Then we write
$v \approx^c w$ if $m=n$ and either $n \leq 1$ and $v =w$, or 
$n > 1$ and there exist $s_0,s_1, \ldots, s_{n-1}, s_n = s_0\in P$ such
that $(s_{i-1}^{\sigma}, x_i),  (x_i, s_i), ([s_{i-1}^{\sigma}x_i], s_i) \in D(P)$ for
$1 \le i \le n$ and $y_i = [s_{i-1}^{\sigma}x_is_i]$. We say that $w$ is a
\emph{cyclic interleave} of $v$. 
\end{defn}

\begin{example}\label{ex:cyc_interleave}
Consider the pregroup $P$ from Example~\ref{ex:interleave}, and let $v
= a_1b_1a_3b_5 \in U(P)$. Notice that since $(b_5, a_1) \not\in D(P)$ the word $v$ is
cyclically $P$-reduced. 

Then one cyclic interleave $w$ of $v$ can be made by setting $s_0 =
s_1 = s_4 = i_2$ and $s_2 = s_3 = i_4$. With this choice
of interleaving elements we find that $w = a_1b_3a_3b_3$. 
\end{example}

\begin{thm} \label{thm:approxcequiv}
Let $v= x_1 \cdots x_n \in X^\ast$ be  cyclically $P$-reduced. If
$w =  y_1 \cdots y_n \approx^c v$, then 
$w$ is cyclically $P$-reduced. Furthermore, $\approx^c$ is
an equivalence relation on the set of all cyclically $P$-reduced words.
\end{thm}
\begin{proof}
Since $\approx^c$ is the identity relation on words of length at most $1$,
we may assume without loss of generality that $n > 1$.

We can move from $v$ to $w$ 
by a sequence of single rewrites.
By Lemma \ref{lem:approxcr}, a single rewrite with $i \ne n$ replaces $v$
by another cyclically $P$-reduced word and, by applying the lemma to a cyclic
permutation of $v$, we see that the same applies when $i=n$. So $w$ is
cyclically $P$-reduced.

To show that $\approx^c$ is an equivalence relation, 
it is sufficient to prove that if $w \approx^c u$, where $u = z_1
\ldots z_n$ is the result
of applying a single rewrite to $w$, then $v \approx^c u$.

Suppose first that this rewrite consists of replacing
$(y_n,y_1)$ by $([y_nt],[t^{\sigma}y_1])$ for some $t \in P$, so that
$u = [t^\sigma y_1] y_2 \ldots y_{n-1} [y_nt]$. Let 
$s_0,s_1, \ldots, s_{n-1}, s_n = s_0\in P$ be as in
Definition~\ref{def:cyc_interleave}, cyclically interleaving $v$ to $w$.
Then  $y_n=[s_{n-1}^{\sigma}x_ns_n]$ and $y_1 = [s_n^{\sigma}x_1s_1]$.
Since $w$ is cyclically $P$-reduced, $y_ny_1$ is $P$-reduced, and hence
$(s_n,t) \in D(P)$ by \cite[3.A.2.6]{Stallings}.  So, putting $s_i'=s_i$ for
$1 \le i < n$ and $s_0'=s_n'=[s_nt]$, we have $z_i = s_{i-1}'^{\sigma}x_is_i'$
for $0 \le i \le n$ and $v \approx^c u$ as claimed.

The argument in the case when the single rewrite from $w$ to $u$
consists of the replacement of $(y_i,y_{i+1})$ with $1 \le i < n$ is similar.
\end{proof}

\begin{defn}\label{def:I(w)}
Let $w \in X^{\ast}$ be cyclically $P$-reduced. We denote by $\cI(w)$,
called the
\emph{cyclic interleave class} of $w$, the set
$$\cI(w) = \{ v \in X^\ast \ : \ v \approx^c w\}.$$ 
\end{defn}

We record the following easy lemma for later use. 

\begin{lemma} \label{lem:rewritetwice}
Let $v  = x_1 \ldots x_n \in X^\ast$ be cyclically $P$-reduced with $n>1$,
and suppose that $w \in \cI(v)$.
Then we can obtain $w$ from $v$ by applying a sequence of
at most $n$ single rewrites.
\end{lemma}

\begin{proof}
By definition of the cyclic interleave class of $v$,
there are elements $y_1,\ldots,y_n$ and $s_0,s_1,\ldots,s_n=s_0$ of $X$ with
$w = y_1 \ldots y_n$  and $y_i=[s_{i-1}^{\sigma}x_is_i]$ for $1 \le i \le n$. 
By Theorem~\ref{thm:approxcequiv}, the relation $\approx^c$ is
transitive, so 
we can introduce the $s_i$ individually (and in any order)
as single rewrites, where each such rewrite consists of replacing the
current word $y_1'y_2'\cdots y_n'$ by either
$y_1'\cdots y_i''y_{i+1}'' \cdots y_n'$ with
$y_i'' = y_i's_i$ and $y_{i+1}'' = s_i^\sigma y_{i+1}'$ for some $i$
with $1 \le i < n$, or by
$y_1''y_2' \cdots y_{n-1}'y_n''$ with $y_1'' = s_0^\sigma y_1$ and
$y_n'' = y_n's_0$. 
\end{proof}

\begin{defn}\label{def:ip}
We write $\cI(\cR)$ for $\cup_{R \in \cR}\cI(R)$, and let $\ip$ be the
pregroup presentation
$$\langle X^\sigma \mid  V_P \mid \cI(\cR) \rangle.$$ 
Note in particular that, since $X$ and $\cR$ are finite, so is $\cI(\cR)$.
Faces of diagrams over $\ip$ are coloured red and green, just as for coloured
diagrams over $\cP$. 
\end{defn}

\begin{thm}\label{thm:ilpres}
The normal closure of $\langle  V_P \cup \cR \rangle$ in $F(X^\sigma)$
is equal to the normal closure of $\langle V_P \cup
\cI(\cR) \rangle$ in $F(X^\sigma)$. Hence, $G$ is defined by $\cP$ if and
only if $G$ is defined by $\cI(\cP)$. 
\end{thm}

\begin{proof}
One containment is clear, since $\cR \subset \cI(\cR)$.
For the other,  by Lemma~\ref{lem:rewritetwice} each element of $\cI(\cR)$ is made by
applying a finite number of single rewrites to each $R =
x_1 \ldots x_n \in \cR$, and so is conjugate to a word that
is equal in $U(P)$ to an element of $\cR$. The final statement follows
from Corollary \ref{cor:pgpres}. 
\end{proof}

As a result of the above theorem, we shall move between working with a
presentation $\cP$ and a presentation $\cI(\cP)$ without further
comment.

\subsection{\Valid diagrams}\label{subsec:valid}

Recall Definitions~\ref{def:van_kampen} and \ref{def:coloured_vkd} 
and in particular our
conventions on counting incidence, and on the boundary of diagrams.
The following definition will be used throughout the rest of the
paper, as it is key to our methods of proving hyperbolicity. 

\begin{defn}\label{def:valid}
A coloured  diagram $\Gamma$ in which each vertex $v$ satisfies
$\delta_G(v, \Gamma) \geq 2$
is \emph{\valid}.
\end{defn}

The following is our main result in this section. 

\begin{prop}\label{prop:two_greens}
Let $\Gamma$ be a loop-minimal coloured diagram over $\ip$
with cyclically $\sigma$-reduced boundary word $w$.

\begin{enumerate}
\item[(i)]
Assume first that $w$ is cyclically $P$-reduced and that, if $|w| = 1$, then
the unique boundary face of $\Gamma$ is green. 
Then some $w' \in \cI(w) $ is the boundary word of a \valid loop-minimal
coloured  diagram $\Delta$ over $\ip$ with
$\CArea(\Delta) \le \CArea(\Gamma)$.

\item[(ii)]
If, instead, $|w| > 1$ and $w$ is not cyclically
$P$-reduced, then $w$ is the boundary word of a loop-minimal
coloured  diagram $\Delta$ over $\ip$ with
$\CArea(\Delta) \le \CArea(\Gamma)$,
in which all non-boundary vertices $v$ satisfy $\delta_G(v, \Delta) \geq 2$. 
\end{enumerate}
In both cases, if $\Gamma$ is not \valid, then $\CArea(\Delta) < \CArea(\Gamma)$.
\end{prop}

\begin{proof}
The assumptions of loop-minimality,  and that if $|w|=1$ then
the unique boundary face of $\Gamma$ is green,  together  imply
that $\Gamma$ does not have a unique boundary face that is red.
The loop-minimality then implies that the vertices
of each red triangle of $\Gamma$ are distinct, and that if $|w| = 1$,
with $v_0$  the unique
boundary vertex, then $\delta_G(v_0, \Gamma) \geq 3$.

By
Theorem~\ref{thm:one_green} we may assume that every vertex  $v$ of
$\Gamma$ satisfies $\delta_G(v, \Gamma) \geq 1$. Assume that
$\delta_G(v, \Gamma) =
1$, 
and that  if $v \in \partial(\Gamma)$ then $w$ is cyclically
$P$-reduced. Let $ab$ be the length two subword of the boundary label of the
unique green face $f$ that is incident with $v$, so that $v$ is
between $a$ and $b$.  
We shall construct a coloured diagram $\Delta$ with boundary word $w'
\in \cI(w)$, in which $v$ no longer exists, and
such that $\CArea(\Delta) < \CArea(\Gamma)$. 
The loop-minimality of $\Gamma$  implies that $\Delta$ is
loop-minimal. 

If $v$ is incident with a unique red triangle $T$, then $\delta(v, \Gamma) = 2$, 
which implies that $b^{\sigma}a^{\sigma}$ is a subword of the boundary
label of $T$, and so $(b^\sigma, a^{\sigma}) \in D(P)$.
Elements of $\cR$ (and by Theorem~\ref{thm:approxcequiv} also of
$\cI(\cR)$) 
are cyclically
$P$-reduced, so $v \in \partial(\Gamma)$. But in this case $w$ was
assumed to be cyclically
$P$-reduced, a contradiction. 

If $v$ is incident with exactly two red triangles, then they must be
distinct and share an edge.   
Let the third edge incident with $v$ be labelled $c$, so that the
triangles have labels with subwords $c^{\sigma}a^{\sigma}$ and
$b^{\sigma}c$. 
Then $(a, c)$ and $(c^{\sigma}, b)$ are also in $D(P)$, and $([ac],
[c^\sigma b]) \notin D(P)$ by Theorem~\ref{thm:approxcequiv}.  A
single rewrite can therefore be applied
to the label of $f$, 
replacing $ab$ with $[ac][c^{\sigma}b]$.  This has the effect
of replacing $f$ by a face labelled with a cyclic
interleave of the label of $f$, 
removing the vertex $v$ and its 
two incident red triangles from $\Gamma$, and leaving the number of green faces
unchanged. 

Assume finally that $\delta_R(v, \Gamma) \geq 3$.
Since $v$ is not a boundary vertex of a
diagram with boundary length $1$,  the loop-minimality of $\Gamma$ 
means that there are no loops based at $v$.
Hence we can repeatedly
apply Lemma~\ref{lem:threereds} to make a diagram $\Gamma'$ in which
$\delta_R(v,\Gamma') = 2$,
and then delete $v$ and the final two red triangles incident
with it, 
as in the previous paragraph. 
\end{proof}

In what follows, we shall often assume that we work with \valid van
Kampen diagrams. But, since we have not been able to eliminate the existence
of non-\valid loop-minimal diagrams with boundary length $1$ and the unique
boundary face a red triangle, we need to take that possibility into
account when developing algorithms, as we shall do in
Theorem \ref{thm:NoVletter1}.

\subsection{Red blobs}\label{subsec:red_blobs}

We now turn our attention to the regions of coloured
diagrams that are comprised entirely of red triangles.

\begin{defn}\label{def:blob}
A \emph{red blob} in a coloured diagram $\Gamma$ is a nonempty
subset $B$ of the set of closed red triangles of $\Gamma$,
with the property that any nonempty proper subset $C$ of $B$ has at
least one edge in common with $B \setminus C$. Equivalently,
the induced subgraph $\overline{B}$ of the dual graph $\overline{\Gamma}$ 
of $\Gamma$ on those vertices that correspond to the triangles
in $B$ is connected.

A red blob is \emph{simply connected} if its interior is
homeomorphic to a disc: its boundary may pass more than once through a
vertex. 
\end{defn}

Notice that a simply-connected red blob corresponds to a van Kampen
diagram over $U(P)$, and so in particular has boundary length greater
than one, by Theorem~\ref{thm:up}. (However, not all van Kampen
diagrams over $U(P)$ correspond to red blobs, since the interior of
such diagrams may be disconnected). 

\begin{lemma}\label{lem:blob_bound}
Let $B$ be a red blob in a coloured diagram $\Gamma$, with boundary length $l$ and area $t$. 
Then $l \leq t +2$,
and $l \leq t$ if $B$ is not simply connected. Furthermore,
if $B$ is
simply connected, and every vertex of $B$ lies on $\partial(B)$ (which holds in
particular when $\Gamma$ is \valid),  then $l = t+2$. 
\end{lemma}

\begin{proof}
Let $\overline{B}$ be the induced subgraph of $\overline{\Gamma}$
that corresponds to the triangles in $B$, as in Definition~\ref{def:blob}. A vertex
of degree 1 in $\overline{B}$ corresponds to a triangle in $B$
that has two edges on
$\partial(B)$. Deleting this triangle reduces both the number of
boundary edges and the number of triangles by $1$, and $\overline{B}$
remains connected.
We repeatedly remove degree $1$ vertices from
$\overline{B}$ until none remain.
At that stage, either a single vertex remains, in which case 
$l = t+2$, or the remaining
vertices all have degree at least two.
Such vertices correspond to triangles with at most one edge on
$\partial(B)$, so the number of triangles is at least the boundary
length, and $l \leq t$.

If $B$ is not simply connected, then $\overline{B}$ contains
a circuit, and so the second of the above two situations arises,
and $l \leq t$.

Now assume that all vertices of $B$ lie on $\partial(B)$, and $B$ is
simply connected. We shall show that $\overline{B}$ is a tree,
from which the final claim follows.
By way of contradiction, let $C$
be a circuit in $\overline{B}$.
It is a standard result from graph theory (see, for example,
\cite[Corollary 4.15]{Wilson}) that the corresponding edges in $B$
form a cutset in $B$. Hence the circuit must enclose at least one
vertex of $B$. We have assumed that all vertices of $B$ lie on 
$\partial(B)$, so this contradicts the fact that $B$ is simply connected.
\end{proof}

\begin{prop}\label{prop:blob_reduction}
Let $\Gamma$ be a loop-minimal coloured diagram with
cyclically $\sigma$-reduced boundary word $w$. 
Then there exists a loop-minimal coloured  diagram $\Delta$
with boundary word $w$, with $\CArea(\Delta) \le \CArea(\Gamma)$,
and in which the (cyclic) boundary word of each simply connected red blob has no proper
subword equal to $1$ in $U(P)$.
If $\Gamma$ is \valid then $\Delta$ is \valid.
Furthermore, if $\Gamma$ does not have the required property already, then
$\CArea(\Delta) < \CArea(\Gamma)$.
\end{prop}

\begin{proof}
Let $B$ be a simply connected red blob in $\Gamma$ with boundary word
$w = x_1 \ldots x_n$.
First note that by Theorem~\ref{thm:one_green}, we may assume that all
vertices of $\Gamma$ have green degree at least one. 
Hence all vertices of $B$ lie on $\partial(B)$,
and $\Area(B) = n - 2$ by 
Lemma~\ref{lem:blob_bound}.

We first consider $\sigma$-reduction, so assume (without loss of
generality) that $w = x_1x_1^{\sigma}w_1$. Notice that $w =_{U(P)} 1$, so
$w_1=_{U(P)} 1$. Hence we can identify the vertices at the beginning
of the edge labelled $x_1$ and the end of the edge labelled $x_1^{\sigma}$,
and  replace $B$ by a coloured sub-diagram $\Theta$ consisting of a
red blob $B_1$ with boundary word $w_1$, 
with a single edge added to the boundary.
The blob $B_1$ is simply connected with all
vertices on the boundary, and $|\partial(B_1)| = |w_1| = n-2$, 
so $\Area(B_1) = n-4$ by Lemma~\ref{lem:blob_bound}.
The 
diagram $\Delta$ in which $B$ has been replaced by $\Theta$ satisfies
$\CArea(\Delta) < \CArea(\Gamma)$, so $\Delta$ is loop-minimal.
Replacing $B$ by $\Theta$ cannot decrease the green degrees of vertices, 
so if $\Gamma$ is \valid then $\Delta$ is \valid. 

Now we consider subwords of $w$ of length greater than $2$.
Assume that $w$ has a factorisation 
$w = w_1 w_2$ such that $w_1$ and $w_2$ have lengths $b_1, b_2 \geq 3$, and
$w_1 =_{U(P)} 1 =_{U(P)} w_2$. 
We can produce a new diagram $\Delta$ in which the two vertices
where $w_1$ and $w_2$ start and end are identified, and $B$ has been
replaced by two red blobs $B_1$ and $B_2$ with boundary words $w_1$ and
$w_2$, 
of area $b_1 - 2$ and
$b_2 - 2$, respectively. From $b_1 + b_2 = n$, we see that
$\Area(B_1) + \Area(B_2) = n-4$, so $\CArea(\Delta) <
\CArea(\Gamma)$,  and hence $\Delta$ is loop-minimal. 
As  before, if $\Gamma$ was \valid then $\Delta$ is still \valid.
\end{proof}

\begin{defn}\label{def:intermult}
We say that $a \in X$ \emph{intermults with} $b \in X$ if
$b \neq a^{\sigma}$ and either $(a, b) \in D(P)$ or there exists $x \in X$
such that $(a, x), (x^{\sigma}, b) \in D(P)$. We also say that $(a, b)$ is an
\emph{intermult pair}.
\end{defn}

\begin{example}\label{ex:intermult}
If $P$ is the natural pregroup 
for a free product (with no amalgamation), constructed as in
Example~\ref{ex:free_prod},  then the intermult pairs are precisely the
non-inverse pairs of non-identity elements contained within a free
factor. 

If $P$ is instead the natural pregroup for a free product with non-trivial
amalgamation, then the intermult pairs are all non-inverse pairs of
non-identity elements. 
\end{example}

In the following lemma, we do not assume that the red blob is simply
connected: it may therefore have more than one boundary word. The
following lemma will be used in the algorithmic part of this paper, 
 to reduce the number of possible boundary
words of red blobs. 

\begin{lemma}\label{lem:intermult}
If $g,a \in X$ and $ga$ is a subword of a  boundary word of a red
blob $B$ with $\sigma$-reduced boundary words, then $g$ intermults with $a$.
\end{lemma}

\begin{proof}
First notice that the assumption that $B$ has
$\sigma$-reduced boundary words shows that $g \neq a^{\sigma}$. 

Let $v$ be the vertex between $g$ and $a$ on the boundary of $B$. 
Reading clockwise around $v$ from $a$ as far as $g$, let the labels of
the outgoing edges be $a = a_1$, $a_2, \ldots, a_k, g^{\sigma} = a_{k+1}$
(the outgoing label is $g^{\sigma}$ not $g$). Notice that the edges
labelled $a_2, \ldots, a_k$ are all interior to $B$, since
$ga$ is a subword of a boundary word. Notice also that
\begin{equation}\label{eq:intermult}
(a_{i+1}^\sigma, a_i) \in D(P) \quad  \mbox{ for } 1 \le i \le k.
\end{equation}
We shall show
by induction that \eqref{eq:intermult} implies that $a_{k+1} = g$ 
intermults with $a_1 = a$. 

If $k = 1$ then $(a_{2}^\sigma, a_1) = (g, a) \in D(P)$. If $k = 2$ then $(a_2^{\sigma}, a), (g, a_2)
\in D(P)$, so $(g, a)$ is an intermult pair. Assume that $k \ge 3$, and  for $1
\le i \le 3$ let $b_i = [a_{i+1}^\sigma a_i]^\sigma$, so that the
boundary label of each triangle is $a_{i+1}^\sigma a_i b_i$.  Then applying Axiom
(P5) to $(b_1, a_2^\sigma), (a_2^\sigma, a_3), (a_3, b_3) \in D(P)$ shows that at least
one of  $([b_1 a_2^\sigma], a_3), (a_2^\sigma, a_3 b_3) \in
D(P)$. Since $b_1 a_2^{\sigma} = a_1^\sigma$ and $a_3 b_3 = a_4$, at least one of $(a_3^\sigma, a_1), (a_4^\sigma, a_2) \in
D(P)$, so the result follows by induction.
\end{proof}

\section{Curvature distribution schemes}\label{sec:curv_dist}

In this section, we introduce the concept of curvature
  distribution schemes, and prove that they can be used to show
that groups given by a pregroup presentation satisfy an explicit
linear isoperimetric inequality, and hence are hyperbolic.

\begin{defn}\label{def:curv_dist}
Let $\Gamma$ be a coloured van Kampen diagram 
with vertex set $V(\Gamma)$, edge set $ E(\Gamma)$, set of red
triangles $F_R(\Gamma)$ and set of internal
green  faces $F_G(\Gamma)$. Let $F(\Gamma) = F_R(\Gamma) \cup F_G(\Gamma)$.
A \emph{curvature distribution} is a function
$\rho_{\Gamma} : V(\Gamma) \cup E(\Gamma) \cup F(\Gamma) \to \R$ such that 
\[ \sum_{x \in V(\Gamma) \cup E(\Gamma) \cup F(\Gamma)} \rho_{\Gamma}(x) = 1. \] 
\end{defn}

\begin{defn}\label{def:curv_dist_scheme}
Let $\cK$ be a set of coloured diagrams over $\cI(\cP)$. 
A \emph{curvature distribution scheme on $\cK$} is a map
$\Psi : \cK \to \{\rho_\Gamma \ : \ \Gamma \in \cK\}$, 
that  associates a
curvature distribution to every diagram in $\cK$. 
\end{defn}

\begin{example}\label{ex:triv}
For any coloured  diagram 
$\Gamma$ we can define a curvature distribution by setting
$\rho_{\Gamma}(v) := +1$ for each vertex $v$, setting $\rho_{\Gamma}(e) := -1$ for each
edge $e$, and setting $\rho_{\Gamma}(f) := +1$ for each internal face $f$.
Euler's formula
ensures that the total sum of all curvature values is $+1$. Since this defines
a curvature distribution for every diagram, it gives rise to a curvature
distribution scheme on $\cK$, where $\cK$ is any
set of coloured diagrams over $\cI(\cP)$, for any pregroup
presentation $\cP$. 
\end{example}

\begin{defn}\label{def:dual}
Let $\Gamma$ be a plane graph, and let $\overline{\Gamma}$ be its
dual. Let $f_1$ and $f_2$ be faces of $\Gamma$, corresponding to
vertices $v_1$ and $v_2$ of $\overline{\Gamma}$. 
The \emph{dual distance} in $\Gamma$ from $f_1$ to $f_2$ is the
distance in $\overline{\Gamma}$ from
$v_1$ to $v_2$. 
\end{defn}

\begin{defn}\label{def:alternative_dehn}
The \emph{pregroup Dehn function} $\PD(n): \mathbb{Z}_{\ge 0} \rightarrow
\mathbb{Z}$ of a pregroup presentation 
$\cP = \langle X^{\sigma} \mid V_P  \mid \cR\rangle$ 
is defined as follows. 
For each $\sigma$-reduced word $w \in X^*$ with $w =_G 1$,
let $A(w)$ be the smallest area of a coloured diagram over $\cP$
with boundary label $w$. Then
$\PD(n) := \max\{ A(w) : w \in X^*, w=_G1, |w| \leq n \}$.
\end{defn}

$\PD(n)$ may differ from the standard Dehn function $\De(n)$ of a corresponding
group presentation, because faces of standard van Kampen diagrams labelled by
relators $x^2$, corresponding to generators $x \in X$ with $x =
x^\sigma$, 
are not counted. To bound $\De(n)$ in terms of $\PD(n)$, we need to fix a
corresponding group presentation.

\begin{defn}\label{def:grp_pres}
Let $\cP = \langle X \mid V_P \mid \cR \rangle$. Let $Y$ be a minimal subset of 
$X$ such that $X = Y \cup Y^{\sigma}$  (so $Y$ generates $\cP$ as a group).
Let subsets $V_P'$ and $\mathcal{T}$ of $F(Y)$ be constructed from
$V_P$ and $\cR$, respectively, by replacing all symbols
$x \in X \setminus Y$ by $(x^{\sigma})^{-1}$.
Then the \emph{standard group presentation corresponding to $\cP$} is 
$\cP_G := \langle Y \mid \{x^2 \ : x \in Y,  x = x^\sigma\} \cup V'_P \cup
\mathcal{T} \rangle$.
\end{defn}

\begin{example} Let $\cP = \langle x,y,z \mid y^3,z^3 \mid
(xz)^7,(xyxz)^4 \rangle$ with $x^\sigma=x$ and $y^\sigma=z$.
Then choosing $Y = \{x,y\}$ gives
$\cP_G = \langle x,y \mid x^2,y^3,y^{-3},(xy^{-1})^7, (xyxy^{-1})^4 \rangle$
(where we could of course omit the redundant relator $y^{-3}$).
\end{example}

The following bound is not at all tight, but suffices to show that if
$\PD(n)$ is linear then so is $\De(n)$. 

\begin{lemma}\label{lem:dehn_convert}
Let $\PD(n)$ and $\De(n)$ be the pregroup and standard Dehn functions of
$\cP$ and $\cP_G$,
respectively. Let $r_I$ be the maximum number of involutory generators
appearing in any $R \in V_P \cup \cR$. Then $\De(n) \leq r_I\PD(n) + n/2$. 
\end{lemma} 

\begin{proof}
To change a coloured diagram into a standard van Kampen diagram, we
first replace each edge label from $X \setminus Y$ by the inverse of the
corresponding element of $Y$. This produces a diagram that is almost a
van Kampen diagram, except that involutions $x \in P$ may appear on
both sides of an edge. But we can rectify that as follows.
For $R \in V_P \cup \cR$, let $R_I$ denote the number of involutory generators
occurring in $R$ (with multiplicity). Then, in a diagram $\Gamma$ over $F(X^\sigma)$, a
face previously labelled by $R^{\pm 1}$ needs to have $R_I$ digons
(with boundary labels of the form $x^2$)
added to its boundary to correct the edge labels. There are at
most $\PD(n)$ faces in $\Gamma$, and at most $n$ boundary edges. Each
boundary edge is incident either with an internal face, or twice with
the external face, so the result follows. 
\end{proof}

The following theorem is one of the key results in this paper. It
appears technical, but the insight behind it is
straightforward. We shall show how curvature distribution schemes give sufficient
conditions for the 
area of a van Kampen diagram $\Gamma$  over $\cP$ to be bounded by a
multiple of the boundary length $n$: this is our generalisation of
small cancellation.

 We first 
ensure that all of the positive curvature is associated with the 
green faces of $\Gamma$, and that there is a fixed upper bound $m$ on
the curvature of any face.  Then we ensure that the green faces at dual
distance greater than $d$ from the boundary in fact have curvature bounded above by
$-\varepsilon$, for some fixed $d$ and fixed $\varepsilon > 0$. 
Notice that the number of green faces at dual distance at most $d$ from
$\partial(\Gamma)$ is
bounded by a function of $n$ and the length of the
longest relator, and this function is linear in $n$. Hence, since the
total curvature must sum to $1$, the total number of green faces is
bounded by a linear 
multiple of $n$.
If 
the number of red faces is
bounded linearly in terms of the number of green faces (for example,
if $\Gamma$ satisfies the conditions of
Proposition~\ref{prop:red_bound}), then the desired
proof of linear area follows. 

\begin{thm}
\label{thm:hypercurvature}
Let $\cP = \left< X^\sigma \mid V_P \mid \cR\right>$ be a pregroup presentation
for a group $G$, and let $r$ be the maximum length of a relator in $\cR$.
Let $\cK$ be a set of coloured van Kampen diagrams over $\cI(\cP)$,
and let $\Psi : \cK \to  \{\rho_\Gamma \ : \ \Gamma \in \cK\}$ be a curvature
distribution scheme.  

Assume that there exist constants $\varepsilon
\in \mathbb{R}_{>0}$, $\lambda, \mu, m \in \R_{\geq 0}$ and 
$d \in \mathbb{Z}_{>0}$ such that the following
conditions hold, for all $\Gamma \in \cK$:
\begin{enumerate}
\item[(a)] $\rho_{\Gamma}(x) \le 0$ for all $x \in V(\Gamma) \cup E(\Gamma)
   \cup F_R(\Gamma)$,
\item[(b)] $\rho_{\Gamma}(f) \le -\varepsilon$ for all faces $f \in
F_G(\Gamma)$ that are dual distance at least $d+1$ from the external face,
\item[(c)] if $\Area(\Gamma) > 1$, then
$\rho_{\Gamma}(f)\le m$ for all faces $f \in F_G(\Gamma)$ that
are dual distance at most $d$ from the external face,
\item[(d)] $\Area(\Gamma) \leq \lambda |F_G(\Gamma)| + \mu|\partial(\Gamma)|$.
\end{enumerate}
Then each $\Gamma \in \cK$ with boundary length $n$ and area greater
than 1 satisfies
\begin{equation}\label{eqn:isoper}
\Area(\Gamma) \leq f(n) = \lambda\left( n \frac{(r-1)^d - 1}{r-2}
  \left(1+ \frac{m}{\varepsilon}\right) - \frac{1}{\varepsilon} \right)
+ \mu n.
\end{equation}

Assume now that, in addition, the following holds:
\begin{enumerate}
\item[(e)] if $w \in X^\ast$ is cyclically $P$-reduced, and satisfies
  $w =_G 1$, then there exists 
a diagram $\Gamma \in \cK$ with
boundary word some $w' \in \cI(w)$. 
\end{enumerate}
Then the group $G$ is hyperbolic. In particular, 
if $\cI(w) = \{w\}$ for all $w \in X^\ast$, then 
the pregroup Dehn function of $\cP$ 
is bounded above by $\mathrm{max}\{f(n), 1\}$. 
\end{thm}

\begin{proof}
We show first that Equation~\eqref{eqn:isoper} holds.
Let $\Gamma \in \cK$ have boundary length $n$, and let
 $F := F(\Gamma)$ and $F_G := F_G(\Gamma)$. If $\Area(\Gamma) = 1$, 
then Equation~\eqref{eqn:isoper} does not apply, so
 assume that $\Area(\Gamma) > 1$.

Let $I$ be the set of green faces that are dual distance at least $d+1$
from the external face (the set $I$ may be empty). From Condition (a)
we deduce that $\sum_{f \in F_G}\rho_{\Gamma}(f)  = \sum_{f \in I} \rho_{\Gamma}(f) +
\sum_{f \in F_G \setminus I} \rho_{\Gamma}(f) \ge 1$.  Combinatorial
considerations show that 
$$|F_G \setminus I| \le n + n(r-1) + n(r-1)^2 +
\cdots + n(r-1)^{d-1} = n\frac{(r-1)^d - 1}{r-2}.$$  Condition (b)
yields
$\sum_{f \in I}\rho_{\Gamma}(f) \le -\varepsilon |I|$ and then applying Condition (c)
(since $\Area(\Gamma) > 1$), we deduce that 
$$\varepsilon|I| \le - \sum_{f \in I} \rho_{\Gamma}(f) \leq  \sum_{f \in F_G \setminus I} \rho_{\Gamma}(f) - 1 \le
mn\frac{(r-1)^d - 1}{r-2}-1.$$
From this we get $|I| \leq \frac{1}{\varepsilon}\left(mn\frac{(r-1)^d
    - 1}{r-2}-1\right)$,
 and so by Condition (d) we see that 
$$\begin{array}{rl}
 |F|  
& \leq  \lambda(|I| + |F_G\setminus I|) + \mu n \\
& \leq \lambda\left(\frac{1}{\varepsilon} \left(mn\frac{(r-1)^d -
  1}{r-2}  - 1\right) + n \frac{(r-1)^d-1}{r-2}\right) + \mu n 
\end{array}$$
and Equation (\ref{eqn:isoper}) follows. 

Now assume that Condition (e) also holds. Each single rewrite of the
boundary word of a diagram $\Gamma$ adds two red triangles to the diagram,
as in the proof of Proposition~\ref{prop:two_greens}. It therefore
follows from Lemma~\ref{lem:rewritetwice} that, if there is a diagram
$\Gamma$ with boundary label a word $w$ of length $n$ then, for any
$w' \in \cI(w)$, there is a coloured
diagram of area at most $\Area(\Gamma) +2n$ with boundary
label $w'$. So there is
a linear upper bound on the pregroup Dehn function, and hence by
Lemma~\ref{lem:dehn_convert} the Dehn function, of $G$. The remaining assertions now follow.
\end{proof}

In the curvature scheme that we shall study in the remainder of this
paper, we shall generally set $d = 1$, and we shall prove that $m=1/2$.

\begin{cor}\label{cor:easy_curve}
Let $\cP = \left< X^\sigma \mid V_P \mid \cR\right>$ be a pregroup presentation
for a group $G$, such that each $x \in X$ is nontrivial in $G$.
Let $r$ be the maximum length of a relator in
$\cR$, let $\cK$ contain 
all diagrams over $\cP$ of minimal coloured area for each cyclically $P$-reduced
word $w$ that is trivial in $G$, and 
let $\Psi$ be a curvature
distribution scheme on $\cK$.

If there exists $\varepsilon > 0$ such that Conditions (a), (b) and (c)
of Theorem~\ref{thm:hypercurvature} hold, with $m = 1/2$ and $d = 1$, 
then $G$ is hyperbolic, and the pregroup Dehn function of 
$\cP$  is bounded above
by 
$$n \left( 4+r + \frac{3+r}{2 \varepsilon}\right) -
\frac{3+r}{\varepsilon}.$$
\end{cor}

\begin{proof}
Notice that since each $x \in X$ is nontrivial in
$G$,  all diagrams are loop-minimal. By
Theorem~\ref{thm:one_green}, each diagram $\Gamma$ of minimal coloured area for
its boundary word satisfies $\delta_G(v, \Gamma) \geq 1$ for each vertex
$v$. Hence by Proposition~\ref{prop:red_bound} each diagram of minimal
coloured area for its boundary word satisfies Condition (d) of
Theorem~\ref{thm:hypercurvature},  with $\lambda = 3+r$ and
$\mu = 1$. Substituting for $\lambda$, $\mu$, $m$ and $d$ into
\eqref{eqn:isoper} yields
$$
f(n)  = (3+r)\left(n (1 + \frac{1}{2\varepsilon}) - \frac{1}{\varepsilon} \right)
+ n.
$$
This gives an upper bound on the area of all diagrams whose area is
greater than $1$. The assumption
that each $x \in X$ is nontrivial in $G$ implies that any
diagram of area $1$ has $n \geq 2$, and one may check that if $n \geq
2$ then $f(n) \geq 1$. Hence $f(n)$ bounds the pregroup Dehn function of $G$. 
\end{proof}

 In general, it is not practical to let $\cK$
consist \emph{only} of diagrams of minimal coloured area for their boundary
word, as membership of $\cK$ cannot easily be tested. We shall
however define in the next section a useful set of diagrams with an
easily-testable membership condition. We shall also deal with the
condition in the above corollary that each generator is nontrivial in
the group $G$.

The remainder of this paper presents and analyses one curvature
distribution scheme, chosen because it
 can be tested in time that is bounded by a low-degree polynomial
 function of  $|X|$, $|\cR|$ and $r$, and
 because it verifies that $V^\sigma$-letters are nontrivial in $G$. There are, of
course, infinitely many possible such schemes, and we leave as an open
problem the development of others that are also computationally or
theoretically useful.

\section{The $\RSym$ scheme}\label{sec:rsym}

In this section we describe a curvature
distribution scheme that treats each vertex and each edge of each
diagram symmetrically, and so is called the $\RSym$ scheme.
We first specify the set  $\calD$ of diagrams on which $\RSym$
operates. 

We remind the reader that all definitions and notation are recorded in
the Appendix.

\begin{defn}\label{def:calD}
Let $\cP$ be a pregroup presentation. Then $\calD$ denotes the
set of all coloured  diagrams $\Gamma$ over $\ip$ with the
following properties:
\begin{enumerate}
\item the boundary word of $\Gamma$ is cyclically $P$-reduced (see
  Definition~\ref{def:reduced});
\item $\Gamma$ is $\sigma$-reduced and semi-$P$-reduced  (see
  Definitions~\ref{def:sigma_reduced} and \ref{def:P_reduced});
\item $\Gamma$ is \valid  (see
  Definition~\ref{def:valid});
\item no proper subword of the boundary word of a simply connected
red blob in $\Gamma$ is
equal to $1$ in $U(P)$. 
\end{enumerate}
\end{defn}
 
\noindent
Recall Definitions~\ref{def:van_kampen} and \ref{def:coloured_vkd} for
our conventions on coloured diagrams. 

\begin{defn}\label{def:half-edge}
In a coloured diagram, we shall consider each 
edge to be composed of two coloured
\emph{half-edges}, oppositely oriented. Each half-edge is
associated with the face on that side, and inherits its colour and orientation
from that face. 
\end{defn}

The following algorithm, \ComputeRSym, takes as input a diagram
$\Gamma \in \calD$, and returns a curvature distribution $\kappa_\Gamma: \Gamma
\rightarrow \mathbb{R}$. The algorithm assigns and alters
curvature on the vertices, edges and faces of $\Gamma$ in several successive
steps: the external face has curvature $0$ throughout. 
In the algorithm description, when we say (for example) that a half-edge $e$
\emph{gives} curvature $c$ to vertex $v$, we mean that the curvature of $e$ is
reduced by $c$, and that of $v$ is increased by $c$. When we say that a vertex
$v$ \emph{distributes} its curvature equally among green faces $f_1,\ldots,f_k$,
we mean that, if $k>0$, then the current curvature $c$ of $v$ is replaced by
$0$, and $c/k$ is added to the curvature of each of $f_1,\ldots,f_k$.

\medskip

\begin{alg}\label{alg:rsym}
\noindent $\ComputeRSym(\Gamma)$:
\begin{mylist2}
\item[Step 1] Initially, each vertex, red triangle, and internal green face of
$\Gamma$ has curvature $+1$, and each half-edge has curvature $-1/2$.
\item[Step 2] Each green half-edge gives curvature $-1/2$ to its end vertex, and each
red half-edge gives curvature $-1/2$ to its triangle. 
\item[Step 3] Each vertex distributes its curvature equally amongst its incident
internal green faces, counting incidences with multiplicity. 
\item[Step 4] Each red blob $B$ such that $\partial(B) \not\subseteq \partial(\Gamma)$
 sums the curvatures of its red triangles, to
  get the \emph{blob curvature} $\beta(B)$. A red blob with
  $b:= |\partial(B) \setminus \partial(\Gamma)| > 0$ then gives curvature
  $\beta(B)/b$ across each such edge to the (internal) green face
  on the other side. 
\item[Step 5] Return the function
$\kappa_{\Gamma}:V(\Gamma) \cup E(\Gamma) \cup F(\Gamma) \to \mathbb{R}$,
where $\kappa_{\Gamma}(x)$ is the current curvature of $x$. 
\end{mylist2}
\end{alg}

\begin{defn}\label{def:rsym}
We define  \RSym to be the map from $\calD$ to $\{
\kappa_{\Gamma}(x) : x \in \calD\}$ evaluated by \ComputeRSym, so that
$\kappa_{\Gamma} = \RSym(\Gamma)$. 
We denote the
 curvature given by a vertex $v$ to a face $f$ in Step 3 of \ComputeRSym by $\chi(v, f,
 \Gamma)$, noting that if $f$ is incident more than once with $v$ then
 $f$ will receive a proper multiple of $\chi(v, f, \Gamma)$ of
 curvature from $v$. Similarly, we denote the curvature given by a
 blob $B$ to a face $f$ in Step 4 of \ComputeRSym by $\chi(B, f,
 \Gamma)$. We shall omit the $\Gamma$ from $\chi(v, f, \Gamma)$ and
 $\chi(B, f, \Gamma)$ when the meaning is clear. 
\end{defn}

Since the curvature in Step 1 of \ComputeRSym is precisely the curvature
distribution from Example~\ref{ex:triv}, and curvature is neither
created nor destroyed by the algorithm, the following is immediate:

\begin{prop}\label{prop:r_sym_dist} 
$\RSym$ is a curvature distribution scheme on $\calD$.
\end{prop}

Recall Definition~\ref{def:dual} of \emph{dual distance}.

\begin{defn}\label{def:succeed}
Let $\cP = \langle X^\sigma  \ | \ V_P \ | \ \cR \rangle$ be a
pregroup presentation, and
let $\varepsilon >0$ be a constant.
We say that $\RSym$ {\em succeeds with constant $\varepsilon$}
on a diagram $\Gamma \in \calD$
if $\kappa_{\Gamma}(f) \le -\varepsilon$ for all internal non-boundary
green faces of $\Gamma$.

More generally if, for some $d \ge 1$, we can bound $\kappa_{\Gamma}(f) \le -\varepsilon$
for all green faces of $\Gamma$ that are at dual distance at least
$d+1$ from the external face, then we say that $\RSym$ succeeds with constant $\varepsilon$
{\em at level $d$}. (So the default level is $d=1$.) 

We say that $\RSym$ \emph{succeeds on $\cP$ with constant $\varepsilon$}
(at level $d$) if this is true for every $\Gamma \in \calD$, and 
$\RSym$ \emph{succeeds on $\cP$} (at level $d$) if there exists an
$\varepsilon > 0$ for which \RSym succeeds. 
\end{defn}

Our goal in the rest of this section is to show that, if \RSym succeeds
on a pregroup presentation $\cP$, then the group presented by $\cP$ is
hyperbolic. Before we can do that, we need to study the
behaviour of \RSym, and then prove two technical lemmas which
will allow us to deal with our frequent assumption of loop-minimality in earlier sections. 

We first show, amongst other things, that for each $\Gamma \in \calD$ the
curvature distribution $\kappa_{\Gamma} = \RSym(\Gamma)$ satisfies Condition
(a) of Theorem~\ref{thm:hypercurvature} for vertices.

\begin{lemma}\label{lem:vertex_curve}
Let $v$ be a vertex of a diagram $\Gamma \in \calD$, incident with
$v_G := \delta_G(v, \Gamma)$
green faces, of which $x$ are the external face.  If $x \neq v_G$ then
let $f$ be a non-external green face incident with $f$. Then
\begin{enumerate}
\item[(i)] $\kappa_{\Gamma}(v) \leq 0$, and $\kappa_{\Gamma}(v) = 0$ if $x \neq v_G$;
\item[(ii)]  if $x \neq v_G$ then
$\chi(v, f, \Gamma) = \frac{2 -v_G}{2(v_G - x)}$;
\item[(iii)] if $v_G > 2$ and $x \neq v_G$ then $\chi(v, f, \Gamma) \leq
-1/6$. 
\end{enumerate}
\end{lemma}

\begin{proof}
The vertex $v$ begins with curvature $+1$, and $v_G \ge 2$
since $\Gamma$ is \valid.
Thus $v$ has at least two incoming green
half-edges, so receives at most $-1$ in curvature in Step 2 of 
$\ComputeRSym(\Gamma)$. Thus 
Part (i) holds, and Part (ii) is now clear. For Part (iii), notice that if
$v_G> 2$ and $x \neq v_G$ then the maximum value of
$\frac{2 -v_G}{2(v_G - x)}$ is attained when $x=0$ and $v_G = 3$. 
\end{proof}

We now show that, for each $\Gamma$ in $\calD$, the curvature
distribution
$\kappa_{\Gamma}$ satisfies
Condition (a) of Theorem~\ref{thm:hypercurvature} for red
faces. Recall our conventions in Definition~\ref{def:van_kampen} on
boundaries of faces. 

\begin{lemma}\label{lem:blob_curve}
Let $B$ be a red blob composed of $t$ triangles in a diagram
$\Gamma \in \calD$. Then $\kappa_{\Gamma}(T) \leq 0$ for each triangle
$T$ of $B$.

Let $d := |\partial(B) \cap \partial(\Gamma)|$. 
Then 
$$\chi(B, f, \Gamma) = \frac{-t}{2|\partial(B) \setminus \partial(\Gamma)|} \leq
\frac{-t}{2(t - d) +4} \leq -
\frac{1}{6}.$$
\end{lemma}

\begin{proof}
After Step 2 of $\ComputeRSym(\Gamma)$, 
the curvature of  each triangle $T$ is $-1/2$, so $\kappa_{\Gamma}(T)
\leq 0$, as required.  Hence 
$\chi(B, f, \Gamma) = -t/(2|\partial(B) \setminus \partial(\Gamma)|)$. By Lemma \ref{lem:blob_bound},
$|\partial(B)| \le t+2$, so $$\frac{-t}{2|\partial(B) \setminus
\partial(\Gamma)|}  \leq \frac{-t}{2(t - d) +4}  \leq \frac{-t}{2t+4} \leq -\frac{1}{6}.$$
\end{proof}

Recall Definition~\ref{def:van_kampen} of a consolidated edge. It
follows from the fact that all diagrams $\Gamma \in \calD$ are \valid 
that if a consolidated edge of $\Gamma$ has length greater than $1$, then both of
the incident faces are green. We now show that for all diagrams
$\Gamma \in \calD$, the curvature distribution $\kappa_{\Gamma}$ satisfies
Condition (c) of
Theorem~\ref{thm:hypercurvature}, with $m = 1/2$ and $d = 1$. The
second part of the next lemma will be used when we attack the word
problem, in Section~\ref{sec:wp}. 

\begin{lemma}\label{lem:boundary_face}
Let $\Gamma \in \calD$ have area greater than $1$, and let $f$ be a
boundary green face of $\Gamma$.
Then $\kappa_{\Gamma}(f) \leq 1/2$. 

Furthermore,
if $\kappa_{\Gamma}(f) > 0$ then the consolidated edges and vertices in 
$\overline{\partial(f) \setminus
\partial(\Gamma)}$ form a single path $p$,
and at most three of the vertices in $p$ lie on $\partial(\Gamma)$.  If
there are three such vertices, let $v$ be the middle one (as in
Figure~\ref{fig:isolbv}). 
Then $\delta_G(v, \Gamma) \geq 4$, and $f$ is incident with no red blobs at $v$. 
\end{lemma}
\begin{figure}
\setlength{\unitlength}{0.75pt}
\begin{center}
\begin{picture}(240,145)(-20,0)
\thicklines
\qbezier(10,50)(100,-50)(190,50)
\put(10,50){\line(-1,1){40}}
\put(190,50){\line(1,1){40}}
\put(100,125){\line(-1,2){20}}
\put(100,125){\line(1,2){20}}
\thinlines
\put(10,50){\line(1,2){15}}
\put(25,80){\line(1,1){35}}
\put(60,115){\line(4,1){40}}
\put(100,125){\line(4,-1){40}}
\put(140,115){\line(1,-1){35}}
\put(175,80){\line(1,-2){15}}
\put(98,115){$v$}
\put(15,50){$v_1$}
\put(172,50){$v_2$}
\put(45,90){$\beta_1$}
\put(144,90){$\beta_2$}
\put(97,10){$\beta_3$}
\put(10,125){{\Large $\Gamma$}}
\put(190,125){{\Large $\Gamma$}}
\put(96,50){{\Large $f$}}
\end{picture}
\end{center}
\caption{Isolated boundary vertex\label{fig:isolbv}}
\end{figure}
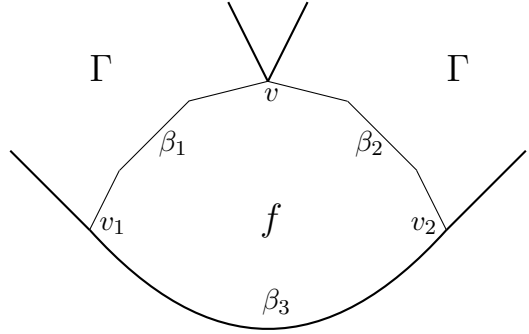
\begin{proof}
First assume that $\partial(f)\setminus \partial(\Gamma)$ contains no
edge. 
Since $\Area(\Gamma) > 1$,  
each vertex $v_0$ that lies on both $\partial(f)$ and an edge
of $\partial(\Gamma) \setminus \partial(f)$ 
satisfies $\delta_G(v_0, \Gamma) \geq 3$, and is incident at least
twice with the external face. By Lemma~\ref{lem:vertex_curve},
each such vertex $v_0$ therefore 
satisfies $\chi(v_0, f, \Gamma) \leq -1/2$, so $\kappa_{\Gamma}(f) \le 1/2$
and both claims follow.

Assume instead that $\partial(f) \setminus \partial(\Gamma)$ 
contains an edge. 
Let $\beta$ denote a maximal sequence of incident vertices and
(consolidated) edges on
$\partial(f)$, such that each edge of $\beta$ is internal
in $\Gamma$, and let $v_1$ and $v_2$ be the vertices at the beginning
and end of $\beta$. 
If $\delta_G(v_i, \Gamma) \geq 3$ then we can apply Lemma~\ref{lem:vertex_curve} with
$x \ge 1$ and $v_G \geq 3$ to deduce that $\chi(v_i, f, \Gamma) \leq
-1/4$, 
and so $\kappa_{\Gamma}(f) \leq 1/2$. 
So assume that $\delta_G(v_i, \Gamma) = 2$ for at least one $i \in \{1, 2\}$. 
Then $f$ is adjacent to a red blob $B_i$ at $v_i$,
and $|\partial(B_i) \cap \partial(\Gamma)|  \geq 1$. By 
Lemma~\ref{lem:blob_curve},  $\chi(B_i, f, \Gamma) \leq -1/4$. It follows that $\kappa_{\Gamma}(f) \leq 1/2$ unless
$\delta_G(v_1, \Gamma)=\delta_G(v_2, \Gamma)=2$ and $B_1 = B_2$. But in this case,
$|\partial(B_1)
\cap \partial(\Gamma)| \geq 2$, and so $\chi(B_1, f, \Gamma) \leq -1/2$,
by Lemma~\ref{lem:blob_curve}.

Suppose now that $\kappa_{\Gamma}(f) > 0$. Then, by the previous paragraph,
there must be exactly one such maximal sequence $\beta$ 
on 
$\partial(f)$.  Suppose that $\beta$ contains a vertex $v \neq v_1, v_2$
that lies on $\partial(\Gamma)$. 
If $\delta_G(v, \Gamma) = 2$, then there are two red blobs adjacent to $f$
at $v$ (either or both of which may be equal to $B_1$ or $B_2$),
giving additional combined curvature at most $-1/2$ to $f$, and so $\kappa_{\Gamma}(f) \le 0$,
contrary to assumption.  If $\delta_G(v, \Gamma) = 3$, then $f$ is adjacent to at least
one red blob $B_3$ at $v$, and the combined additional 
curvature that $B_3$ and $v$
give to $f$ is at most $-1/2$, giving $\kappa_{\Gamma}(f) \le 0$ again.
Hence $\delta_G(v, \Gamma) \geq 4$, and so $v$ gives at most $-1/3$ of curvature to $f$,
and there can be at most one such $v$.
\end{proof}

In particular, we have now shown that, if $\RSym$ succeeds on a diagram
$\Gamma \in \calD$, then $\kappa_{\Gamma}$ satisfies 
Conditions (a), (b) and (c) (with $m=1/2$) of
Theorem~\ref{thm:hypercurvature}. Since all diagrams in $\cal{D}$ are
green-rich, if no $V^\sigma$-letter is trivial in $G$ it follows
immediately from Proposition~\ref{prop:red_bound} that $\kappa_{\Gamma}$ satisfies
Condition (d). 

We shall show next  that the set $\calD$ satisfies
Condition  (e) of Theorem~\ref{thm:hypercurvature}, 
provided that no $V^{\sigma}$-letter is trivial in $G$. 

\begin{prop}\label{prop:d_complete}
Let $\cP$ be a pregroup presentation for a group $G$, and let $w$ be a
cyclically $P$-reduced word that is equal to $1$ in $G$. 
Assume that no $V^{\sigma}$-letter is trivial
in $G$.
Then there exists $w' \in \mathcal{I}(w)$  that is the boundary word of
a coloured  diagram $\Gamma \in \calD$. 
\end{prop}

\begin{proof}
Let $\Gamma$ be a coloured  diagram of minimal coloured
area, amongst all  coloured diagrams of words $w' \in
\mathcal{I}(w)$, and let $w'$ be the boundary word of $\Gamma$. We
shall show that $\Gamma \in \calD$.

Since $w' \approx^c w$, it follows from
Theorem~\ref{thm:approxcequiv} that $w'$ is cyclically $P$-reduced.
The assumption that $\Gamma$ has minimal coloured area implies that
$\Gamma$ is semi-$P$-reduced, by Proposition~\ref{prop:semired}.

We have assumed that no $V^{\sigma}$-letter is trivial in $G$, so $\Gamma$
is loop-minimal, and if $|\partial(\Gamma)| = 1$ then
the unique boundary face is green.
No   letter $x \in X$ such that $x = x^{\sigma}$
is trivial in $G$, so by Proposition~\ref{prop:invtriv} the diagram 
$\Gamma$ is $\sigma$-reduced. 
It now follows from 
Proposition~\ref{prop:two_greens}(i) and the minimality of
$\CArea(\Gamma)$ that $\Gamma$ is \valid.

Finally, the loop-minimality of $\Gamma$  and the minimality of
$\CArea(\Gamma)$
imply that the boundary word of each simply connected red blob has
no proper subwords equal to $1$ in $U(P)$, by
Proposition~\ref{prop:blob_reduction}.
Hence $\Gamma \in \calD$.
\end{proof}
 
It remains to deal with the assumption that no $V^\sigma$-letter is
trivial in $G$. 
To do so, we first prove that, subject to some easily-checkable
conditions on the set $\cR$ of relators, if $\RSym$ succeeds then
there are no diagrams in $\calD$ of boundary length two. 

\begin{lemma}\label{lem:norellen2}
Let $\cP  = \langle X^\sigma  \ | \ V_P \ | \ \cR \rangle$ 
be a pregroup presentation.
Suppose that no $R \in \cR$ has length $1$ or $2$ and that no two
distinct cyclic conjugates of relators $R,S \in \cI(\cR)^{\pm}$ have a
common prefix consisting of all but one letter of $R$ or $S$. 
Let $\Gamma$ be a diagram in $\calD$ with boundary length $2$.
Then $\RSym$ does not succeed on $\Gamma$.
\end{lemma}
\begin{proof}
Suppose that $\RSym$ succeeds on $\Gamma$.
Since each $R \in \cR$ has length at least three, $\Area(\Gamma) > 1$. 
Each boundary face $f$  of $\Gamma$ satisfies $\kappa_{\Gamma}(f) \leq 1/2$ if
$f$ is green, 
by Lemma~\ref{lem:boundary_face}, and $\kappa_{\Gamma}(f) \leq 0$ if $f$ is red.
Hence, since all of the positive curvature of $\kappa_{\Gamma}$ lies
on the boundary faces and sums to at least $1$, the diagram 
$\Gamma$ has exactly two boundary faces, $f_1$ and $f_2$ say,
both green, and $\kappa_{\Gamma}(f_1) = \kappa_{\Gamma}(f_2) =
1/2$. Now any other green face $f$ of $\Gamma$ would satisfy
$\kappa_{\Gamma}(f) < 0$, so no such face exists.

The two vertices $v_1$ and $v_2$ on $\partial(\Gamma)$ both 
satisfy $\delta_G(v_i, \Gamma) \geq 3$, so by Lemma~\ref{lem:vertex_curve} 
in Step 3 of $\ComputeRSym$ they each give curvature at
most $-1/4$ to each of $f_1$ and $f_2$. 
On the other hand, the assumptions on common prefixes of relators
imply that $\Area(\Gamma) > 2$.
So  $f_1$ and $f_2$ are both adjacent to red blobs, and by
Lemma~\ref{lem:blob_curve} they
receive a negative amount of curvature from these blobs in Step 4 of
$\ComputeRSym$, giving a contradiction.
\end{proof}

We now give a condition under which the success of $\RSym$ shows that
no $V^{\sigma}$-letter is trivial in $G$. 

\begin{thm}\label{thm:NoVletter1}
Let $\cP$ be a pregroup presentation for a group $G$.
Assume that no $R \in \cR$ has length $1$ or $2$ and that no two distinct
cyclic conjugates of relators  $R,S \in \cI(\cR)^{\pm}$ have a common
prefix consisting of all but one letter of $R$ or $S$.
If $\RSym$ succeeds on $\cP$ at level $1$, then no $V^{\sigma}$-letter
is trivial
in $G$. 
\end{thm}

\begin{proof}
Suppose that some $V^{\sigma}$-letter $x$  is equal to $1$ in $G$, and let
$\Gamma$ be a coloured  diagram over $\ip$ with
boundary word $x$, and with smallest possible coloured area for
diagrams with boundary word a single $V^{\sigma}$-letter. We do
\emph{not} assume that $\Gamma \in \calD$. We shall show that $\Gamma$
does not exist. 

Proposition~\ref{prop:semired} shows that
$\Gamma$ is semi-$P$-reduced. The diagram 
$\Gamma$ is loop-minimal by  definition, and $\Gamma$ is
$\sigma$-reduced by
Proposition \ref{prop:invtriv}. Our assumption that no $R \in \cR$ has
length $1$ implies that $\Area(\Gamma) > 1$. Let $f$ be the unique
boundary face of $\Gamma$. 

Suppose that $\Gamma$ is \valid. Then, since we chose $\Gamma$
to have minimal coloured area, Proposition~\ref{prop:blob_reduction}
implies that $\Gamma \in \calD$. Hence $\RSym$ succeeds 
on $\Gamma$ for
some $\varepsilon > 0$, and in particular 
Lemma~\ref{lem:boundary_face} implies that
 $\kappa_{\Gamma}(f) \leq 1/2$, which contradicts the
total curvature of $\Gamma$ being $1$.

Hence $\Gamma$ is not \valid.
By Proposition~\ref{prop:two_greens} (i), this can only occur when
$f$ is a red triangle. The fact that $\Gamma$ is loop-minimal implies that
$\Gamma$ looks like the left hand picture in Figure~\ref{fig:tri_meet}. 
The boundary label of the subdiagram labelled $\Theta$ in 
Figure~\ref{fig:tri_meet} is cyclically $\sigma$-reduced, since the
label of $f$ is in $V_P$, so by
Proposition~\ref{prop:two_greens} (ii), with the possible 
exception of  $v$ and the other vertex in Figure~\ref{fig:tri_meet}, which we
shall call $u$, all vertices $w$ in $\Theta$ satisfy $\delta_G(w,
\Theta) = \delta_G(w, \Gamma) \geq
2$.  We shall show that $\delta_G(u, \Gamma) \geq 2$ and
$\delta_G(v, \Gamma) \geq 2$, and hence that $\Gamma$ is in fact \valid, a contradiction.

Consider first the vertex labelled $v$ in
Figure~\ref{fig:tri_meet}. Theorem~\ref{thm:one_green} 
shows that $\delta_G(v, \Gamma) \geq 1$, so
assume, by way of contradiction, that $\delta_G(v, \Gamma) = 1$, and let $f_1$
be the green face incident with $v$. 
If $\delta_R(v, \Gamma) = 1$, then the boundary label $z$ of $f_1$ is not cyclically
$P$-reduced, a contradiction. If $\delta_R(v, \Gamma) = 2$, then
both incident red faces are adjacent to $f_1$.
Then $f_1$ can be replaced by a 
green face with boundary label an interleave of $z$, yielding a
diagram with boundary word $x$ but with smaller coloured area, a
contradiction. If $\delta_R(v, \Gamma) \geq 3$ then, since
the loop-minimality of $\Gamma$ implies that there are no loops
labelled by elements of $V_P$ 
based at $v$, we can apply Lemma \ref{lem:threereds} to reduce
$\delta_R(v, \Gamma)$ to two, and then we reach a contradiction as before. 
Hence $\delta_G(v, \Gamma) \geq 2$. 

Assume next that $\delta_G(u, \Gamma) = 1$ 
(so the only green face incident with  $u$ is the external face), and
notice that $\delta_R(u, \Gamma) \geq 3$, since $u$ is
incident twice with the red face $f$, and
with at least one other red face. If $\delta_R(u, \Gamma) = 3$,
then $\Gamma$ contains a loop at $v$ labelled by a letter from
$V_P$, contradicting the minimality of $\Gamma$.
If $\delta_R(u, \Gamma) = 4$, then $\Theta$ 
consists of two red
triangles that meet at $v$, and enclose a subdiagram $\Delta$ of boundary
length $2$, as in Figure \ref{fig:tribj_meet}. 
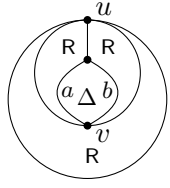
\begin{figure}
\begin{center}
\begin{picture}(80,65)(-20,5)
\thinlines
\put(13,62){$u$}
\put(10,60){\circle*{3}}
\put(10,20){\circle*{3}}
\put(10,45){\circle*{3}}
\put(10,30){\circle{60}}
\put(10,40){\circle{40}}
\put(9,6){\scriptsize \textsf{R}}
\put(0,47){\scriptsize \textsf{R}}
\put(15,47){\scriptsize \textsf{R}}
\put(6,28){\footnotesize $\Delta$}
\put(13,14){$v$}
\put(10,45){\line(0,1){15}}
\qbezier(10,20)(-13,32.5)(10,45)
\qbezier(10,20)(33,32.5)(10,45)
\put(0,31){\footnotesize{$a$}}
\put(15.5,30.5){\footnotesize{$b$}}
\end{picture}
\end{center}
\caption{Boundary vertex with red degree 4}\label{fig:tribj_meet}
\end{figure}
Let the boundary label of $\Delta$ be $ab$, and notice that $b \neq
a^{\sigma}$, as otherwise there would exist a diagram proving that $x
=_{U(P)} 1$,  contradicting 
Theorem~\ref{thm:up}. 
If $(a, b)$ or $(b, a)$ is in $D(P)$,
then there would exist a diagram consisting of $\Delta$ surrounded by
a single red triangle, which would have boundary a single letter from $V_P$
but have coloured area less than that of $\Gamma$, contradicting the
minimality of $\Gamma$. 
Hence $\Delta$ is a $\sigma$-reduced, semi-$P$-reduced, 
\valid loop-minimal coloured diagram with cyclically
$P$-reduced boundary word, and the minimality of $\Gamma$ means that
Proposition~\ref{prop:blob_reduction} shows that $\Delta \in
\calD$. 
However, we showed in Lemma~\ref{lem:norellen2}
that $\RSym$ fails on all diagrams in $\calD$ of boundary length $2$, 
contradicting our assumption that \RSym succeeds. 
Hence $\delta_R(u, \Gamma) \geq 5$, and so $\delta_R(u, \Theta) \geq 3$. Applying
Lemma~\ref{lem:threereds} to $\Theta$ (in which there are no
loops based at $u$ with labels from $V_P$),
 allows one to reduce $\delta_R(u, \Theta)$ down to three, eventually
 yielding a contradiction as in the previous 
paragraph. Hence $\delta_G(u, \Gamma) \geq 2$, and so $\Gamma$ is in fact
\valid, a contradiction. 
\end{proof}

Finally, we are able to prove that, subject to the same
conditions on $\cR$ as in Theorem~\ref{thm:NoVletter1}, if \RSym succeeds on a pregroup
presentation $\cP$ for a group $G$, then $G$ is hyperbolic. Recall
Definition~\ref{def:alternative_dehn} of a pregroup Dehn function. 

\begin{thm}\label{thm:rsymsucceeds}
Let $\cP = \langle X^\sigma  \ | \ V_P \ | \ \cR \rangle$ be a
pregroup presentation of a group $G$, and let $r$ be the maximum length of a
relator in $\cR$. 

Assume that no $R \in \cR$ has length $1$ or $2$ and that no two distinct
cyclic conjugates of relators  $R,S \in \cI(\cR)^{\pm}$ have a common
prefix consisting of all but one letter of $R$ or $S$.
\begin{enumerate}
\item[(i)]
Suppose that $\RSym$ succeeds on the presentation $\cP$ (at level $1$)
for some $\varepsilon > 0$.
Then the pregroup Dehn function of $\cI(\cP)$ is bounded above by
$$f(n) = n\left( 6+r + \frac{3+r}{2\varepsilon} \right) -
                                    \frac{3+r}{\varepsilon}.$$
\item[(ii)] If $V_P$ is empty, and $\RSym$ succeeds on $\cP$, 
then the pregroup Dehn function of $\cI(\cP)$ is bounded above by $
  n(\frac{1}{2 \varepsilon} + 1) - \frac{1}{\varepsilon}$. 
\item[(iii)] If $V_P$ is nonempty, $\cI(w) = w$ for all cyclically
  $P$-reduced words $w$, and $\RSym$ succeeds on $\cP$, 
then the pregroup Dehn function of $\cP$ 
is bounded above by $f(n)
- 2n$, where $f(n)$ is as in Part (i). 
\item[(iv)] If no
$V^{\sigma}$-letter is trivial in $G$, and 
$\RSym$ succeeds at level $d$ on $\cP$ 
then the pregroup Dehn function of $\cI(\cP)$ is bounded above
by $$n\left((3+r)\frac{(r-1)^d - 1}{r-2}
\left( 1+\frac{1}{\varepsilon}\right) +3\right) -
\frac{3+r}{\varepsilon}.$$ 
\end{enumerate}In particular, if \RSym succeeds at level 1, or if no
$V^\sigma$-letter is trivial in $G$ and \RSym succeeds at level $d$,
then $G$ is hyperbolic.
\end{thm}

\begin{proof}  
We first prove (i), by showing that $\RSym$ satisfies all conditions of
Theorem~\ref{thm:hypercurvature}.  By
Proposition~\ref{prop:r_sym_dist}, \RSym is a curvature distribution
scheme on $\calD$.
Let $\Gamma \in \calD$, and let $\kappa_{\Gamma} = \RSym(\Gamma)$. 
The fact that $\kappa_{\Gamma}$ is non-positive on vertices and red triangles
follows from Lemmas~\ref{lem:vertex_curve} and \ref{lem:blob_curve}, and it is
clear from Step 2 of $\ComputeRSym$ 
 that $\kappa_{\Gamma}(e)=0$ for each edge $e$. Hence $\kappa_{\Gamma}$ satisfies
Condition (a). 
 Condition (b) is satisfied with $d = 1$ by our assumption that \RSym
 succeeds (at level 1). 
By Lemma~\ref{lem:boundary_face}, $\kappa_{\Gamma}$ satisfies Condition (c)
with $m = 1/2$. 
By Theorem~\ref{thm:NoVletter1}, 
no $V^{\sigma}$-letter is trivial in $G$. Hence all coloured diagrams
over $\cI(\cP)$ 
are loop-minimal and green-rich, so by Proposition~\ref{prop:red_bound}
all diagrams in $\calD$ satisfy Condition (d) with 
$\lambda = 3+r$ and $\mu = 1$. 

It now follows from
Theorem~\ref{thm:hypercurvature} that if  $\Gamma \in \calD$ has
boundary length $n$ and area greater than $1$ then, as in
Corollary~\ref{cor:easy_curve}, 
$$\Area(\Gamma) \leq 
n\left( 4+r+ \frac{3+r}{2\varepsilon}\right)  -
\frac{3+r}{\varepsilon}.$$  A diagram of area $1$ has boundary length
$n \geq 3$, since no $R \in \cR$ has length less than $3$. For $n \geq
3$ the above bound evaluates
to at least 1, so in fact this area bound applies to all
$\Gamma \in \calD$.  

We showed in Proposition~\ref{prop:d_complete} that if $w =_G 1$ then
there exists an  $w' \in \cI(w)$ that is the boundary of a coloured
diagram $\Gamma \in \calD$. Hence it follows from the definition of $\cI(w)$ and
Lemma~\ref{lem:rewritetwice} that there is a coloured
diagram $\Gamma'$ with boundary word $w$ and area at most
$\Area(\Gamma)+2n$, which gives the bound in the theorem statement.

(ii) Since there are no red triangles, $\kappa_{\Gamma}$
satisfies Condition (d) of Theorem~\ref{thm:hypercurvature} with
$\lambda = 1$ and $\mu = 0$, so the formula in Part (i) simplifies as
given (and is valid for diagrams consisting of a single face).

(iii) We keep $\lambda$ and $\mu$ as in Part (i),
but take $w = w'$ in the final paragraph of the proof. 

(iv)  The assumption that no 
$V^{\sigma}$-letter is trivial in $G$
means that all diagrams are loop-minimal.
Hence as in Part (i), by Proposition~\ref{prop:red_bound} 
we can set $\lambda = 3+r$ and
$\mu = 1$ in Theorem~\ref{thm:hypercurvature}.
We can then apply Theorem~\ref{thm:hypercurvature} with $m=1$ and the specified
value of $d$ to diagrams $\Gamma \in \calD$. 
Proposition~\ref{prop:d_complete} applies, since all diagrams are loop-minimal, 
and so we can complete the argument  as in the case $d=1$.
\end{proof}

\section{A polynomial-time $\RSym$ tester}\label{sec:tester}

In this section we describe a pair of polynomial-time procedures,
$\RSymVerify(\cP, 
\varepsilon)$ and
$\RSymIntVerify(\cP, \varepsilon)$  (see
Procedures~\ref{proc:RSymVerify} and \ref{proc:RSymIntVerify})
that attempt to verify that \RSym succeeds on a given presentation
$\cP = \langle X^\sigma  \ | \ V_P \ | \ \cR \rangle$ with a given value of
$\varepsilon$. They return either $\true$ or $\fail$, together with
some additional data in the event of $\fail$.

If \true is returned, then $\RSym$ is guaranteed to succeed on
$\cP$ with constant $\varepsilon$.  If $\fail$ is returned, then it does not
necessarily mean that $\RSym$ does not succeed on
$\cP$. The additional returned data describes one or more configurations 
that could arise in a diagram in $\calD$ over $\cI(\cP)$ on which
$\RSym$ might fail,   but such a diagram may not exist. The user can
either attempt to show that such a diagram does not exist, or
 try again with a smaller value of $\varepsilon$.

For convenience of exposition, \RSymVerify works under the assumption that
$\cI(\cR) = \cR$. This is a
commonly-occurring special case  -- for example all quotients of free
products of free and finite groups can be presented this way -- and is
the case that is currently implemented (see
Section~\ref{sec:implementation}). 
\RSymIntVerify is the
generalisation of \RSymVerify to the case where $\cI(\cR) \neq \cR$,
and will be presented in Subsection~\ref{subsec:interleave}.

Since
we have had to introduce many new data structures and auxiliary
functions to produce a polynomial-time procedure, we shall start by
giving an outline of our approach. We remind the reader that our
newly-defined terms, our notation, and our procedures are all listed
in the Appendix. 

Recall from Definition~\ref{def:succeed} 
that to test whether $\RSym$ succeeds on $\cP$ means to test whether
$\RSym$  succeeds on all of the coloured
diagrams $\Gamma$ in $\calD$.  There are infinitely many 
such diagrams, but fortunately there are only finitely many elements
in $\cR$. Hence
$\RSymVerify$ will consider each possible relator $R \in \cR^{\pm}$
as the label of a  non-boundary green face $f$ of such a
diagram, and attempt to show that
$\kappa_{\Gamma}(f) \leq \varepsilon$. 

After describing some pre-processing, which is not part of
\RSymVerify, we begin the real work in
Subsection~\ref{subsec:locplace}, where we introduce some
data structures that will be used to work efficiently with our relators
as cyclic words, and to record information about the possible faces
edge-incident with such a green face $f$ labelled by $R$. 

The curvature $\kappa_{\Gamma}(f)$ is equal to $1 + \sum_{v \in \partial(f)} \chi(v, f,
\Gamma) + \sum_{B} \chi(B, f, \Gamma)$, where the second sum is over
those red blobs $B$ which share an edge with $f$. We will therefore
bound $\kappa_{\Gamma}(f)$ by bounding $\chi(v, f,
\Gamma)$ and $\chi(B, f, \Gamma)$, over all diagrams $\Gamma \in
\calD$ containing a face $f$ labelled by $R$. There are infinitely
many possible vertices and red blobs in diagrams in $\calD$, but we
shall take a pragmatic approach of determining partial data about
those vertices and blobs which could give $f$ close to zero curvature in
Steps 3 and 4 of \ComputeRSym, and otherwise use an upper bound of $-1/3$ for
vertices and $-5/14$ for red blobs. This analysis results in the
creation of two functions, called \Vertex and \Blob (Algorithms~\ref{alg:vertex}
and \ref{alg:blob}), which take as input information about $f$ and its 
adjacent faces, and return these curvature bounds. 

The next complication arises from our desire for a low degree
polynomial time algorithm. This means
that we cannot run, say, a backtrack search which tries
all possible ways of fitting neighbouring faces 
around $f$ to bound $\kappa_{\Gamma}(f)$, as this would lead to a complexity with the length
$r$ of the longest relator in the exponent. To keep the complexity
where we want it, we use a combinatorial result:
Lemma~\ref{lem:GUSU}. This tells us that, provided we associate to each
vertex on $f$  the length of the consolidated edge before it, and require
the cumulative curvature coming from red blobs and vertices to be less
than the proportion of $|R|$ that their edges take up, we can
avoid doing a backtrack search. 
To make this approach work, we need a careful analysis of edge lengths
and their corresponding curvature values, and this is the content of
Subsection~\ref{subsec:onestep}.

In Subsection~\ref{subsec:tester}, we shall finally be able to present
\RSymVerify, then in Subsection~\ref{subsec:effic} we shall prove
Theorem~\ref{thm:r_sym_complexity}, which states that \RSymVerify runs
in time $O(|X|^5 + r^3|X||\cR|^2)$, where $r$ is the length of the
longest relator.

Finally, in Subsection~\ref{subsec:interleave}, we shall describe the
modifications that must be made if $\cI(\cR) \neq \cR$, and present
Procedure~\ref{proc:RSymIntVerify} (\RSymIntVerify), which is the
generalisation of \RSymVerify to this case.  The outline
of the procedure and the majority of the subroutines barely change,
but we must work with slightly more complex data structures to handle
the potentially exponential size set $\cI(\cR)$, whilst still
terminating in polynomial time.

The majority of the subroutines used by \RSymVerify will be useful
for testing other curvature distribution schemes, not just \RSym.

\subsection{Preprocessing}\label{subsec:preprocess}

Before running \RSymVerify, we assume that some preprocessing
has been done to the presentation, to ensure that the
assumptions of  Theorem~\ref{thm:rsymsucceeds} hold, and to improve 
the likelihood that \RSym succeeds. The first two steps of
preprocessing are done \emph{before} a pregroup $P$ is chosen, when we
just have a group presentation $\langle X \mid \cR \rangle$. 

\smallskip

\noindent
{\bf Preprocessing Step 1:}
Eliminate any relators of the form $x$ or $xy$ with $x, y \in X$,
and $x \neq y$, by eliminating generators.
Delete any relators of the form $x^2$ with $x \in X$, and
require that $x^{\sigma} = x$ in the pregroup $P$. 

\smallskip

\noindent
{\bf Preprocessing Step 2:} 
Look for pairs $R_1$, $R_2 \in \cR$ for which there are distinct
cyclic conjugates $S_1$, $S_2$ of $R_1^{\pm 1}$, $R_2^{\pm 1}$ that have
a common prefix of length greater than half of $|S_1|$. That is, $S_1=ww_1$
and $S_2=ww_2$ with $|w|>|w_1|$. If $R_1 \ne R_2$, then replace
$R_2$ by the shorter relator $w_1^{-1}w_2$. 
Notice that it is not possible to have $R_1=R_2$ and $S_1 \ne S_2$ with
$|w_1|=|w_2|=1$.

\smallskip

We may need to carry out these two steps repeatedly, but at the end of
them all relators have length at least $3$, and no two distinct cyclic conjugates of
relators  $R,S \in \cR^{\pm}$
have a common prefix consisting of all
but one letter of $R$ or $S$.
See, for example, \cite[Section 5.3.3]{HEO} for discussion of how to
do this efficiently. Note that we are only carrying out Steps 1 and 3
of the simplification in \cite{HEO}:  we are not attempting to
eliminate generators using relators of length greater than two.

\smallskip

\noindent
{\bf Preprocessing Step 3:} Define a pregroup
 $P$ on the remaining generators
(adding additional generators if necessary to close $P$), and ensure  that all
elements of $\cR$ are cyclically $P$-reduced. Ideally, as many remaining
relators of length $3$ as possible should become elements of $V_P$
rather than $\cR$. 

\smallskip

After this, all hypotheses in Theorem~\ref{thm:rsymsucceeds}  are
satisfied. Recall that we assume for now that after Preprocessing Step 3,
$\cI(\cR) = \cR$: see Subsection~\ref{subsec:interleave} for further
preprocessing that is required if this assumption does not hold.

\subsection{Steps, locations and places}\label{subsec:locplace}

\RSymVerify needs to check that, for every diagram 
$\Gamma \in \calD$, every non-boundary internal
green face $f \in \Gamma$ receives at most $- 1 - \varepsilon$ of 
curvature from its incident vertices and red blobs in Steps 3 and 4 of
\ComputeRSym.
Each such face $f$ has boundary label some $R \in
\cR^{\pm} = \cI(\cR^{\pm})$, and $\partial(f)$ is split up into the
consolidated edges (see Definition~\ref{def:van_kampen})
that $f$ shares with its adjacent faces in $\Gamma$. In this
subsection, we describe how we represent this decomposition into
consolidated edges.

The idea is to consider each relator $R \in \cR$ in turn, and to
use the relators in $V_P$ and $\cR^{\pm}$ to determine the possible
decompositions into consolidated edges $e_i$ of the boundary of a
non-boundary face $f$ with label $R$.
We do not need to consider
$\cR^{-1}$, as the situation for each $R^{-1}$ will be equivalent to that for $R$.
Each such decomposition corresponds to an expression of some
cyclic conjugate $R'$ of $R$ as $w_1w_2 \cdots w_k$, where $w_i$ is the label
of $e_i$.

We attach a colour $C_i \in \{\sfG,\sfR\}$ to each $w_i$ in the
decomposition, which is the colour of the adjacent region (green face or red
blob): there could be more than
one decomposition $R' = w_1w_2 \cdots w_k$ with the same $w_i$ but with
different colours $C_i$. We have assumed that $\Gamma \in \calD$, so
$\Gamma$ is  \valid. Hence if 
$C_i = \sfR$ then  $|w_i|=1$. For
reasons that will become clear shortly, we do not allow $C_1=\sfR$ and
$C_k=\sfG$; in that situation, we shall instead consider the decomposition
$w_2 \cdots w_kw_1$.

\begin{defn}\label{def:edge_curve}
If $C_i = \sfG$ then  let  $\epsilon_i$ be the maximum value of 
$\chi(v_i, f, \Gamma)$, considered over all possible
diagrams $\Gamma \in \cal{D}$ in which $w_i$ labels a maximal green
consolidated edge on $f$. If $C_i = \sfR$ then let $\epsilon_i$ be the maximum
value of $\chi(v_i, f, \Gamma) + \chi(B, f, \Gamma)$, considered
over all possible diagrams $\Gamma \in \cal{D}$ on which $w_i$ labels a
red consolidated edge of $f$,  and all possible incident red blobs $B$ at $w_i$. 
\end{defn}

If we find a
decomposition with $\epsilon_1 + \cdots + \epsilon_k > -1-\varepsilon$,
then \RSymVerify returns \fail and gives details of the decomposition.

We shall see later, in  Lemma \ref{lem:step_bound}, that $\epsilon_i \le -1/6$ when
$C_i = \sfR$ and when $C_i=C_{i+1} = \sfG$, but when $C_i = \sfG$ and
$C_{i+1}=\sfR$, we can have $\epsilon_i=0$. In our main algorithm it is convenient
to have an upper bound of $-1/6$ on all curvature contributions and, for this
reason, we combine subwords $w_iw_{i+1}$ with $C_i = \sfG$ and $C_{i+1}=\sfR$
into a single unit, which we call a {\em step}, and use the curvature
contributions from these steps, rather from each individual $w_i$.

\begin{defn}\label{def:step}
For a given coloured decomposition $R' = w_1w_2 \cdots w_k$ as above,
a {\em step} consists either of a single subword $w_i$, or of two consecutive
subwords $w_iw_{i+1}$, determined as follows, where subscripts should be
interpreted cyclically.
\begin{enumerate}
\item[(i)] If $C_i=\sfG$ and $C_{i+1}=\sfR$, then $w_iw_{i+1}$ is a step.
\item[(ii)]  If neither $w_{i-1}w_{i}$ nor $w_iw_{i+1}$ is a step by 
Condition (i), then $w_i$ is a step.
\end{enumerate}
We have disallowed the combination $C_1=\sfR$, $C_k=\sfG$, so this cannot give
rise to a step $w_kw_1$. Note that the steps are the same for any
cyclic permutation of the decomposition that does not violate the
$C_1=\sfR$, $C_k=\sfG$ condition.

Let $\epsilon_i$ be as in Definition~\ref{def:edge_curve}. 
We define the {\em stepwise curvature} $\chi$ of a
step to be $\chi = \epsilon_i$ when
the step is $w_i$ and $\chi = \epsilon_i+\epsilon_{i+1}$ when it is $w_iw_{i+1}$.
The {\em length} of a step is the number of letters of $R'$ that it comprises.
\end{defn}

To carry out the required estimates of upper bounds on
the curvature $\epsilon_1 + \cdots + \epsilon_k$, we need to study the possible
diagrams $\Gamma$ in which the decomposition $w_1 \cdots w_k$ of $R$ under
consideration can occur. We can choose the amount of detail in which
we analyse possible neighbourhoods of the face labelled $R$ in such diagrams.
More detail may lead to better estimates, but will take longer to compute.
In \RSymVerify, we generally  
limit our consideration to the faces of $\Gamma \in \calD$ that
have at least one edge in common with the face labelled $R$.

To devise efficient algorithms for carrying out this analysis and for storing
the information in a useful form, we need to devise some suitable data
structures. The definitions of {\em locations} and {\em places} that follow may
seem somewhat arbitrary on a first reading but, after some experimentation,
they have turned to be efficacious for the purpose in hand.

\begin{defn}\label{def:location}
Let $R \in \cR^{\pm}$, and fix a word $w = x_1x_2 \cdots x_{|w|}$ such that
$R = w^k$ with $k$ maximal amongst such expressions for $R$. 
A \emph{location} on $R$ is an ordered triple $(i, a, b)$, denoted
$R(i, a, b)$,  where
$i \in \{1, \ldots, |w|\}$, $a = x_{i-1}$ (or $x_{|w|}$ if $i  = 1$),
and $b = x_i$. 
\end{defn}

For example, if $R = abab = (ab)^2$ then the locations are
$R(1, b, a)$ and $R(2, a, b)$. We shall present a method to find
such a $w$ and $k$ in the proof of
Theorem~\ref{thm:r_sym_complexity}. Note that writing $R$ as a proper
power, where possible, is not essential to the running of our
algorithm, but merely reduces the number of locations. It therefore
helps with hand calculations, as in Section~\ref{sec:examples}, and
also with the running time of our implementations, but if the reader
wishes to think of every consecutive pair of letters on every relator as being a
different location, no great harm will be done. 

In the case of places, we needed to distinguish between places that
could conceivably occur, and those 
that really do occur in some diagram $\Gamma \in \calD$: recall that
all such diagrams are $\sigma$-reduced.

\begin{defn}\label{def:place}
A \emph{potential place} $\bP$ is a triple $(R(i, a, b), c, C)$,
where $R(i, a, b)$ is a location, $c \in X$,  and $C \in \{\sfG, \sfR\}$.
A potential place is a \emph{place} if it is \emph{instantiable},
in the following sense.
(See Figure \ref{fig:place}\,(a).)
\begin{enumerate}
\item[(i)] There exists a $\sigma$-reduced diagram $\Gamma$ (see
  Definition~\ref{def:sigma_reduced}) with a face $f$ labelled $R$, a
  face $f_2$ meeting $f$
  at $b$, and a vertex  between $a$ and $b$ on $\partial(f)$ of degree
  at least three; 
\item[(ii)] the half-edge
  on $f_2$ after $b^{\sigma}$ is labelled $c$;
\item[(ii)] if $C = \sfG$ then $f_2$ is green, and if $C = \sfR$ then $f_2$ is a
red blob.
\end{enumerate}
We say that $\bP$ is \emph{green} if $C = \sfG$ and
\emph{red} otherwise.
\end{defn}
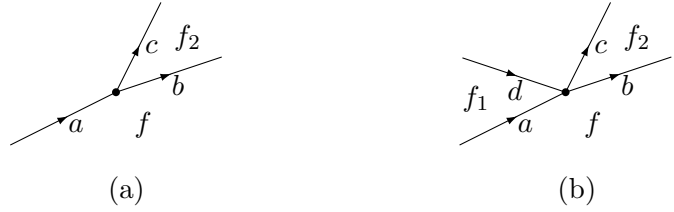
\begin{figure}
\begin{center}
\begin{picture}(200,75)(-3,-5)
\thinlines
\put(10,30){\circle*{3}}
\put(10,30){\line(3,1){40}}
\put(31,37){\vector(3,1){0}}
\put(-30,10){\line(2,1){40}}
\put(-8,21){\vector(2,1){0}}
\put(10,30){\line(1,2){17}}
\put(19,48){\vector(1,2){0}}
\put(-8,15){$a$}
\put(31,29){$b$}
\put(21,45){$c$}
\put(32,48){$f_2$}
\put(17,15){$f$}
\put(7,-10){(a)}

\put(180,30){\circle*{3}}
\put(180,30){\line(3,1){40}}
\put(201,37){\vector(3,1){0}}
\put(140,10){\line(2,1){40}}
\put(162,21){\vector(2,1){0}}
\put(180,30){\line(1,2){17}}
\put(189,48){\vector(1,2){0}}
\put(141,43){\line(3,-1){39}}
\put(162,36){\vector(3,-1){0}}
\put(162,15){$a$}
\put(201,29){$b$}
\put(191,45){$c$}
\put(158,27){$d$}
\put(202,48){$f_2$}
\put(187,15){$f$}
\put(141,25){$f_1$}
\put(177,-10){(b)}
\end{picture}
\end{center}
\caption{(a) Instantiation of place on face
  $f$;  (b) Instantiation of partial vertex on face $f$}
\label{fig:place}
\end{figure}

Notice that if $C = \sfG$ then, by the fact that $\Gamma$ is 
$\sigma$-reduced, there
exists a location $R'(j, b^{\sigma}, c)$ such that the label of $R'$ beginning
at $b^{\sigma}$ is not equal in $F(X^\sigma$) to the inverse of the label of $R$ that ends at $b$.
If $C = \sfR$, then $b^{\sigma}$ must intermult with $c$, as in
Definition~\ref {def:intermult}.

\RSymVerify computes an array of all intermult pairs, and then
finds all locations of relators
$R \in \cR^{\pm}$. For each location $R(i, a, b)$ with  $R \in \cR$,
it must find all instantiable places. 
To do so, it considers 
each letter $c \in X$, and each $C \in \{\sfR, \sfG\}$.
For $C = \sfR$ it checks that $b^{\sigma}$ intermults with $c$.
For $C = \sfG$ it checks that there exists a location $R'(j, b^{\sigma}, c)$ such
that a diagram of area two with faces labelled by $R$ and $R'$, sharing
the edge labelled $b$, is $\sigma$-reduced.

\subsection{Vertex data and the \Vertex function}\label{subsec:vertex}
Fix a relator $R \in \cR$ and let $f$ be a
non-boundary face labelled $R$ in a diagram in $\calD$.
Whilst decomposing  $R$ into steps, we find an upper bound on $\sum_{v
  \in \partial(f)} \chi(v, f, \Gamma)$. 
Although there are infinitely many possible vertices $v$ that can arise in such a
diagram, almost all  give  $f$ not much more than $-1/2$ of
curvature.  In this subsection we show how to 
produce a list of descriptions of the finitely many possible vertices $v$
on $\partial(f)$ such that $\chi(v, f, \Gamma) \geq -1/3$. 
For reasons of efficiency, we
do not completely describe each such  $v$, but only store
information on the possibilities for three consecutive incident faces,
reading anticlockwise, together with an upper bound on $\chi(v, f,
\Gamma)$, over all $\Gamma \in \calD$.

From now on, we shall often think of the triangles within a red blob as having been
merged, and treat the blob as having no internal structure other than
its area. As a
consequence, we will never consider a vertex to have more than
one consecutive incident red face. The following lemma therefore
bounds the total degree of the vertices that we shall classify, as
well as their green degree.

\begin{lemma}\label{lem:vert_bounds}
Let $v$ be a vertex in a diagram $\Gamma \in \calD$,
and let $f$ be an internal green face incident with $v$. If $v$ is
incident more than once with the external face, then $\chi(v, f,
\Gamma) \leq -1/2$. 
Otherwise,  the curvature $\chi(v, f, \Gamma)$  is as in Table~\ref{tab:vertex_bounds}.
\end{lemma}

\begin{table}\caption{Vertex curvature $\chi(v, f, \Gamma)$}\label{tab:vertex_bounds}
\begin{center}
\begin{tabular}{c|cc}
$\delta_G(v, \Gamma)$ & $v \not\in \partial(\Gamma)$ &
                                                           $v \in \partial(\Gamma)$\\
\hline
2 & 0 & 0 \\
3 & $-1/6$ & $-1/4$ \\
4 & $-1/4$ & $-1/3$ \\
5 & $-3/10$ & $(-3/8)$\\
6 & $-1/3$ & $(-2/5)$ \\
$\ge 7$ & $(\le -5/14)$ & ( $\le -5/12$)\\
\end{tabular}
\end{center}
\end{table}

\begin{proof}
Since $\Gamma \in \calD$, the diagram $\Gamma$ is \valid, so $\delta_G(v, \Gamma) \ge 2$. Let $x$ be the number
of times that $v$ is incident with the external face. Then
$\delta_G(v, \Gamma) > x$ because $f$ is 
incident with $v$.
By Lemma~\ref{lem:vertex_curve}, $v$ gives curvature 
$\chi = \chi(v, f, \Gamma) = (2 - \delta_G(v, \Gamma))/(2(\delta_G(v, \Gamma) - x))$ to $f$. 
The result follows by computing $\chi$ for $\delta_G(v, \Gamma) \le 7$
and $0 \leq x < \delta_G(v, \Gamma)$, and noticing that $\chi$ decreases as
$\delta_G(v, \Gamma)$ and $x$ increase.
\end{proof}

Let $\bP = (R(i, a, b), c, C)$ be a place, and let $v$ be a
vertex on $\partial(f)$ between the edges labelled $a$ and $b$ in a
diagram $\Gamma \in \calD$ that instantiates $\bP$,  in which $f$ is non-boundary.
Let $f_1$ and $f_2$ be the faces incident with the edges 
labelled $a$ and $b$, respectively,
as in  Figure~\ref{fig:place}\,(b), so that $f_2$ is as in
Definition~\ref{def:place}. We describe how to encode this situation,
and efficiently bound  $\chi(v, f, \Gamma)$.

Suppose that $v$ is internal and  has degree $k$. Let 
the outgoing edges from 
$v$, reading anticlockwise, be labelled $x_1=a^\sigma$, $x_2=b$, $x_3=c$,
$x_4,\ldots,x_k$, and let  the face adjacent to $v$ with edges labelled
$x_i^\sigma$, $x_{i+1}$ have colour $C_i$ for $1 \le i \le k$ (with $x_{k+1}=x_1$).
In particular, $C_1 = \sfG$ and $C_2 = C$.
To help us compute the required upper bound on $\chi(v, f, \Gamma)$, we record
these edge labels and face colours in a certain directed graph $\cG$.
To avoid confusion with the vertices and edges of coloured diagrams,
we shall always refer to the vertices and edges of $\cG$ as
{\em $\cG$-vertices} and {\em $\cG$-edges}.

The location $R(i, a, b)$
of $v$ on $f$ is encoded by the $\cG$-vertex $(a^\sigma, b, \sfG)$,
and there exist $\cG$-vertices labelled $(x_i^\sigma,x_{i+1},C_i)$ for each
$i$. There are
$\cG$-edges from $\cG$-vertex $(x_{i-1}^\sigma,x_i,C_{i-1})$ to $\cG$-vertex
$(x_i^\sigma,x_{i+1},C_i)$ for each $i$, so the vertex $v$ of $\Gamma$  is
represented by a circuit of length $k$ in $\cG$ starting
at the $\cG$-vertex $(a^\sigma,b,\sfG)$.  We assign the
$\cG$-edge from $(x_{i-1}^\sigma,x_i,C_{i-1})$ to $(x_i^\sigma,x_{i+1},C_i)$
the weight $1$ if $C_{i-1} = \sfG$ and $0$ if $C_{i-1} = \sfR$. Then
$\delta_G(v,\Gamma)$ is equal to the total weight of the circuit in
$\cG$, and so we can bound $\chi(v, f, \Gamma)$ by bounding the weight
of circuits in $\cG$ that start at $(a^\sigma, b, \sfG)$.

We shall now present the formal definition of the vertex graph $\cG$.

\begin{defn}\label{def:vertex_graph}
The \emph{vertex graph} $\cG$ of $\cP$ has $\cG$-vertices of the
form $(a,b,C)$ with $a,b \in X$ and $C \in \{ \sfG, \sfR \}$.  
There is a \emph{green $\cG$-vertex} $(a,b,\sfG)$ if and only if there
exists a location $R(i,a,b)$.
There is a \emph{red $\cG$-vertex} $(a,b,\sfR)$ if and only if $(a,b)$ is an
intermult pair (see Lemma~\ref{lem:intermult}).

There is a (directed) $\cG$-edge from $(a, b, \sfG)$ to $(b^{\sigma}, c, \sfG)$
if there exist
 locations $R(i,a,b)$ and $R'(j, b^\sigma, c)$ such that the one-face
 or two-face
diagram in which faces labelled $R$ and $R'$ share this edge
labelled $b$ is $\sigma$-reduced.
There is a $\cG$-edge from each $(a, b, \sfG)$ to each
$(b^{\sigma},c,\sfR)$. 
There is a $\cG$-edge from each
$(a,b,\sfR)$ to each $(b^{\sigma},c,\sfG)$. 
There are no $\cG$-edges between red $\cG$-vertices, since we do not
allow red blobs to share edges with other red blobs. 
The $\cG$-edges
have weight 1 if their source is green and weight 0 if it is red.
\end{defn}

After computing the list of all places, the next
step in \RSymVerify is to construct $\cG$.  We also 
store a list of the locations that correspond to each green
$\cG$-vertex. For each $\cG$-vertex $\nu$, and for each path $\nu_1, \nu, \nu_2$ in
$\cG$,
we let $w(\nu_2, \nu_1)$ denote the smallest weight of a path in $\cG$ from
$\nu_2$ to $\nu_1$
with at least one $\cG$-edge. (So in the case $\nu_1=\nu_2$ this
cannot be $0$.)  If there is no $\cG$-path from $\nu_2$ to $\nu_1$, then we take its
weight to be infinite. We can use the Johnson-Dijkstra algorithm \cite{Johnson} to
find and store the weights of all of these paths.

The \texttt{Vertex} function takes as input a triple $(\nu_1, \nu,
\nu_2)$ of $\cG$-vertices for which $\nu$ is green and there is a 
(directed) path $\nu_1, \nu, \nu_2$ in
$\cG$. 
The existence of such a directed path means that
$\nu_1$, $\nu$ and $\nu_2$ represent subwords of boundary
labels, and colours,  of adjacent faces
$f_1$, $f$ and $f_2$ around a vertex $v$ in a coloured diagram, as in
Figure~\ref{fig:place}\,(b). 
The \Vertex
function returns an upper bound on $\chi(v, f, \Gamma)$.

\begin{alg}\label{alg:vertex}
\texttt{$\Vertex(\nu_1, \nu, \nu_2)$:} Require that $(\nu_1, \nu)$ and
$(\nu, \nu_2)$ are $\cG$-edges. 
\begin{enumerate}
\item
If $\nu_1$ and $\nu_2$ are both green, then
return $-1/6$, $-1/4$, $-3/10$, or $-1/3$ when
$w(\nu_2, \nu_1)$ is respectively
$1$, $2$, $3$, or greater than $3$.
\item
If $\nu_1$ is green and $\nu_2$ is red, then return
$0$, $-1/6$, or $-1/4$ when
$w(\nu_2, \nu_1)$ is respectively
$0$, $1$, or greater than $1$.
\item
If $\nu_1$ is red and $\nu_2$ is green, then
return $0$, $-1/6$, or $-1/4$ when $w(\nu_2, \nu_1)$ is respectively
$1$, $2$, or greater than $2$.
\item If $\nu_1$ and $\nu_2$ are both red, then
return $0$.
\end{enumerate}
\end{alg}

\begin{lemma}\label{lem:vertex_correct}
Let $f_1, f, f_2$ be three consecutive faces around a vertex $v$ in a
diagram $\Gamma \in \calD$, in locations corresponding to $\cG$-vertices $\nu_1$,
$\nu$, $\nu_2$. Then $\chi(v, f, \Gamma) \leq \Vertex(\nu_1, \nu, \nu_2)$.
\end{lemma}

\begin{proof}
It is clear from Definition~\ref{def:vertex_graph} that there are
$\cG$-vertices $\nu_1$, $\nu$, $\nu_2$, as required, and 
$\cG$-edges from $\nu_1$ to $\nu$ and from $\nu$ to $\nu_2$.  Let
$\chi = \chi(v, f, \Gamma)$, and let $\beta = \Vertex(\nu_1, \nu,
\nu_2)$.  

Assume first that  $\nu_1$ and $\nu_2$ are both green, so that
$\delta_G(v, \Gamma) \geq 3$, and $\delta_G(v, \Gamma) \geq 4$ if $v$ is boundary. 
If $\chi > -1/3$ then, by Lemma~\ref{lem:vert_bounds}, 
 $v$ is not boundary and $\delta_G(v, \Gamma) \leq 5$,
so $w(\nu_2, \nu_1) \leq 3$. If $\beta = -1/3$,
then $w(\nu_2, \nu_1) \geq 4$. Hence the shortest $\cG$-path from
$\nu_2$ to $\nu_1$ passes through at least three additional green 
$\cG$-vertices (and possibly some red ones), and so either $v$ is boundary or
$\delta_G(v, \Gamma) \geq 6$. Hence by Lemma~\ref{lem:vert_bounds}, $\chi \leq
-1/3 = \beta$. If $\beta  =
-3/10$, then $w(\nu_2, \nu_1) = 3$, so either $v$ is boundary or
$\delta_G(v, \Gamma) \geq 5$. Therefore by
Lemma~\ref{lem:vert_bounds}, $\chi \leq -3/10 = \beta$. If $\beta = -1/4$,  then
$w(\nu_2, \nu_1) = 2$, so $\delta_G(v, \Gamma) \geq 4$, and hence by
Lemma~\ref{lem:vert_bounds}, $\chi \leq -1/4 = \beta$. 
Similarly, if $\beta = -1/6$ then $w(\nu_2, \nu_1) = 1$, so
$\delta_G(v, \Gamma) \geq 3$, and $\chi \leq -1/6 = \beta$.

Next assume that $\nu_1$ is green and $\nu_2$ is red, so that
$\delta_G(v, \Gamma) \geq
 2$, and  $\delta_G(v, \Gamma) \geq 3$ if $v$ is boundary. 
If $\beta = -1/4$ then $w(\nu_2, \nu_1) \geq 2$, so the shortest
$\cG$-path from $\nu_2$ to $\nu_1$ passes through at least two additional
green vertices. Hence either $v$ is boundary, or $\delta_G(v, \Gamma) \geq 4$,
and so by Lemma~\ref{lem:vert_bounds}  $\chi \leq -1/4 = \beta$. If
$\beta = -1/6$ then $w(\nu_2, \nu_1) = 1$, so either $v$ is boundary
or $\delta_G(v, \Gamma) \geq 3$, and hence $\chi \leq -1/6 = \beta$.  

The case $\nu_1$ red and $\nu_2$ green is similar, except that the
weights $w(\nu_2, \nu_1)$ are increased by one, since every edge leaving
$\nu_2$ has weight one. 

Finally, if $\nu_1$ and $\nu_2$ are both red, then $\beta = 0$, 
which by Lemma~\ref{lem:vert_bounds} 
is an upper bound on $\chi$. 
\end{proof}

\begin{remark}\label{rem:vert_blob_boundary}
When testing $\RSym$ at level $1$, each non-boundary face $f \in \Gamma$
may be at dual distance two from the external face. Since we are not
recording whether the next edges on $f_1$ and $f_2$ (the ones labelled
$c$ and $d$ in Figure~\ref{fig:place}\,(b)) are boundary
edges, we allow the $\Vertex$ function to consider them as
boundary. See 
Remark~\ref{rem:blob_boundary}. 

Similarly, as a consequence of the possible presence of  the external face at dual
distance two, the bracketed values of $\chi(v, f, \Gamma)$
in Table~\ref{tab:vertex_bounds} are not used by
\RSymVerify, and we impose a lower
bound of $-1/3$ on the return values of the $\Vertex$ function. 
However if, for example, we are testing $\RSym$ at level $2$ and $V_P
= \emptyset$ (so there are no red triangles in any diagram, and
$U(P)$ is a free product of copies of $\Z$ and $C_2$), then
all boundary vertices $v$ of a face $f$ to be tested 
 satisfy $\delta_G(v, \Gamma) \ge 6$, and so we can make use of these smaller
curvature values.
\end{remark}

\subsection{Red blob data and the \Blob function}\label{subsec:blobdata}

Similarly to vertices, there are infinitely many possible red blobs in
diagrams in $\calD$. The methods described in this subsection collect
information about possible red blobs $B$ such that there exists a
diagram $\Gamma \in D$ and a green face $f$ of $\Gamma$, with 
$\chi(B, f, \Gamma) > -5/14$.

\begin{lemma}\label{lem:tab_blob_bound}
Let $B$ be a red blob in a diagram $\Gamma \in \calD$,  let $f$ be an
internal green face adjacent to $B$ at an edge $e$. If $B$ is not
simply connected  or if $|\partial(B) \cap \partial(\Gamma)| \geq 2$ 
then $\chi(B, f, \Gamma) \leq -1/2$. 
If $\chi(B, f, \Gamma) > -5/14$, then 
$|\partial(B)|$, $|\partial(B) \cap \partial(\Gamma)|$, and 
$\chi(B, f, \Gamma)$ are as in Table~\ref{tab:blob_bound}.
\end{lemma}

\begin{proof}
Let $B$ have boundary length $l$ and area $t$. 
If $B$ is not
simply connected,  then 
by Lemma~\ref{lem:blob_bound}, $l \leq t$, so $\chi(B, f, \Gamma) \leq  -1/2$. 
Hence $B$ is simply connected and $\Gamma$ is \valid, and so $t = l-2$, by
 Lemma~\ref{lem:blob_bound}. The result now follows easily from Lemma~\ref{lem:blob_curve}.
\end{proof}

\begin{table}\caption{Bounds on simply connected 
red blob curvature}\label{tab:blob_bound}
\begin{center}
\begin{tabular}{c|c|c}
$|\partial(B)|$ & $|\partial(B) \cap \partial(\Gamma)|$ & $\chi(B, f, \Gamma)$ \\
\hline
3 & 0 & $-1/6$ \\
3 & 1 & $-1/4$ \\
4 & 0 & $-1/4$ \\
4 & 1 & $-1/3$ \\
5 & 0 & $-3/10$ \\
6 & $0$ & $-1/3$\\
\end{tabular}
\end{center}
\end{table}

Recall from Lemma~\ref{lem:intermult} and Definition~\ref{def:calD}
that two consecutive letters of the boundary word
of a red blob in a diagram in $\calD$ must intermult. 
  We create a function $\Blob(a, b, c)$, which takes as
input $(a, b, c) \in X^3$ such that $(a,b)$ and $(b,c)$ intermult,
and returns an upper  bound  on 
\begin{center}
$\{\chi(B, f, \Gamma) \ : \ \Gamma \in
\calD $ contains both a red blob $B$ with $abc$ a
subword of its boundary label, and a green face $f$ that is incident with $B$ at $b\}.$
\end{center}

\begin{defn}\label{def:cr}
We call $x \in X$ an \emph{$\cR$-letter} if $x$ occurs in an
element of $\cI(\cR^{\pm})$. 
\end{defn}
Notice that non-$\cR$-letters can only appear on the boundary of a
diagram, and that  our present assumption that $\cI(\cR) = \cR$
implies 
that an $\cR$-letter occurs in an
element of $\cR^{\pm}$.

It is straightforward (see the proof of Theorem~\ref{thm:r_sym_complexity}
for details) to compute a list $\mathcal{B}$ of all cyclic words $w \in X^\ast$ that satisfy all of the
following conditions.
\begin{enumerate}
\item The word $w$ is equal to $1$ in $U(P)$.
\item $3 \leq |w| \leq 6$. 
\item Each consecutive pair of letters in $w$ intermult.
\item No proper nonempty subword of $w$ is equal to $1$ in $U(P)$.
\item $w$ contains at most one
  non-$\cR$-letter, and none if $|w| > 4$.
\end{enumerate}

When \RSymVerify bounds $\chi(B, f, \Gamma)$,
 it will have specified three consecutive letters $a, b, c$
on $\partial(B)$.

\begin{alg}\label{alg:blob}$\Blob(a, b, c):$ Require that $(a, b), (b,
  c)$
  intermult. 
\begin{enumerate}
\item If $abc$ is a cyclic subword of
a word in $w \in \mathcal{B}$ then return  the maximal
curvature from Lemma~\ref{lem:tab_blob_bound} over all such words $w$,
with $|\partial(B) \cap \partial(\Gamma)|$ assumed to be nonzero if
and only if $w$ contains a non-$\cR$-letter. 
\item Otherwise, if at most one of $a$ and $c$ is not an $\cR$-letter then
  return  $-5/14$.
\item Otherwise, return  $-1/2$. 
\end{enumerate}
\end{alg}

\begin{lemma}\label{lem:blob_correct}
Let $B$ be a red blob with subword $abc$ of
its boundary word, in a diagram $\Gamma \in \calD$,
 and let $f$ be the green face incident with $B$ at
$b$. Then $\chi(B, f, \Gamma) \leq \Blob(a, b, c)$.
\end{lemma}

\begin{proof}
By Lemma~\ref{lem:intermult} both $(a, b)$ and $(b, c)$ intermult, so
$\Blob(a, b, c)$ is defined. 

If $B$ is not simply connected, then from
 Lemma~\ref{lem:tab_blob_bound} we see that $\chi(B, f, \Gamma) \leq -1/2
\leq \Blob(a, b, c)$, so assume that
$B$ is simply connected, and let $w$ be the boundary word of $B$.
If $w$ contains at least two $\cR$-letters, then
$|\partial(B) \cap \partial(\Gamma)| \ge 2$, and so $\chi(B, f, \Gamma)
\leq -1/2 \leq \Blob(a, b, c)$, by
Lemma~\ref{lem:tab_blob_bound}, so assume that $w$ contains at most
one non-$\cR$-letter, and in particular that at most one of $a$ and
$c$ is not an $\cR$-letter.

Let $l$ be the length of $w$.  If $l \ge 7$ then $\chi(B, f,
\Gamma) \le -5/14 \leq \Blob(a, b, c)$, so assume that $l
\leq 6$. If $l \in \{5, 6\}$ and $w$ contains an $\cR$-letter, then
$\chi(B, f, \Gamma) \leq -5/14 \leq \Blob(a, b, c)$, so assume not. 
 
Then: $w$ is equal to
$1$ in $U(P)$, since $B$ is simply-connected; 
each consecutive pair of letters of $w$ intermult, by
Lemma~\ref{lem:intermult}; and no proper empty subword of $w$ is equal
to $1$ in $U(P)$, since $\Gamma \in \calD$ (see
Definition~\ref{def:calD}). 
Hence $w \in \mathcal{B}$, and
so $\chi(B, f, \Gamma) \leq \Blob(a, b, c)$, as required. 
\end{proof}

\begin{remark}\label{rem:blob_boundary}
In Remark~\ref{rem:vert_blob_boundary} we observed that
 the \Vertex function often assumes that the
face $f$ is dual distance two from the external face. At present this data is not being used by
the \Blob function, which may be bounding curvature as if any corresponding blobs
have no boundary edges. Some curvature is potentially being missed.
We plan to rectify this in future versions of $\RSymVerify$, 
by modifying our definition of places to also record whether the edge labelled
by the extra letter
is on the boundary, and hence enabling the $\Vertex$ and $\Blob$
functions to use this extra data. 
\end{remark}

\subsection{One-step reachable places and the \OneStep lists}\label{subsec:onestep}
Recall Definition~\ref{def:step} of \emph{step} and 
\emph{step curvature}. In this subsection, we describe how to find the
steps, and use the \Vertex and \Blob functions to bound the
corresponding step curvature. 
For each place $\bP$ on each relator $R$, we shall create
a list $\OneStep(\bP)$ of those places $\bQ$ on $R$ that can be
reached from $\bP$ in a single step, together with the largest possible
associated step curvature $\chi$.

If $R=w^k$ is  proper power, then each place occurs $k$ times on
$R$. So a place $\bQ$ could occur several times in
$\OneStep(\bP)$, 
corresponding to different positions on $R$ relative to $\bP$.
To differentiate them, we store the number of
letters of $R$ between $\bP$ and $\bQ$ for each item on the list.

\begin{defn}\label{def:one_step}
Let $\bP$ be a place with location $R(i, a, b)$. A place $\bQ$
is \emph{one-step reachable} at distance $l$ from $\bP$, where
$1 \le l < |R|$, if the following hold:
\begin{enumerate}
\item[(i)] $\bQ$ has location $R(j, s, t)$ for some 
  $s, t \in X$, where $j = i+l$ (interpreted cyclically).  
\item[(ii)]  If $\bP$ is red, then $l = 1$  (and so $s = b$). 
\item[(iii)] If $\bP$ is green, then exactly one of the following occurs:
\begin{enumerate} 
\item there exists a green face $f'$ instantiating $\bP$, and a
consolidated edge between $f$ and $f'$ of length $l$ from the location
   of $\bP$ to that of $\bQ$, and $\bQ$ is green; 
\item there exists an \emph{intermediate place} $\bP'$ 
whose location is $R(j-1, u, s)$ and whose colour is red, there is a
green face $f'$ instantiating $\bP$ such that there is a consolidated edge
between $f$ and $f'$ of length $l-1$ between the locations of $\bP$ and
$\bP'$, and there is red edge between $\bP'$ and $\bQ$.
\end{enumerate}
\end{enumerate}
\end{defn}

\begin{figure}
\begin{center}
\begin{picture}(300,80)(-20,-5)
\thinlines
\put(10,30){\circle*{3}}
\put(61,47){\circle*{3}}
\put(10,30){\line(3,1){51}}
\put(37,39){\vector(3,1){0}}
\put(-20,15){\line(2,1){30}}
\put(-4,23){\vector(2,1){0}}
\put(10,30){\line(1,2){16}}
\put(19,48){\vector(1,2){0}}
\put(61,47){\line(4,1){30}}
\put(76,51){\vector(4,1){0}}
\put(61,47){\line(1,2){15}}
\put(70.5,66){\vector(1,2){0}}
\put(61,47){\line(-2,3){15}}
\put(55,56){\vector(2,-3){0}}
\put(-4,17){$a$}
\put(37,31){$b$}
\put(21,45){$c$}
\put(74,43){$d$}
\put(73,62){$x$}
\put(43,56){$y$}
\put(32,48){$f_1$}
\put(45,20){$f$}
\put(84,65){$f_2$}
\put(9,21){\footnotesize \bf P}
\put(59,38){\footnotesize \bf Q}
\put(25,-10){(a)}

\put(180,30){\circle*{3}}
\put(249,54){\circle*{3}}
\put(180,30){\line(3,1){27}}
\put(195,35){\vector(3,1){0}}
\put(150,15){\line(2,1){30}}
\put(166,23){\vector(2,1){0}}
\put(180,30){\line(1,2){15}}
\put(189,48){\vector(1,2){0}}
\put(210,41){\line(3,1){2}}
\put(214,42.33){\line(3,1){2}}
\put(218,43.67){\line(3,1){2}}
\put(222,45){\line(3,1){27}}
\put(240,51){\vector(3,1){0}}
\put(249,54){\line(4,1){30}}
\put(265,58){\vector(4,1){0}}
\put(249,54){\line(1,2){15}}
\put(258.5,73){\vector(1,2){0}}
\put(249,54){\line(-2,3){15}}
\put(243,63){\vector(2,-3){0}}
\put(166,17){$a$}
\put(194,26){$b$}
\put(191,45){$c$}
\put(236,42){$d$}
\put(262,50){$e$}
\put(261,70){$x$}
\put(232,61){$y$}
\put(212,54){$f_1$}
\put(240,25){$f$}
\put(272,73){$f_2$}
\put(179,21){\footnotesize \bf P}
\put(248,45){\footnotesize $\bP'$}
\put(221,-10){(b)}
\end{picture}
\begin{picture}(180,120)(0,-5)
\thinlines
\put(30,30){\circle*{3}}
\put(99,54){\circle*{3}}
\put(147,66){\circle*{3}}
\put(30,30){\line(3,1){27}}
\put(45,35){\vector(3,1){0}}
\put(0,15){\line(2,1){30}}
\put(16,23){\vector(2,1){0}}
\put(30,30){\line(1,2){15}}
\put(39,48){\vector(1,2){0}}
\put(60,41){\line(3,1){2}}
\put(64,42.33){\line(3,1){2}}
\put(68,43.67){\line(3,1){2}}
\put(72,45){\line(3,1){27}}
\put(90,51){\vector(3,1){0}}
\put(99,54){\line(4,1){48}}
\put(127,61){\vector(4,1){0}}
\put(99,54){\line(1,2){15}}
\put(108.5,73){\vector(1,2){0}}
\put(99,54){\line(-2,3){15}}
\put(93,63){\vector(2,-3){0}}
\put(147,66){\line(1,0){30}}
\put(165,66){\vector(1,0){0}}
\put(147,66){\line(-1,4){6}}
\put(145,76){\vector(1,-4){0}}
\put(147,66){\line(1,1){20}}
\put(160,79){\vector(1,1){0}}
\put(16,17){$a$}
\put(44,26){$b$}
\put(41,45){$c$}
\put(86,42){$d$}
\put(125,53){$e$}
\put(111,70){$x$}
\put(82,61){$y$}
\put(64,53){$f_1$}
\put(90,25){$f$}
\put(122,73){$f_2$}
\put(168,73){$f_3$}
\put(29,21){\footnotesize \bf P}
\put(98,45){\footnotesize $\bP'$}
\put(146,57){\footnotesize \bf Q}
\put(70,-10){(c)}
\end{picture}
\end{center}
\caption{One-step reachable places}
\label{fig:onestep1}
\end{figure}
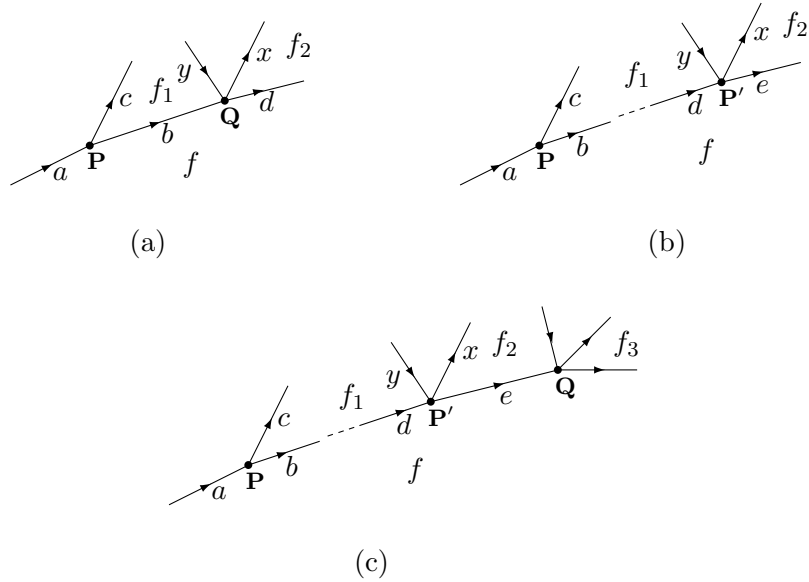

For each place $\bP = (R(i,a,b),c,C)$ on $R$, we compute the list
$\OneStep(\bP)$ as follows. In the description below, by 
\emph{including} an item $(\bQ,l,\chi)$ in $\OneStep(\bP)$, we mean
append it to the list if there is no entry of the form $(\bQ,l,\chi')$
already or, if there is such an entry with $\chi > \chi'$, then replace
that entry with $(\bQ,l,\chi)$. (If there is such an entry with $\chi
\leq \chi'$, we do nothing). 

\begin{alg}\label{alg:one_step} \texttt{ComputeOneStep($\bP = (R(i, a, b),
    c, C)$)}:

\begin{mylist2}
\item[Step 1] Initialise $\OneStep(\bP)$ as an empty list. 

\item[Step 2] \begin{description}
\item[Case $C = \mathbf{\sfR}$.] 
For each place $\bQ = (R(i+1, b, d), x, C')$, and
for each $y \in X$ such that $y$ intermults with $b^{\sigma}$,
proceed as follows.
(See Figure \ref{fig:onestep1}\,(a).)

Let $\chi_1 := \Blob(y, b^{\sigma}, c)$, and let
$\chi_2 := \Vertex((y, b^{\sigma}, \sfR), (b, d,\sfG), (d^{\sigma},x,C'))$. 
Include $(\bQ,1,\chi_1+\chi_2)$ in $\OneStep(\bP)$.

\item[Case $C = \mathbf{\sfG}$.] 
For each location $R'(k, b^{\sigma}, c)$ instantiating $\bP$, proceed
as follows.

For each  place $\bP' = (R(j,d,e),x,C')$ on $R$ that
can be reached from $\bP$ by a single (not necessarily maximal) consolidated
edge $\alpha$ between $R$ and $R'$, let $\nu_1 := (y,d^{\sigma},\sfG)$ be the
green $\cG$-vertex corresponding to the location on $R'$ at the end of
$\alpha$, and let $l := \ell(\alpha)$. 
For each out-neighbour $\nu_2 :=(e^{\sigma},x,C')$ of the $\cG$-vertex
$\nu:= (d,e,\sfG)$, compute $\chi' := \Vertex(\nu_1, \nu, \nu_2)$.
\begin{mylist}
\item[(i)] If $\bP'$ is green then include $(\bP',l,\chi')$ in $\OneStep(\bP)$.
(See Figure \ref{fig:onestep1}\,(b).)
\item[(ii)] If $\bP'$ is red then  $\bP'$ is the intermediate
place of the step. 
Find all places $\bQ$ 
that are one letter further along $R$ than $\bP'$.
(See Figure \ref{fig:onestep1}\,(c).)
Just as in Case \sfR, compute the combined maximum
curvature $\chi''$ returned by the red blob between $\bP'$
and $\bQ$ and the vertex at $\bQ$. Include
$(\bQ,l+1,\chi'+\chi'')$ in $\OneStep(\bP)$.
\end{mylist}
\end{description}
\end{mylist2}
\end{alg}

\begin{lemma}\label{lem:step_bound} 
Let $\bP = (R(i, a, b), c, C)$ and $\bQ = (R(j, d, e), c', C')$ 
be places on the same relator $R$. Then the following are
equivalent.
\begin{enumerate}
\item[(i)] 
The place $\bQ$ is one-step reachable from
$\bP$ at distance $l$.
\item[(ii)] There exists a coloured decomposition of a cyclic
conjugate $R'$ of $R$ such that a subword $w_k$ or $w_k w_{k+1}$ of $R'$ between a 
location of $\bP$ and a location of $\bQ$ is a
step of length $l$, the face adjacent to $f$ at  $w_k$ has colour $C$,
and edge after $w_k^{-1}$ labelled $c$, and the face adjacent to $f$ at the letter after
the end of the step has colour $C'$ and next letter $c'$. 
\item[(iii)] There exists $\chi$ such that $(\bQ, l, \chi) \in
\OneStep(\bP)$.
\end{enumerate}
Furthermore, if $(\bQ, l, \chi) \in \OneStep(\bP)$ then
$\chi$ is an upper bound for the step curvature, and $\chi \leq
-1/6$. 
\end{lemma}
\begin{proof} 
It follows from the definitions that (i) and (ii) are equivalent: the
only thing for the reader to check is that enough conditions have been
placed in (ii) to uniquely identify the place $\bQ$ at distance $l$
from $\bP$.

It is also clear that the $\OneStep$ algorithm finds all one-step reachable
places, and does not find any places that are not one-step reachable,
so the equivalence of (i) and (iii) follows.

The step curvature is the sum of the curvature given to $f$ by at most two vertices
and at most one blob. Hence the fact that $\chi$ is an upper bound on the step curvature is
immediate from the fact that the \Vertex
and \Blob functions return an upper boud on $\chi(v, f, \Gamma)$ and
$\chi(B, f, \Gamma)$ (see
Lemmas~\ref{lem:vertex_correct} and \ref{lem:blob_correct}). 

When $C = \sfR$,  or when $C = \sfG$ and $\bP' \neq \bQ$ so that $\bP'$
is red, the claim  that
$\chi \leq -1/6$
follows from $\Blob(y, b, c) \le \chi(B, f, \Gamma) \leq -1/6$ for all
$y,b,c, B$.
Otherwise $C = \sfG$ and $\bP' = \bQ$ is green, and the vertex $v$ at any
instantiation of $\bP'$ in a diagram $\Gamma \in \calD$
has three consecutive incident green faces. These are encoded by three
green $\cG$-vertices
$\nu_1$, $\nu$ and  $\nu_2$ which form a directed $\cG$-path, and so $\chi(v, f, \Gamma) \leq
\Vertex(\nu_1, \nu, \nu_2) \leq -1/6$. 
\end{proof}

\subsection{The main $\RSymVerify$ procedure}\label{subsec:tester}

The user chooses a value of $\varepsilon > 0$ to test, and we must check 
whether the steps computed in the lists $\OneStep(\bP)$ 
could be combined around
a relator to leave it with more than $-\varepsilon$  of curvature.
In this subsection, we shall present the main  procedure $\RSymVerify(\cP,
\varepsilon)$
which carries out these checks and 
prove that it works. 

After computing the data and functions from the previous subsections, 
$\RSymVerify(\cP, \varepsilon)$ runs a sub-procedure, 
$\RSymVerifyAtPlace(\bP_s, \varepsilon)$, at each start place $\bP_s$ on each relator
$R\in \cR$ in turn. If every call to $\RSymVerifyAtPlace$ returns \true, 
then $\RSymVerify$ returns
\true, but if any fail, then it aborts and returns $\fail$.

$\RSymVerifyAtPlace(\bP_s, \varepsilon)$ creates a list $L$, whose entries are quadruples
$(\bQ,l,k,\psi)$. The first three components represent a place $\bQ$ at
distance $l$ from $\bP_s$ along $R$ that can be reached from $\bP_s$ in $k$
steps. The final component $\psi$ is equal to
$\left((1+\varepsilon)l/|R|\right)+\chi$, 
where $\chi$ is the largest possible total
curvature arising from these $k$ steps.
If $\psi \le 0$ then we
are on track for a final curvature of most $-\varepsilon$, 
whereas if $\psi>0$ then we are not. We shall show in the proof of
Theorem~\ref{thm:rsym} that there is no
need to keep a record of situations in which $\psi<0$. In other words, we may
assume that if the test fails then $\psi \ge 0$ after each step in the failing
decomposition.

By Lemma \ref{lem:step_bound}, the largest possible step curvature
is $-1/6$, so the cumulative 
curvature after $\lceil 6(1 + \varepsilon) \rceil$
steps is at most $-\varepsilon$, and we only need consider
$\lceil 6(1 + \varepsilon) \rceil - 1$ steps from $\bP_s$.
There can also be at most $r$ steps, where $r$
is the length of the longest relator in $\cR$, so we define
$$\zeta := \min(\lceil 6(1 + \varepsilon) \rceil - 1,r),$$
and use $\zeta$ as an upper bound on the number of steps. Notice that if
$\varepsilon < 1$ then $\zeta \leq 11$: the default value of
$\varepsilon$ in our implementations (see
Section~\ref{sec:implementation}) is $1/10$. 

Similarly to  Algorithm~\ref{alg:one_step}, by \emph{including} an entry
$(\bQ, l+l', i, \phi')$ in a list $L$, we mean appending it to $L$ if
there is no entry $(\bQ, l+l', j, \phi'')$ 
in $L$ 
or, if there is such an entry with $\phi' >
\phi''$, then replacing it by $(\bQ, l+l', i, \phi')$. 
\medskip

\begin{proc}\label{proc:RSymVerifyAtPlace}
 $\RSymVerifyAtPlace(\bP_s = (R(i, a, b), c, C), \varepsilon)$: 
\vspace{-1mm}
\begin{mylist2}
\item[Step 1] Initialise $L := [(\bP_s,0,0,0)]$.
\item[Step 2] For $i :=1$ to $\zeta$ do:
\begin{enumerate}
\item[] For each $(\bP,l,k,\psi) \in L$ with $k = i-1$, and for
each $(\bQ,l',\chi) \in \OneStep(\bP)$ with $l+l' \le |R|$, do:
\begin{mylist}
\item[(i)]  Let $\psi' := \psi + \chi + (1+\varepsilon)l'/|R|$. 
\item[(ii)] If $\psi' < 0$, or if $l + l' = |R|$ and $\bQ \neq \bP_s$, then do
nothing;
\item[(iii)] else if $\psi'>0$, $\bQ=\bP_s$ and $l+l'=|R|$, then return $\fail$
and $L$;
\item[(iv)] else include $(\bQ,l+l',i,\psi')$ in $L$.
\end{mylist}
\end{enumerate}
\item[Step 3] Return $\true$.
\end{mylist2}
\end{proc}

We now make a couple of remarks on
Procedure~\ref{proc:RSymVerifyAtPlace}. 
First, notice that if on the $i$th iteration of Step 2 there are no places
$\bQ$ that can be reached from any $\bP$
with non-positive curvature, then $\RSymVerifyAtPlace$ stops early. 
As we shall see in the proof of Theorem \ref{thm:rsym},
Lemma \ref{lem:GUSU} implies that, if there is a decomposition of a
cyclic conjugate of $R$ that leads to failure of  \RSymVerifyAtPlace,
then there is a start place from which
each intermediate place can be reached with non-negative
curvature. 

Another observation is that, in Step 2 (iv), a list entry
$(\bQ,l+l',k,\psi'')$ can be replaced by $(\bQ,l+l',i,\psi')$ with $i > k$.
If there is a failing decomposition of $R'$ involving the original
entry, then there must be a (possibly longer) failing decomposition involving
the new entry. 

We are now able to summarise the overall procedure, $\RSymVerify$. 

\begin{proc}\label{proc:RSymVerify} $\RSymVerify(\cP = \langle
  X^\sigma \mid V_P \mid 
  \cR\rangle,
 \varepsilon)$:
\begin{mylist2}
\item[Step 1]
For all $a, b \in X$, use $V_P$ and $\sigma$ to test whether $a$ and $b$ intermult: Subsection~\ref{subsec:locplace}. 
Store the list of intermult pairs.
\item[Step 2]
Express each relator $R \in \cR^{\pm}$ as a power $w^k$  and find all locations on $w$:
Subsection \ref{subsec:locplace}. Store the list of locations. 
\item[Step 3]
For each location $R(i, a, b)$, and for each choice of extra letter
$x\in X$  and colour $C \in
\{\sfG, \sfR\}$,
test whether the potential place $(R(i, a, b), x, C)$ is instantiable:
Subsection \ref{subsec:locplace}. Store a list of all places, and for
each green place store the list of locations $R'(k, b^\sigma, c)$  which instantiate it. 
\item[Step 4]
Use the list of places and the list of intermult pairs to 
compute the vertex graph $\cG$. Store $\cG$ and the lists of locations corresponding to each green
$\cG$-vertex: Subsection \ref{subsec:vertex}.
\item[Step 5]
For each pair of $\cG$-vertices $\nu_2$ and $\nu_1$, find the
minimal weight of a  non-trivial $\cG$-path from $\nu_2$ to $\nu_1$. Hence create  the \Vertex function:
Algorithm~\ref{alg:vertex}. 
\item[Step 6] 
Identify $\cR$-letters and compute the list $\mathcal{B}$, to create
the \Blob function:
Algorithm~\ref{alg:blob}.
\item[Step 7]
For each relator $R \in \cR$, and each place $\bP$ on $R$, use the \Vertex and \Blob functions to run
$\texttt{ComputeOneStep}(\bP)$: Algorithm~\ref{alg:one_step}. Store the list
$\OneStep(\bP)$, for each such place $\bP$.
\item[Step 8]
For each relator $R \in \cR$, and each place $\bP_s$ on $R$, run $\RSymVerifyAtPlace(\bP_s, \varepsilon)$.
 If $\RSymVerifyAtPlace(\bP_s,
\varepsilon)$ ever returns \texttt{fail} and a list $L$, then return \texttt{fail}
and $L$, otherwise do nothing. 
\item[Step 9] Return \texttt{true}. 
\end{mylist2}
\end{proc}

\begin{thm}\label{thm:rsym}
Let $\cP = \langle X^\sigma  \ | \ V_P \ | \ \cR \rangle$ be a pregroup
presentation, such that
$\cI(R) = R$ for all $R \in \cR$,  
 and let $\varepsilon > 0$. 
If $\RSymVerify(\cP, \varepsilon)$ 
returns \true then $\RSym$ succeeds on $\cP$
with constant $\varepsilon$.
\end{thm}

Before proving Theorem~\ref{thm:rsym} we prove a useful combinatorial
lemma.

\begin{lemma}
\label{lem:GUSU}
Let $\ell \in \Z_{>0}$, let $L := \{1,2,\ldots,\ell\}$, 
and let $a_1, a_2, \ldots, a_\ell \in \R$. For
$m \in \Z$, denote by $\mybar m$ the element of $L$ with $m \equiv \mybar m
\pmod \ell$.
If\/ $S := \sum_{m\in L} a_m \ge 0$ then there exists $j \in L$
such that for all $i \in \Z_{>0}$ the partial sum
\[ s_{j,i} := \sum_{m=0}^{i-1} a_{\mybar{j+m}} \ge 0. \]
\end{lemma}

\begin{proof} If $S=s_{1,\ell}$ is the minimum of $\{s_{1,i} \mid i \in L\}$,
then all partial sums starting at $a_1$ are positive, so we can set $j
= 1$. 
Otherwise, choose $j \in L\setminus\{1\}$ such that 
$s_{1,j-1} \le s_{1,i}$ for all $i \in L$. 

Notice that $s_{a,b} + s_{a+b,c} = s_{a,b+c}$ for all
$a,b,c \in \Z_{>0}$. 
Thus $s_{j,i} = s_{1,j+i-1}-s_{1,j-1}$ for all $i \in \Z_{>0}$. 
However, $s_{1,j+i-1} = s_{1,\mybar{j+i-1}} + kS$ for some $k \in \Z_{\ge 0}$.
So in any case, $s_{j,i} = kS + s_{1,\mybar{j+i-1}} - s_{1,j-1}\ge 0$ 
by the choice of $j$.
\end{proof}

\medskip

\noindent {\sc Proof of Theorem~\ref{thm:rsym}}
Suppose that  $\RSym$ does not succeed on $\cP$
with the constant $\varepsilon$.
Then there exists a diagram $\Gamma \in \calD$
over $\cP$, and a face $f \in \Gamma$, such that $f$ is green, 
has no boundary edges, and satisfies $\kappa_{\Gamma}(f) > -\varepsilon$,
where $\kappa_{\Gamma} = \RSym(\Gamma)$. We shall show that
$\RSymVerifyAtPlace(\bP_s, \varepsilon)$  returns \fail
for some $\bP_s$ on $f$. 

Let the label on $\partial(f)$ be $R  \in \cR^{\pm}$. 
If $R \not\in \cR$, then the corresponding face $f'$ in the diagram where all
faces have labels the inverses of the labels of those in $\Gamma$ will also
satisfy $\kappa_{\Gamma}(f')> -\varepsilon$, so assume without loss of
generality that $R \in \cR$. 

As discussed in Subsection \ref{subsec:locplace}, for some cyclic conjugate
$R'$ of $R$, we have $R' = w_1 w_2 \cdots w_k$, where each $w_i$ labels
a consolidated edge $e_i$ in $\Gamma$, and each $w_i$ has an associated colour
$C_i \in \{\sfG,\sfR\}$ describing the colour of the other face incident
with $e_i$
in $\Gamma$.  Recall that we do not allow the combination $C_1=\sfR$,
$C_k=\sfG$.  From this we derive a decomposition $R' = v_1v_2 \cdots v_\ell$,
where each $v_i$ labels a step and is equal either to a single $w_j$ or
to some $w_jw_{j+1}$ with $w_j$ green and $w_{j+1}$ red. 

Let the step curvature given to $f$ by the step corresponding to
$v_i$ be $\chi_i$, let  
$l_i$ be the number of letters in $v_i$, let
$\lambda_i = (1+ \varepsilon)l_i/|R|$, and let $a_i = \chi_i +
\lambda_i$. 
Then $\kappa_{\Gamma}(f) = 1 + \sum_{i = 1}^\ell
\chi_i > - \varepsilon$, so $\sum_{i = 1}^\ell a_i > 0$. 

By Lemma~\ref{lem:GUSU} there exists an $i$ such that the partial sums
\[a_i, a_i + a_{i+1}, \ldots, a_i + \cdots+  a_\ell + a_1 + \cdots + a_{i-1}\]
are all non-negative.
Replace $R'$ if necessary by its cyclic conjugate
$R'' := v_iv_{i+1}\cdots v_{i-1}$, and observe that the
steps induced by the corresponding decomposition into consolidated edges
are the same subwords $v_i$ as before.

Let $\bP_s$ be the place on $R$ at the beginning of $R''$
(which is instantiable because $\Gamma \in \calD$). We showed in
Lemma~\ref{lem:step_bound} that each step $v_i$ corresponds to a pair
$\bP, \bQ$ of places on $R$, and that there exist $l$ and $\chi$ such
that $(\bQ, l, \chi) \in \OneStep(\bP)$, with $\chi \geq \chi_i$.
$\RSymVerifyAtPlace(\bP_s, \varepsilon)$ uses $\chi$ in place of
$\chi_i$, 
so the partial sums
calculated for each place are greater than or equal to the
actual curvature sums, and in particular are all non-negative.
Thus $\RSymVerifyAtPlace(\bP_s, \varepsilon)$ returns \fail.
\hfill $\Box$

\subsection{Complexity of \RSymVerify}\label{subsec:effic}
We now show that \RSymVerify runs in  time
polynomial in $|X|$, $|\cR|$ and $r$, 
where $r := \max\{|R|: R \in \cR\}$ is the length of the longest relator.

Recall that we assume that \emph{before} running \RSymVerify, 
the presentation has been simplified and the pregroup has been defined:
see Subsection \ref{subsec:preprocess}. 
The presentation simplification process involves comparing subwords
of cyclic conjugates of the relators, and any simplification reduces the total
length of the presentation, so it is clear that this can be carried out in
polynomial time.

\begin{thm}\label{thm:r_sym_complexity}
$\RSymVerify(\cP =  \langle X^\sigma  \ | \ V_P \ | \ \cR
  \rangle, \varepsilon)$ runs in time
$O(|X|^5 + r^3|X|^4|\cR|^2)$.
\end{thm}

\begin{proof}
We shall work through Steps 1 to 8 of Procedure~\ref{proc:RSymVerify}, 
bounding the time complexity of  each step. We are not
attempting to find the optimal bounds, simply to show that the process runs in
low-degree polynomial time. We assume that products and inverses in the
pregroup can be computed in constant time.

In Step 1, we compute an $X \times X$ boolean array
describing the set of all intermult pairs $(a, b) \in X^2$.  For each $a \in X$, and
for each $b \in X \setminus \{a^{\sigma}\}$, we must check whether $(a, b) \in
D(P)$ and, if not, whether there exists an $x \in X$ such that $(a, x) \in
D(P)$ and $(x^{\sigma}, b) \in D(P)$. This
can be done in time $O(|X|^3)$.

In Step 2, 
 for each $R \in \cR$ we first find
$w$ that maximises the value of $k$ for which $R = w^k$.
For $2\leq l \leq |R|/2$, we let $w$ be the length
$l$ prefix of  $R$,  and
test whether $w^{|R|/l} = R$, in total time
$O(r^2|\cR|)$. 
There are at most $2r|\cR|$ locations defined by $\cR^{\pm}$. 
When compiling the list of locations, we record which pairs
of locations are mutually inverse, in the sense that they describe
corresponding positions in inverse pairs of relators.

In Step 3, we find the $O(r|X||\cR|)$ places $\bP = (R(i, a, b), c, C)$ with
$R \in \cR^{\pm}$. 
To do so, for each of the $O(r|R|)$ locations $R(i, a, b)$ we first
find all $c$ such that $(b^{\sigma}, c)$ is an intermult pair, and
hence all instantiable red places $\bP = (R(i, a, b), c, \sfR)$, in time $O(|X|)$.
Then, for each location $R'(j, b^{\sigma}, c)$ there is
a  green place $(R(i, a, b), c, \sfG)$
if and only if the locations $R(i,a,b)$ and
$R'(j,b^{\sigma},c)$ are not mutual inverses. We computed the
inverse pairs of locations in Step 2.
So  Step 3  requires time
$O(r|\cR|(|X| + r|\cR|)) = O(r^2|X||\cR|^2)$.

In Step 4, we compute the vertex graph $\cG$. It has at most
$2|X|^2$
$\cG$-vertices. We can find these, and also list the locations
corresponding to each green $\cG$-vertex, in time  $O(r|\cR| + |X|^2)$.
Let $\nu = (a,b,\sfG)$ be a green $\cG$-vertex. There is a
$\cG$-edge from $\nu$ to the $\cG$-vertex $\nu_1 = (c, d, \sfG)$ if and only if
(i) $c = b^{\sigma}$; and (ii) if $d = a^{\sigma}$ then there is more than one
location corresponding to $\nu_1$.
(For Condition (ii), note that there is at least
one such location, since there is one coming from the corresponding position
of the inverse $R'$ of each relator $R$ associated with $\nu$.
But if there was only one, then the only diagram in which faces labelled $R$
and $R'$ shared the edge labelled $b$ would fail to be $\sigma$-reduced.
So Condition (ii) ensures that there exists a
$\sigma$-reduced diagram instantiating this $\cG$-edge.)
These two conditions can be
tested in constant time for each $\cG$-vertex $\nu_1$.
There is a $\cG$-edge from $\nu$ to each $\cG$-vertex
$(b^{\sigma}, c,\sfR)$ and one from each $\cG$-vertex $(c, a^{\sigma},\sfR)$ to $\nu$.
We can define all of these edges from and to $\nu$ in time $O(|X|)$. There are no edges
between red $\cG$-vertices. So each $\cG$-vertex has $\cG$-degree $O(|X|)$,
and we can find the $O(|X|^3)$ $\cG$-edges and assign
their weights in time $O(|X|^4)$. The time complexity of Step 4
is $O(r|\cR| + |X|^4)$. 

In Step 5, we compute the values of $\Vertex(\nu_1, \nu, \nu_2)$. 
We begin by using the Johnson--Dijkstra
algorithm \cite{Johnson} to find the smallest weights of paths between
all pairs of $\cG$-vertices in the vertex graph. This algorithm runs in
time $O(|V|^2 \log |V|+ |V| |E|)$, on a graph with $|V|$ vertices and $|E|$
edges, so $O(|X|^5)$ in our case, and returns a matrix of path weights. 
We then consider each of the $O(|X|^2)$ green $\cG$-vertices $\nu$ in turn, and for
each of the $O(|X|^2)$ 
directed $\cG$-paths
$\nu_1, \nu, \nu_2$, we record the appropriate curvature value for
Algorithm~\ref{alg:vertex}.  The total time complexity of Step 5  is $O(|X|^5)$. 

In Step 6, we compute the values of the $\Blob$ function.
We first identify the set of $\cR$-letters, in time $O(r|\cR|)$, and store
this information as a boolean array.
We next  use the intermult table to
construct all of the $O(|X|^5)$ words of length $l$ between $3$ and $5$ such
that each cyclically consecutive pair of letters intermult, and that
include at most one non-$\cR$-letter when they have length $3$ or $4$,
and none otherwise. We then discard all such words $w$ that have length $3$ or
$4$ and are not equal to $1$ in $U(P)$, or have length $5$ and are not
equal in $U(P)$ to some $a \in P$. Finally, if $l = 5$, and $w =_{U(P)} a \neq 1$,
then we check that $a^{\sigma}$ is an $\cR$-letter, and that all cyclic subwords
of $wa^{\sigma}$ of length $2$ or $3$ are not equal to $1$ in $U(P)$. 
The checks on each word take constant time, so the time complexity of 
Step 6 is $O(|X|^5)$.

In Step 7, we compute $\OneStep(\bP)$, for each of the $O(r|X||\cR|)$ places $\bP$ in turn.
If  $\bP$ is red then there are $O(|X|)$ 1-step reachable places
$\bQ$ and, for each of them we make $O(|X|)$ calls to both \Vertex and \Blob
to find the maximum step curvature.
If $\bP$ is green, then there are $O(r|X|)$ 1-step reachable places 
 $(\bQ,l)$.
To find them, we look up all $O(r|\cR|)$ locations
for the second face $f_1$ that instantiates $\bP$,  and for each of them
we find the length $l_1$ of the maximal consolidated edge between $f$ and
$f_1$ in time $O(r)$.
For each such $f_1$,  there are $O(l|X|) = O(r|X|)$ possibilities for the
place $\bP'$ at the end of the consolidated edge.
If $\bP'$ is green, then $\bP'=\bQ$, and we include its step curvature with a single
call to \Vertex.
If $\bP'$ is red, then we need a further $O(|X|^2)$ calls to \Vertex and
\Blob to find all possible triples  $(\bQ,l, \chi)$ at the end of the step.
So the time complexity of Step 7 is $O(r^3|X|^4|\cR|^2)$.

Step 8 runs $\RSymVerifyAtPlace$ at each of the 
$O(r|X||\cR|)$ places. The length of
the list $L$ constructed by $\RSymVerifyAtPlace$
 is $O(r|X|)$. Each item on $L$ is considered at most
$\zeta \le r$ times, so 
the time complexity of Step 8 is $O(r^3|X|^2|\cR|)$.
\end{proof}

\subsection{$\RSym$ with interleaving}\label{subsec:interleave}

In the previous subsections, we described a procedure,
$\RSymVerify(\cP, \varepsilon)$, that
checks whether $\RSym$
succeeds on a pregroup presentation $\cP$ under the assumption that
$\cI(\cR) = \cR$.
We now describe the modifications that we make when this assumption
does not hold, defining a more general procedure $\RSymIntVerify(\cP,
\varepsilon)$. It follows the same overall steps as \RSymVerify, and the
reader may wish to refer to Procedure~\ref{proc:RSymVerify} for these
steps. After presenting the key ideas, we describe \RSymIntVerify
at the end of this
subsection: see Procedure~\ref{proc:RSymIntVerify}. 

We remind the reader that all terms, notation and procedures are
listed in the Appendix.

The set $\cI(\cR)$ is potentially of exponential
size, since it might be possible to interleave between each pair of
letters of each $R \in \cR$ (for an example of this, let $R$ be any
cyclically $P$-reduced word in $U(P)$, where $P$ is the pregroup for a
free product with amalgamation given in 
 Example~\ref{ex:interleave}).
Despite this, we shall see that $\RSymIntVerify$ runs in
polynomial time.

The overall strategy of \RSymIntVerify is the same as that of \RSymVerify: we consider each
relator $R \in \cR$ in turn, and look for ways to decompose each
cyclic conjugate
$R'$ of each element of $\cI(R)$ into words $w_1w_2\cdots w_k$ that maximise the
curvature received by $R'$ in Steps 3 and 4 of 
\ComputeRSym. Each $w_i$ is an
interleave of $[s_{i-1}^{\sigma}w'_{i1}] w'_{i2} \cdots [w'_{il_i} s_i]$ for some $s_{i-1}, s_i \in P$, where
$w_i' = w'_{i1} \cdots w'_{il_i}$ is the corresponding subword of the corresponding cyclic conjugate of $R$. It follows from Lemma~\ref{lem:rewritetwice} that
all elements of $\cI(R)$ have a description of this form.

In the preprocessing stage (see Subsection~\ref{subsec:preprocess}), we first carry
out Preprocessing Steps 1 to 3. Once the pregroup has been chosen,
if $\cI(\cR) \neq \cR$ then in Preprocessing Step 4 we  find all
$R_1,R_2 \in \cR$ for which there exist distinct cyclic conjugates $S_1$,
$S_2$ in $\cI(R_1^{\pm})$, $\cI(R_2^{\pm})$ with $S_1=ww_1$,
$S_2=ww_2$ and $|w|>|w_1|$.
Each such common prefix $w$ is equal in $U(P)$ to $s^{\sigma}w't$ for some
$s,t \in P$, where $w'$ has length $|w|$ and
is a prefix of a cyclic conjugate of $R_1$ or $R_1^{-1}$.
 We can solve the word problem in $U(P)$ in
linear time by Corollary~\ref{cor:wpuplinear}, 
so we can find all such $S_1,S_2$ in polynomial time.

As in preprocessing Step 2, if we find such a pair and 
$R_1 \ne R_2$, then we replace $R_2$ by $w_1^{-1}w_2$.
It is conceivable that $R_1=R_2$, $S_1 \ne S_2$ and $|w_1|=|w_2|=1$.
In that case we use the implied length two relator $w_1^{-1}w_2$ to 
adjust the pregroup.  So again the
simplification process ensures that the hypothesis, and hence also the
conclusion, of Theorem~\ref{thm:rsymsucceeds} holds.

\begin{defn}\label{def:int_set}
For each $(a, b) \in X \times X$ the \emph{interleave set}, denoted 
$\cI(a, b)$, consists of all $s \in P$ such that $(a, s), (s^{\sigma}, b)
\in D(P)$. We explicitly permit $s = 1$, so that no interleave set is
empty.
\end{defn}

\begin{example}\label{ex:int_set}
Let $P_1$ be the pregroup for a free product $G \ast H$, as in the
first part of
Example~\ref{ex:free_prod}. If $g, h \in G$ then $\cI(g, h) = G$,
whilst if $g \in G$ and $h \in H$ then $\cI(g, h) = 1$.

Let $P_2$ be the pregroup for a free product with amalgamation, as in
the second part of Example~\ref{ex:free_prod}. Then $\cI(a, b) = G
\cap H$, for
all $a \in G \setminus H$ and all $b \in H \setminus G$. 
\end{example}

\begin{defn}\label{def:dec_location}
Let $R \in \cR^{\pm}$. A \emph{decorated location} on $R$ is a
4-tuple $(i, a, b, s)$, denoted $R(i, a, b, s)$, where $R(i, a, b)$ is a location, and
$s \in \cI(a, b)$.  Let the subword of the cyclic word $R$ containing the
location $R(i,a,b)$ be $dabe$ with $d,e \in X$.
The \emph{pre-interleave set} $\Pre(R(i))$ of $R(i, a, b, s)$ is $\cI(d, a)$,
and the \emph{post-interleave set} $\Post(R(i))$ is $\cI(b, e)$.
\end{defn}

We now generalise Definition~\ref{def:place} to cover non-trivial
interleaves. Recall that $[a_1a_2 \cdots a_n]$, with
$a_i \in P$, denotes the element of $P$ that is the product in $U(P)$ of the
$a_i$. 

\begin{defn}\label{def:dec_place}
A \emph{potential decorated place} $\mathbf{P}$ is a triple
$(R(i, a, b, s), c, C)$, where $R(i, a, b, s)$ is a decorated location,
$c \in X$,  and $C \in \{\sfG, \sfR\}$.
A potential decorated place is a \emph{decorated place} if it is
instantiable in the following sense:
\begin{enumerate}
\item There exists a $\sigma$-reduced 
and semi-$P$-reduced diagram $\Gamma$ over $\cI(\cP)$ with a face $f$ labelled by an
element of $\cI(R)$, elements $t \in \Pre(R(i))$ and $u \in
\Post(R(i))$,  a 
face $f_2$ meeting $f$ at an edge labelled $[s^{\sigma} bu] \in X$,
 and the vertex between the edges of $f$ labelled $[t^{\sigma} a s] \in X$
 and $[s^{\sigma} bu]$ has degree at least three;
\item the half-edge on $f_2$ after
  $[s^{\sigma}bu]^{\sigma}$ is labelled $c$;
\item  if $C = \sfG$ then $f_2$ is green, and if $C = \sfR$, then $f_2$
is a red blob.
\end{enumerate}
\end{defn}
Notice in particular that we require $([s^\sigma b], u), ([t^{\sigma}
a], s) \in D(P)$.

We determine whether a potential decorated place
$\bP = (R(i, a, b,
s), c, C)$ is instantiable as
follows.  
If $C = \sfR$,  then $\bP$ is a decorated place if and only if there exists 
$u \in \Post(R(i))$ such that $[s^{\sigma}bu]^{\sigma}$ is an
element of $X$ that intermults with $c$.
When $C = \sfG$, for each decorated location $R'(j, d, e, v)$,
 we first check  whether
there exists $x \in \Post(R'(j))$ such that $[v^{\sigma}ex] =_P  c$. Then, for each
$u \in \Post(R(i))$ we check whether there exists $y \in \Pre(R'(j))$ such
that $[y^{\sigma}dv] =_P  [s^{\sigma}bu]^{\sigma}$. Finally, we check whether the
resulting diagram is $\sigma$-reduced and semi-$P$-reduced. If there exist such $R'(j,
d, e, v)$,
$u$,  $x$
and $y$, then $\bP$ is a green decorated place.

\begin{defn}\label{def:dec_vertex_graph}
The \emph{decorated vertex graph} $\cV$ of $\cI(\cP)$ has two sets of
vertices. There is a \emph{green $\cV$-vertex} $(a, b, s, \sfG)$ if and
only if there exists a decorated location $R(i, a, b, s)$. 
There is a \emph{red $\cV$-vertex} $(a, b, \sfR)$ for each intermult
pair $(a, b)$. 

There is a $\cV$-edge from $(a, b, s, \sfG)$ to $(d, e, v, \sfG)$ if
there exist decorated locations $R(i, a, b, s)$ and $R'(j, d, e, v)$
such that 
each one-face or two-face diagram with faces 
labelled by elements of $\cI(R)$ and $\cI(R')$,  sharing an edge at
these locations
labelled $[s^{\sigma}bu] \in X$ on the $R$-side and $[y^{\sigma}dv] = [s^{\sigma}bu]^{\sigma}$ on the
$R'$-side, is $\sigma$-reduced and semi-$P$-reduced, for some $u \in \Post(R(i))$ 
and $y \in \Pre(R'(j))$. 

There is a $\cV$-edge from $(a, b, s, \sfG)$ to
each $([s^{\sigma}bu]^{\sigma}, c, \sfR)$, where $u \in \Post(R(i))$
for some $R(i, a, b, s)$ and $[s^{\sigma} b u] \in X$. 
There is a $\cV$-edge from $(a, b, \sfR)$ to 
$(c, d, t, \sfG)$ if and only if there exists a decorated location
$R(j, c, d, t)$, and a $u \in \Pre(R(j))$, such that $[u^{\sigma}ct] = b^{\sigma}$. 

The $\cV$-edges have weight $1$ if their source is green, and weight
$0$ if it is red.  
\end{defn}

The most significant difference between \RSymVerify and \RSymIntVerify
is in finding the one-step reachable decorated places
(see Algorithm~\ref{alg:one_step}).
To simplify the exposition, we shall break this task
into two parts: finding the edges between green faces, and finding the
steps. It would be quicker to carry out these tasks concurrently. 

The following algorithm, \FindEdges, takes as input $R \in \cR$ and
$S \in \cR^{\pm}$, and 
 returns a list $L_{R, S}$ of all possible consolidated edges $e$
between faces
$f$ and $f_1$ with labels in $\cI(R)$ and $\cI(S)$.
The list $L_{R, S}$ consists of  a 5-tuple for each (not necessarily
maximal) consolidated edge
$e$: 
the decorated locations 
of $R$ at the beginning and end of $e$ in $f$, the two
corresponding decorated locations of $S$, and the
length of $e$. 

\begin{alg}\label{alg:find_edges} $\FindEdges(R,S)$:
\begin{mylist2}
\item[Step 1] Initialise $L_{R, S}:= [ \ ]$.
\item[Step 2] For each pair of decorated locations
$R(i, u_0, u_1, s)$ and $S(j, v_1, v_0, t)$, with corresponding cyclic conjugates
$u_0 u_1 \ldots u_{n-1}$ of $R$ and $v_0^{\sigma}v_1^{\sigma} \ldots
v_{m-1}^{\sigma}$ of $S^{-1}$:
\begin{enumerate}
\item[(a)] Test, using $\cV$,  whether these could be
  the beginning of a consolidated edge $e$.
\item[(b)] If so, then consider each possible  consolidated edge length $l = 1, 2,
  \ldots,  r$. 
For each $s_l \in \cI(u_l, u_{l+1})$ and each
$t_l \in \cI(v_{l+1}, v_l)$, if 
$$[s^{\sigma}u_1]u_2\cdots u_{l-1}[u_ls_l] =_{U(P)} 
([t_l^{\sigma}v_l]v_{l-1}\cdots v_2[v_1t])^{-1}$$ 
then add $(R(i, u_0, u_1, s), R(i + l, u_l, u_{l+1}, s_l),
  S(j, v_1, v_0, t),$ $S(j - l, v_{l+1}, v_l, t_l),$  $l)$ to $L_{R,
    S}$. If not, then do nothing. 
\end{enumerate}
\item[Step 3] Return $L_{R, S}$.
\end{mylist2}
\end{alg}

\begin{lemma}
Let $R \in \cR$ and $S \in \cR^{\pm}$. Then $\FindEdges(R, S)$ returns 
all of the consolidated
edges between cyclic conjugates of elements of $\cI(R)$ and
$\cI(S)$.  Furthermore, $\FindEdges(R, S)$ runs in time $O(r^4|X|^4)$.
\end{lemma}

\begin{proof}
To simplify the notation, we will write
$R = u_1u_2 \cdots u_n$ and $S= v_mv_{m-1} \cdots v_1$.
Assume that
$$\begin{array}{l}
R' = [s_0^{\sigma}u_1s_1][s_1^{\sigma}u_2s_2] \cdots
    [s_{n-1}^{\sigma}u_n s_0] \in \cI(R), \\
S' = [t_0^{\sigma}v_m t_{m-1}][t_{m-1}^{\sigma}v_{m-1}t_{m-2}]
\cdots [t_1^{\sigma}v_1t_0] \in \cI(S) \end{array}$$ have a
consolidated edge $e$ between them. 

Since $\FindEdges$ considers each possible pair of starting decorated locations,
we may assume without loss of generality that $e$
is labelled by the subwords
$$[s_0^{\sigma}u_1s_1][s_1^{\sigma}u_2s_2]\cdots
[s_{l-1}^{\sigma}u_ls_l] \quad \mbox{and}\quad
[t_l^{\sigma}v_lt_{l-1}][t_{l-1}^{\sigma}v_{l-1}t_{l-2}]\cdots
[t_1^{\sigma}v_1t_0]$$
of $R'$ and $S'$, respectively.  Hence
$$[s_0^{\sigma}u_1s_1][s_1^{\sigma}u_2s_2]\cdots
[s_{l-1}^{\sigma}u_ls_l] =_{F(X^\sigma)}
[t_0^{\sigma}v_1^{\sigma}t_1][t_1^{\sigma}v_2^{\sigma}t_2] \cdots
[t_{l-1}^{\sigma}v_l^{\sigma}t_l],$$
 and these words are $P$-reduced.
So
$$\begin{array}{rl}
[s_0^{\sigma} u_1]u_2\cdots u_{l-1}[u_ls_l] & =_{U(P)}
  [s_0^{\sigma}u_1s_1][s_1^{\sigma}u_2s_2]\cdots [s_{l-1}^{\sigma}u_ls_l] \\
& =_{U(P)} [t_0^{\sigma}v_1^{\sigma}t_1][t_1^{\sigma}v_2^{\sigma}t_2] \cdots
[t_{l-1}^{\sigma}v_l^{\sigma}t_l] \\
& =_{U(P)} [t_0^{\sigma}v_1^{\sigma}]v_2^{\sigma} \cdots
v_{l-1}^{\sigma} [v_l^{\sigma}t_l]\\
& =_{U(P)} ([t_l^{\sigma}v_l]v_{l-1}\cdots v_2[v_1t_0])^{-1}. \end{array}$$
So $e$ will be found by $\FindEdges(R, S)$. 

For the complexity claims, notice that there are $O(r^2 |X|^2)$
decorated locations in Step 2 of \texttt{FindEdges}, that $l \leq r$, 
that for each $l$ we consider $O(|X|^2)$ pairs $(s_l, t_l)$ of
interleaving elements, and that 
$$[s^{\sigma}u_1]u_2\cdots u_{l-1}[u_ls_l]
[t_l^{\sigma}v_l]v_{l-1}\cdots v_2[v_1t]=_{U(P)}1$$
can be tested in time $O(l) = O(r)$,  by Corollary
\ref{cor:wpuplinear}. 
\end{proof}

For each decorated place $\bP = (R(i,a,b,s),c,C)$, we
compute a list $\OneStep(\bP)$ of decorated places that are 1-step reachable
from $\bP$, together with an upper bound on the corresponding step
curvature, as follows. 

If $C=\sfR$ then for each decorated place
$\bQ = (R(i+1,b,d,t),x,C')$, for each $y \in X$ such that $y$
intermults with $[t^{\sigma}b^{\sigma}s]$, and for 
each $\cV$-vertex $\nu_2$ of colour $C'$ with a $\cV$-edge from $(b, d,
t, \sfG)$ to $\nu_2$, we let $\chi_1 = \Blob(y, [t^\sigma b^{\sigma} s], c)$ and 
$\chi_2 = \Vertex((y, [t^\sigma b^\sigma s], \sfR), (b, d, t,\sfG), \nu_2)$. 
We include $(\bQ,1,\chi_1+\chi_2)$ in $\OneStep(\bP)$.

If $C=\sfG$ then we use those 5-tuples in the list $L_{R, S}$ with
first entry $R(i,a,b,s)$ to locate the possible decorated places
$\bP' = (R(j,d,e,s_l),c,C')$ that can be reached from $\bP$ by a single
consolidated edge. 
For each such 5-tuple, the fifth data item specifies the length of this
edge, the second identifies the location of $\bP'$, and the fourth
identifies the $\cV$-vertex $\nu_1$, in the notation of Case $C = \sfG$ in
Algorithm~\ref{alg:one_step}. Furthermore, the component $c$ of $\bP$
must be equal to $[t^{\sigma}v_0u]$ for some $u \in P$, where $R_1(j, v_1, v_0, t)$
is the second entry of the 5-tuple. 
Otherwise, Case $\sfG$ is as in Algorithm~\ref{alg:one_step}.

Here is an overall summary of \RSymIntVerify.

\begin{proc}\label{proc:RSymIntVerify} \RSymIntVerify($\cP = \langle X
  \ | \ V_P \ | \ \cR\rangle$, $\varepsilon$):
\begin{mylist2}
\item[Step 1] For all $a, b \in X$, use $V_P$ to test whether $(a, b)
  \in D(P)$, and store the result. If $(a, b) \not\in D(P)$ then compute and store the
  interleave set $\cI(a, b)$. 
\item[Step 2] Express each relator $r \in \cR$ as a power $w^k$ and find
  all decorated locations on $w$. Store  the pre- and post-interleave
  sets of each decorated location. 
\item[Step 3] For each decorated location, find all corresponding
  decorated places. 
\item[Step 4] Compute the decorated vertex graph, $\cV$. 
\item[Step 5] Use $\cV$ to create the \Vertex function.
\item[Step 6] The set of
$\cR$-letters is all elements of $X$ of the form
$[s^\sigma b u]$, where $R(i, a, b, s)$ is a decorated location and $u \in
\Post(R(i))$. Create the \Blob function. 
\item[Step 7] For each $R, S \in \cR$, use \FindEdges to compute the
  list $L_{R, S}$ of possible consolidated edges between faces with
  labels in $\cI(R)$ and $\cI(S)$. Hence, for each decorated place
  $\bP = (R(i, a, b, s), c, C)$, construct the list $\OneStep(\bP)$. 
\item[Step 8] From each decorated start place $\bP_s$, run
  $\RSymVerifyAtPlace(\bP_s, \varepsilon)$ (using decorated places rather than places). If
  it returns \texttt{fail} and a
  list $L$, then return \texttt{fail} and $L$, otherwise do nothing. 
\item[Step 9] Return \texttt{true}.
\end{mylist2}
\end{proc}

It is clear that \RSymIntVerify
runs in polynomial time, although with a higher time complexity than that
of \RSymVerify.

\section{$\RSym$ and the word problem}\label{sec:wp}

Suppose that $\RSym$ succeeds on a presentation
$\cP = \langle X^\sigma \mid V_P \mid \cR \rangle$ for a group $G$.  We shall
show in this section that this leads to a linear time algorithm for solving
the word problem in $G$, which can be made into a practical algorithm
in many examples.  
We remind the reader that all newly defined terms, notation and
procedures are listed in the Appendix. 

One approach to solving the word problem is to use the following
result. 

\begin{prop}\label{prop:rsymdehn}
Let $G$ be defined by the pregroup presentation
$\cP = \langle X^\sigma \mid V_P \mid \cR \rangle$, and let $\cP_G$ be
the standard group presentation of $G$. Let $r$ be the 
length of the longest relator in $V_P \cup \cR$.
Suppose that, for some
constant $\lambda$, the Dehn function $\De(n)$ of $\cP_G$ satisfies
$\De(n) \le \lambda n$ for all $n \ge 0$. 
Then any word $w$ over $X$ with $w =_G 1$ has a subword of length at
most $384 \lambda r(r-1)  + 64$ that is not geodesic.
\end{prop}
\begin{proof}
It is shown in the proof of \cite[Theorem 6.5.3]{HRR} that $G$
is hyperbolic and that all geodesic triangles in its Cayley graph of $G$
are $\delta$-slim with $\delta \le 96\lambda r^2 + 4$.  In fact the result
of \cite[Lemma 6.5.1]{HRR} can easily be improved from $\mathsf{area}(\Delta)
\ge mn/l^2$ to $\mathsf{area}(\Delta) \ge 4mn/l(l-1),$ which results in the
improved bound  $\delta \le 24\lambda r(r-1) + 4$.
It is proved in \cite[Theorem 6.1.3]{HRR} that all geodesic triangles in
the Cayley graph are $4 \delta$-thin, and then \cite[Theorem 6.4.1]{HRR}
implies that any word $w$ with $w=_G 1$ must contain a non-geodesic word
of length at most $16\delta$, which is at most $384 \lambda r(r-1)  + 64$.
\end{proof}

We proved in Theorem~\ref{thm:rsymsucceeds} that if $\RSym$ succeeds
with constant $\varepsilon$, then the pregroup Dehn function $\PD(n) \le \lambda_0 n$, where the
constant $\lambda_0$ depends only on $r$ and $\varepsilon$.  It follows
from this and Lemma~\ref{lem:dehn_convert} that $\De(n) \le \lambda
n$, with $\lambda = r\lambda_0 + 1/2$. 
So we can apply Proposition~\ref{prop:rsymdehn}, and compute $\gamma =
384 \lambda r(r-1)  + 64$
explicitly, which is typically a moderately large but not a huge number.
We can therefore solve the word problem in linear time using a Dehn
algorithm, provided that we can solve it for words of length at most
$\gamma$, 
which can in principal be accomplished in constant time using a brute
force algorithm which tests all products of conjugates of up to
$\De(n)$ relators. This is clearly impractical. 
In practice, one possibility is to use \KBMAG for this purpose, but that
is only possible if \KBMAG can compute the automatic structure, which may
not be feasible, particularly for examples with large numbers of generators.

An alternative approach to solving the word problem, which succeeds in many
examples, is to use the success of 
$\RSym$ directly to produce a linear-time word problem
solver \RSymSolve, and the main purpose of this section is to describe how to do
that. It is not guaranteed that such a solver can be constructed, even when
$\RSym$ succeeds, but the attempted construction of \RSymSolve takes
low-degree polynomial time.

\medskip

The remainder of the section is structured as follows: 
first, in Definition~\ref{def:verify} we define an extra condition that \RSym may satisfy, which
guarantees that 
\RSymSolve works. 
 Then we present a procedure called
\VerifySolver, which tests for
this extra condition. Finally we describe the algorithm
\RSymSolve, which solves the word problem, and prove its correctness and
complexity.
It is similar in spirit to a Dehn algorithm, with complications
arising from the fact that we work over $\cI(\cP)$ whilst only
storing rewrites arising from $V_P \cup \cR$, and that we need to work
with $P$-reductions rather than just free reductions. 

\begin{defn}\label{def:verify}
$\RSym$ \emph{verifies a solver} for $\cI(\cP)$
if, for any green boundary face $f$ in any $\Gamma \in \calD$ with
$\kappa_{\Gamma}(f) > 0$, the removal of $f$
shortens $\partial(\Gamma)$.
\end{defn}

We assume in this section that \RSym has succeeded. 
In fact, it is possible to use $\RSym$ to solve the word problem under
somewhat weaker hypotheses: see Remark~\ref{rem:triangles_dehn}. 

The procedure \texttt{VerifySolver} 
 seeks to check  that \RSym verifies a solver for $\cI(\cP)$. 
It is very similar to the main $\RSym$ tester,
except that some of the places are on
$\partial(\Gamma)$. We shall assume
that $\RSymVerify(\cP, \varepsilon)$ has returned \true for some
$\varepsilon > 0$, and that the data
computed have been stored.

We describe \VerifySolver only for the case where $\cI(\cR) = \cR$: the
modifications necessary when $\cI(\cR) \neq \cR$ are
straightforward. We shall describe the word problem solver \RSymSolve in the
general case, when  we do not assume that $\cI(\cR) = \cR$:
as we shall explain in
Remark~\ref{rem:standard_dehn} below, if \VerifySolver succeeds then 
we can use a standard Dehn algorithm when
$\cI(\cR) = \cR$.

\begin{proc}\label{proc:VerifySolver} \texttt{VerifySolver($\cP =
    \langle X^\sigma \ | \ V_P \ | \ \cR \rangle$)}:
\begin{mylist2}
\item[Step 1] For each $R \in \cR$ do
\begin{enumerate}
\item[(a)]  For each place $\bP_s$ on $R$ do
\begin{enumerate}
\item[(i)]  If \texttt{VerifySolverAtPlace($\bP_s$)} returns \texttt{fail}
  then return \texttt{fail}.
\end{enumerate}
\end{enumerate}
\item[Step 2] Return \texttt{true}.
\end{mylist2}
\end{proc}

For a start place $\bP_s =  (R(i, a, b), c, C)$ on a face $f$
labelled by $R$, \texttt{VerifySolverAtPlace} works along
all possible sequences of internal edges of  $f$ starting at
$\bP_s$, bounding  the resulting curvature of $f$. We seek
to show that if $\kappa_{\Gamma}(f) > 0$, then these
internal 
edges take up less than half of $\partial(f)$.

The vertex where \texttt{VerifySolverAtPlace} terminates is on
$\partial(\Gamma)$, so
need not be a place in the sense of Definition~\ref{def:place}, as it 
need not be instantiable for any choice of
extra letter $c \in X$. We now rectify this. 

\begin{defn}\label{def:terminal}
A \emph{terminal place} is a
triple $(R(i, a, b), \texttt{terminal}, \sfG)$ where $R(i, a, b)$ is a
location. A terminal place is green, and has no extra letter. For use in this section, 
Definition~\ref{def:one_step} should be modified, to say that 
no place is 1-step reachable \emph{from} a terminal place. 
\end{defn}

 For the rest of this section, a \emph{place} may be terminal, unless
 specified otherwise.
Before running \texttt{VerifySolver} we re-run 
$\texttt{ComputeOneStep}(\bP)$
(Algorithm~\ref{alg:one_step}), 
 to also find the one-step reachable terminal
places from each non-terminal place $\bP$. 
The curvature values $\chi$ of the triples $(\bQ, l, \chi)$ in
$\OneStep(\bP)$, when $\bQ$ is terminal,  are as
follows (we shall justify them in the proof of
Theorem~\ref{thm:verify_solver}). 
\begin{mylist2}
\item[$C = \sfR$.]  We set $\chi$ to be the maximum curvature given to
  a face labelled $R$ by a
boundary red blob at $\bP$ (calculated using
Lemma \ref{lem:tab_blob_bound} with $|\delta(B) \cap \delta(\Gamma)| \geq 1$).

\item[$C = \sfG$.] If $\bP' = \bQ$ we set
$\chi = -1/4$. If $\bP' \neq \bQ$
then we let $\chi$ be the maximum sum of the
curvature given to a face labelled  $R$ by the vertex at $\bP'$ and a boundary red blob
at $\bP'$. 
\end{mylist2}

For a non-terminal place $\bP_s$ on $R$, $\VerifySolverAtPlace(\bP_s)$
makes a list of 4-tuples
$(\bQ, l, t, \psi)$. The first three components represent a place $\bQ$
at distance $l$ from $\bP_s$ along $R$ that can be reached from $\bP_s$
in $t$ steps,  where if $\bP_s$ is green then the first
``step'' consists of only the boundary vertex. 
The final component $\psi$ is $1 + \chi$, where
$\chi$ is the largest possible curvature given to $R$ by these $t$ steps:
unlike in Subsection~\ref{subsec:tester} we do not adjust for the
length of the step.  
The algorithm $\VerifySolverAtPlace(\bP_s)$ returns \fail if 
there is a  sequence of neighbouring green faces and red blobs, starting at
$\bP_s$ and ending at a terminal place, 
that occupies at least half of $\partial(f)$ and from which $f$ receives
more than $-1$ of curvature. It returns \true otherwise.

In the description below, the meaning of \emph{including} an entry in
the list $L$ is the same as
in  Procedure~\ref{proc:RSymVerifyAtPlace}. 
\medskip

\begin{proc}\label{proc:VerifySolverAtPlace}
\texttt{VerifySolverAtPlace}($\bP_s = (R(i, a, b), c, C)$): 
\begin{mylist2}
\item[Step 1]  Let $n := |R|$, and initialise $L:= [ \, ]$. 
\item[Step 2]  If $C  = \sfG$ then include $(\bP_s, 0, 1, 3/4)$ in $L$. 
\item[Step 3] If $C = \sfR$ then, for each non-terminal place $\bP_1$ at distance $1$ from
  $\bP_s$, calculate the maximum value of $\chi = \chi(B, f, \Gamma) +
  \chi(v, f, \Gamma)$, where $B$ is 
boundary red blob between $\bP_s$ and $\bP_1$, and $v$ is 
the vertex at $\bP_1$. Include $(\bP_1, 1, 1, 1 +
  \chi)$ in  $L$.
\item[Step 4] For $i := 1$ to $3$, for each $(\bP, l, i, \psi) \in L$, 
and for each $(\bQ, l', \chi) \in \OneStep(\bP)$ do:
\begin{enumerate}
\item[(a)] If $l + l' < n/2$ and $\bQ$ is not terminal then let
         $\psi' := \psi + \chi$;

if $\psi' > 0$ then include $(\bQ, l+l', i+1, \psi')$ in $L$.
\item[(b)] If $l + l' \geq n/2$, $\bQ$ is terminal, and
 $\psi +\chi >  0$,  then return \texttt{fail} and $L$.
\end{enumerate}
\item[Step 5] Return \texttt{true}.
\end{mylist2}
\end{proc}

\noindent
Note that we do nothing in Step 4 if neither of the specified conditions hold. 
\medskip

\begin{thm}\label{thm:verify_solver}
Assume that $\cI(\cR) = \cR$. If $\texttt{VerifySolver}$ returns \texttt{true}, then 
\RSym verifies a solver for $\cP$. 
The procedure \texttt{VerifySolver} runs in polynomial time.
\end{thm}

\begin{proof}
Let $f$ be a boundary green face of a diagram $\Gamma \in \calD$ such
that $\kappa_{\Gamma}(f) > 0$, and assume that at most half of
the edges of $f$ are contained in $\partial(\Gamma)$. Let the label of
$f$ be $R \in \cR$. We show that
\texttt{VerifySolverAtPlace} returns \texttt{fail} for at least one
start place $\bP_s$
on $R$. 

By Lemma~\ref{lem:vert_bounds}, a boundary vertex $v$ with
$\delta_G(v, \Gamma) \ge 3$ satisfies $\chi(v, f, \Gamma) \leq -1/4$, so 
 Step 2 of \VerifySolverAtPlace, and
the value of $\chi$ in $\OneStep(\bP)$
 when $\bP$ is green and $\bQ = \bP'$ is terminal, 
correctly bound $\chi(v, f, \Gamma)$ when the first or
last edge of $\partial(f)
\setminus \partial(\Gamma)$ is incident with a green internal face. 

Step 3 of \VerifySolverAtPlace, and the remaining cases of
 $\OneStep(\bP)$ when $\bQ$ is terminal,
bound $\chi(B, f, \Gamma)$ as if 
the $B$ is on the boundary. To see why this is correct,
first notice that the label of any non-boundary red blob is
permissable as a label  of a boundary red blob. Let $v$ be the vertex
at $\bP_s$, and let $B$ be a
red blob at $\bP_s$ with boundary length $l$ and area $t$ (the case
where the red blob is at the end of $\partial(f)
\setminus \partial(\Gamma)$ is equivalent). By
Lemma~\ref{lem:blob_bound}, $\chi(B, f, \Gamma)$ is maximised by
assuming that $B$ is simply-connected, in which case $l = t+2$ since
$\Gamma$ is \valid.  If $B$ has a 
boundary edge at $\bP_s$, then by Lemmas~\ref{lem:vertex_curve} and
\ref{lem:blob_curve}, $\chi(v, f, \Gamma) = 0$ and 
$\chi(B, f, \Gamma) \leq \frac{-t}{2(t+1)}$. If $B$ has no 
boundary edge at $\bP_s$, then $\chi(v, f, \Gamma) \leq -1/4$ and
$\chi(B, f, \Gamma) \leq \frac{-t}{2(t+2)}$.
 For all $t$ 
$$\frac{-t}{2(t+1)} > \frac{-t}{2(t+2)} - \frac{1}{4}$$
so $\chi(v, f, \Gamma) + \chi(B, f, \Gamma) \leq \frac{-t}{2(t+1)}$. Hence
Step 3 of \texttt{VerifySolverAtPlace} and  $\OneStep(\bP)$ 
correctly bound the curvature received by $R$ when the first or last
edge of $\partial(f)
\setminus \partial(\Gamma)$ is incident with with a red blob. We
showed in Lemma~\ref{lem:step_bound} that the 
\OneStep lists correctly bound all other step curvatures.

We showed in Lemma~\ref{lem:boundary_face} that $\partial(f) \cap 
\partial(\Gamma)$
consists of at most one consolidated edge together with at most one isolated
vertex $v$, and that the internal edges of $\partial(f)$ form a path. Therefore
each place on $R$ at a vertex 
along these internal edges (except for the last one) has a face of
$\Gamma$ instantiating it, and so is not terminal. Therefore
$\VerifySolverAtPlace$ is correct to require in
Steps 3 and 4(a) that all places are non-terminal.

It remains only to show that $\partial(f) \setminus \partial(\Gamma)$
consists of at most three steps along $R$, after the initial one
(which may not formally be a ``step''). 
Assume first that $f \cap \partial(\Gamma)$ contains an isolated
vertex $v$. Then we can decompose $\partial(f)$ into $v_1, \beta_1,
v, \beta_2, v_2, \beta_3$, where $v_1, v_2, v$ are vertices on
$\partial(\Gamma)$, and each $\beta_i$ is a sequence of edges and
vertices such that $\beta_3 \subset \partial(\Gamma)$,
and $\beta_1$ and $\beta_2$ are internal to $\Gamma$
(see Figure \ref{fig:isolbv}).  As we have just seen, the curvature
given to $f$ by the vertices and blobs 
at the beginning of $\beta_1$ and the end of $\beta_2$ sums to at most $-1/2$.
We showed in Lemma~\ref{lem:boundary_face} that $f$ is incident with
no red blobs at $v$, and that $\delta_G(v, \Gamma) \geq 4$, so
$\chi(v, f, \Gamma) \leq -1/3$. 

Since $\kappa_{\Gamma}(f)$ is assumed to be positive, by
Lemmas~\ref{lem:vert_bounds} and \ref{lem:tab_blob_bound} the face $f$
is adjacent to only these two green internal faces (and
possibly some red blobs at $v_1$ and $v_2$), and so
\VerifySolverAtPlace will return \fail when $i \leq 2$, with  $\bP_s$ a place at
$v_1$. 

Next assume that $f \cap \partial(\Gamma)$ consists of a single
consolidated edge $e$. Write the boundary of $f$ as $e, v_1,
\beta, v_2$, where $\beta$ is a sequence of edges and vertices
internal to $\Gamma$. The vertices $v_1$ and $v_2$, together with any incident red
blobs, give at most $-1/2$ of curvature to $f$. Excluding the
red blobs that might be incident with $v_1$ and $v_2$, by
Lemmas~\ref{lem:vert_bounds} and \ref{lem:tab_blob_bound} the path $\beta$
can be incident with at most two red blobs, or at most two vertices of
green degree greater than two, or exactly one of each. Hence $f$ is
adjacent to at most three internal green faces, and
\texttt{VerifySolver} will return \texttt{fail}. 

The complexity claims follow as in the proof of Theorem~\ref{thm:r_sym_complexity}.
\end{proof}

We now show how to solve the word problem, provided that \RSym
verifies a solver. 
First we show that any word $w = x_1 \ldots x_n$ can be
cyclically $P$-reduced in linear time.

\begin{prop}\label{prop:p-reduced}
Let $w = x_1 \ldots x_n \in X^\ast$. On input $w$ and $\mathcal{P} = \langle X^\sigma
\mid V_P \mid \mathcal{R} \rangle$, a cyclically $P$-reduced word $w'$
that is
conjugate in $U(P)$ to $w$ can be found in time $O(|w| - |w'|) = O(n)$.
\end{prop}

\begin{proof}
The word problem over $U(P)$ can be solved in linear time by 
Corollary~\ref{cor:wpuplinear}, so without loss of generality we may
assume that $x_1 \ldots x_n$ is $P$-reduced.

  It remains to consider
cyclic $P$-reduction.  Assume we have two
pointers, \texttt{start}, initially pointing at $x_1$ and \texttt{end},
initially pointing at $x_n$.
We check whether $(x_n, x_1) \in
D(P)$. If $[x_nx_1] =_{P} 1$, we move \texttt{start} to $x_2$ 
and \texttt{end} to $x_{n-1}$. If
$[x_nx_1] =_P a \neq 1$, we replace $x_1$ by $a$, and move
\texttt{end} to $x_{n-1}$. We continue this process until the letters
$s$ and $t$ 
pointed to by \texttt{end} and \texttt{start} satisfy $(s,t) \not\in
D(P)$. 
We then test $(t,u) \in 
D(P)$, where $u$ is the letter after the one to which \texttt{start}
points.
 If not, we are done. If $tu =_P 1$ then we
 move \texttt{start} forwards by two letters.
If $tu =_P a \neq 1$ then we replace $u$ by $a$ and move
\texttt{start} forward by one letter.
We continue moving \texttt{start} and \texttt{end} towards the middle
of $w$ until no further reductions are possible. We
then return the word reading from
\texttt{start} to
\texttt{end}. 

Let $w'$ be the resulting word, with $m = |w| - |w'|$. Then
each pointer  moves $O(m)$ times, and $O(m)$ 
products in $P$ are calculated.
\end{proof}

We compute a list $\cL$, whose entries are 
pairs of words $(u, v) = (u_1\cdots u_k, v_1 \cdots v_l) \in X^\ast
\times X^\ast$, where
$[s^{\sigma}u_1] u_2 \cdots u_{k-1}[u_kt]
([s^{\sigma}v_1] v_2 \cdots v_{l-1}[v_lt])^{-1}$ is a
cyclic conjugate of some $R \in \cR^{\pm}$ for some $s,t \in P$,
and $k = \lceil (|R| + 1)/2 \rceil$. 

In the light of Proposition~\ref{prop:p-reduced} we may assume that
the input to \RSymSolve is a cyclically $P$-reduced word $w = x_1 \cdots x_n \in
X^\ast$. Let $r$ be
the length of the longest relator in $\cR$. In the description below,
we interpret all subscripts cyclically, so that $x_{n+1} = x_1$.  

\begin{alg}\label{alg:RSymSolve} $\RSymSolve(w = x_1\ldots x_n)$:
\begin{mylist2}
\item[Step 1] Store $w$ as a doubly-linked list: each letter has a
  pointer to the letter before it, and the letter after it.
\item[Step 2] Set $\alpha := 1$.
\item[Step 3] For $\alpha \leq i \leq n$, search for $m \in
\{1, \ldots, \lceil (r+1)/2 \rceil\}$, $a \in \mathcal{I}(x_{i-1},
x_i)$,  $b \in \mathcal{I}(x_{i+m-1}, x_{i+m})$ and $(u, v) \in \cL$ such that
$[a^\sigma x_i] x_{i+1} \dots x_{i+m-2} [x_{i+m-1} b] =_{U(P)} u$.
\begin{enumerate}
\item[(a)]  Let $i, m, a, b, v :=v_1 \ldots v_l$ be the first such found, if any.
\item[(b)]  If none such exist then $w \neq_G 1$. Terminate and return
  \false.
\end{enumerate}
\item[Step 4] Replace $x_{i-1}$ by $[x_{i-1}a]$ and replace
  $x_{i+m}$ by $[b^\sigma x_{i+m}]$. 
Put a pointer \texttt{CutStart} to the new $x_{i-1}$  and a
  pointer \texttt{CutEnd} to the new $x_{i+m}$.
  Store $v$ as a doubly-linked list, with pointers \texttt{NewStart} to
  $v_1$ and \texttt{NewEnd} to $v_l$.
\item[Step 5] $P$-reduce at the beginning of $v$: 
If $[x_{i-1}v_1] =_P s \neq 1$, then replace $v_1$ by $s$,
  and move \texttt{CutStart} to $x_{i-2}$. If $[x_{i-1} v_1] =_P
  1$, then move \texttt{NewStart}  and \texttt{CutStart} to $v_2$ and $x_{i-2}$. 
\item[Step 6]  Repeat Step 5 until one of the following: no further reductions are
  found; \texttt{CutStart} should be moved back past $x_1$; \texttt{NewStart}
  should be moved forward past $v_l$. 
\item[Step 7] Provided that there is at least one letter left in $v$, perform
 Steps 5 and 6 (with the appropriate pointers) to $P$-reduce at the end of $v$. 
\item[Step 8] Update the links in the list describing $w$ 
so that whatever remains of $v$ is now
  inserted into the correct place in $x_1 \ldots x_n$, yielding a word $w_1$. 
\item[Step 9] Cyclically
  $P$-reduce $w_1$, as in the proof of Proposition~\ref{prop:p-reduced},
  yielding a word $w_2$. If $w_2$  is empty, then terminate and return \true.
\item[Step 10]  Let $j$ be the position in $w_2$ to which \texttt{CutStart}
  points, 
and let $\alpha := \max\{1, j - \lceil (r+1)/2 \rceil + 1\}$.  
Replace $n$ by $|w_2|$, and go to Step 3 with $w_2$ in place of $w$. 
\end{mylist2}
\end{alg}

(When returning to Step 3, we start the search for the next rewrite at $x_{\alpha}$,
as earlier untouched letters will still not be eligible for
rewriting.)

\begin{thm}\label{thm:solver}
Let $\cP = \langle X^\sigma  \mid V_P \mid \cR \rangle$ 
be a pregroup presentation for a group $G$, such that $\RSym$ succeeds
on $\cP$. If  
$\VerifySolver$ succeeds on $\cI(\cP)$ 
then for all $n \in \mathbb{N}$, and for all $x_1 \ldots x_n \in X^\ast$, the
algorithm 
$\RSymSolve(x_1 \ldots x_n)$
correctly tests whether $x_1 \ldots x_n =_G 1$ in time $O(n)$. 
\end{thm}

\begin{proof} 
Let $w = x_1 \ldots x_n \in X^\ast$.  By Proposition~\ref{prop:p-reduced} in
$O(n)$ we can replace $w$ by a word $w_1 = y_1 \ldots y_k$ that is cyclically
$P$-reduced. 
By Theorem~\ref{thm:rsymsucceeds}, since $\RSym$
succeeds on $\cP$, for some $w_2
\in \cI(w_1)$ there exists a diagram $\Gamma \in \calD$ with boundary word
$w_2$.  By Lemma~\ref{lem:boundary_face}, either $\Gamma$ consists of
a single face $f_1$, which must be green and have curvature $+1$, or
there are 
at least two boundary faces $f_1$ and $f_2$ of $\Gamma$ with positive
curvature. In this second case, since  \VerifySolver succeeds,
the faces $f_1$ and $f_2$ each  have more than
half of their boundary length as a continuous subword of the
boundary of $\Gamma$. 

Hence a 
rewrite applies to at least one subword $z_i \ldots z_{i+m-1}$ of
$w_2$. Such a subword is
equal in $U(P)$ to $[a^\sigma y_i] \ldots [y_{i+m-1}b]$ for some $a \in
\mathcal{I}(y_{i-1}, y_i)$ and $b \in \mathcal{I}(y_{i+m-1}, y_m)$,
and so will be found by \RSymSolve.
 
On input a cyclically $P$-reduced word $w$ of length
$n$, \RSymSolve
runs $O(n)$ tests of equality  in $U(P)$ of words of the form
$a^\sigma w'b$,
where $w' = t_1 \ldots t_m$ is a subword of $w$ with $m \leq (r+1)/2$ 
such that $(a^\sigma, t_1) \in D(P)$ and
$(t_m, b) \in D(P)$, with the first entry of each pair in
$\mathcal{L}$.
It also tests for $O(n)$ (cyclic) $P$-reductions. Every
time a letter of $w$ (or of a substring $v$ for replacement into $w$)
is changed, it is due to a shortening of $|w|$, so at most $O(n)$ letter replacements occur. 
\end{proof}

\begin{remark}\label{rem:standard_dehn}
If \VerifySolver succeeds and $\cI(\cR) = \cR$,
then we can replace \RSymSolve by a standard Dehn algorithm using
the length reducing rewrite rules derived from  $V_P \cup \cR$. The presence of
the rules from $V_P$ ensures that we reduce our input word $w$ to a $P$-reduced word.
There is no need to carry out any additional cyclic reduction or cyclic $P$-reduction
on $w$.

This is because cyclic ($P$-)reduction might delete some letters at the
beginning and end of $w$, and might insert a single new letter at the
beginning of $w$. If the resulting cyclically $P$-reduced word  $w'$ was equal to
the identity in $G$, then a diagram $\Gamma \in \calD$ for $w'$ would
either consist of a single green face, or have at least
two green regions with more than half of their length on the
boundary. In the former case, all but one letter of $w'$ is a subword
of the original word $w$. In the latter case, the 
label of the intersection of at least one of these two regions with the boundary
of $\Gamma$ would be a subword of $w$. So this subword would
be reduced in length by the application of one of the rewrite rules.

Although a standard Dehn algorithm is no faster than \RSymSolve in terms
of complexity (both are linear in the length of $w$) it has the advantage that
it can be implemented efficiently using a two stack model, as described
for example in \cite{DomanskiAnshel}.
\end{remark}

An improvement to \VerifySolver is sometimes possible: for example,
when $P$ is the standard pregroup for a free product of finite and free groups.

\begin{remark}\label{rem:triangles_dehn}
Consider the situation where if $(a, b)$ is an intermult pair, then
$(a, b) \in D(P)$. Note that this implies in particular that $\cI(\cR)
= \cR$. In this case we design an upgrade to \VerifySolver, which we
call 
$\VerifySolverTrivInt$.

Let $\Gamma$ be a diagram in $\calD$ and let $f$ be a 
boundary face of $\Gamma$ with $\kappa_{\Gamma}(f)>0$, with $k$ boundary edges and
$l$ internal edges. By Lemma~\ref{lem:boundary_face}, the
edges in $\partial(f) \setminus \partial(\Gamma)$  form a path $e_1,
\ldots, e_l$, say. 

If exactly one of $e_1$ or $e_l$  is incident to a red blob whose next
edge is on the boundary,  then deleting $f$ leaves a red blob $B$ with two edges
appearing consecutively on $\partial(\Gamma)$. By
Lemma~\ref{lem:intermult} the labels $a$ and $b$ of these two edges
  intermult. Hence $(a, b) \in D(P)$, so the new boundary word can
be $P$-reduced, deleting at least one red triangle. 
Hence,  deleting $f$ followed by $P$-reduction 
shortens $\partial(\Gamma)$ by at least $(k+1) - l$ edges. 
If both $e_1$ and $e_l$ are incident with  boundary red blobs, then
deleting $f$, followed by $P$-reduction, shortens 
$\partial(\Gamma)$ by at least $(k+2)-l$ edges. 

Hence one can produce a more powerful algorithm, \VerifySolverTrivInt,
by modifying \VerifySolverAtPlace and the \OneStep values for
terminal places. The list $L$  from \VerifySolverAtPlace
should now contain two types of entries,
those that record places on paths that start at a boundary red blob,
and those that do not.

In Step 4(a) of \VerifySolverAtPlace, we replace $l + l' < n/2$ by
$l + l' < (n+2)/2$.

In Step 4(b) of \VerifySolverAtPlace, if there is a boundary red blob at
either $\bP_s$ or just before $\bQ$, 
 then failure is \emph{only} reported if $l+l' \geq (n+1)/2$.
If $\bP_s$ and the edge before $\bQ$ are \emph{both}  incident with boundary
red blobs, then failure is only reported if $l+l' \geq (n+2)/2$. 

When using \VerifySolverTrivInt, 
appropriate additions need to be made to the list $\cL$ of
rewrites for \RSymSolve to adjoin the letters which can
appear on such boundary red blobs. We shall refer to the enhanced version
as $\RSymSolveTrivInt$.

As in Remark~\ref{rem:standard_dehn}, in many situations (we omit the details)
we can use a standard Dehn algorithm in place of $\RSymSolveTrivInt$, with
the modified list $\cL$. 
\end{remark}

We can use the success of $\VerifySolver$ or $\VerifySolverTrivInt$ to lower
our bound on the Dehn function of $G$. Recall 
Definition~\ref{def:alternative_dehn} of the pregroup Dehn function. 

\begin{prop}\label{prop:better_dehn}
Let $\PD(n)$ be the pregroup Dehn function of $\cP$. 
If both $\RSym$ and $\VerifySolver$ succeed, then  $\PD(n)
\leq 
n$. If $\RSym$ and $\VerifySolverTrivInt$ succeed, then $\PD(n) \leq 3n$.
\end{prop}

\begin{proof}
The first claim is clear, since $\RSymSolve$ is a variation of a  Dehn
algorithm. For the second, notice that the removal of at most three
faces from each diagram results in a shortening of the boundary word. 
\end{proof}

\begin{remark}
Assume that we know that all $V^{\sigma}$-letters are nontrivial in $G$.
Then with a little effort one may show that
we can deduce the same bounds on the pregroup Dehn function of $G$ from the
success of $\RSymSolve$ and $\RSymSolveTrivInt$, even when \RSym fails,
provided that \RSym
is able to show that all green faces at dual distance at least two
from the external face have \emph{non-positive} curvature. The assumption
that no $V^\sigma$-letters are trivial in $G$ means that all diagrams
are loop-minimal, which permits us to  use
Proposition~\ref{prop:red_bound}, and to deduce
that if $w =_G 1$ then there is a diagram $\Gamma \in \calD$ with
boundary word $w$.  
We excluded this
case from our earlier analysis as it gave no immediate upper bound on
the Dehn function, but $\RSymSolve$ and $\RSymSolveTrivInt$ together provide
such a bound.
\end{remark}

\section{Applications of \RSym}\label{sec:examples}

In this section we shall first show that \RSym generalises several
small cancellation conditions. We then show how \RSym
 can be verified by hand to prove
the hyperbolicity of various infinite families of presentations. 
This is an advantage over the algorithm based
on the theory of automatic groups that is used by the \KBMAG\ package, which
can only handle individual groups and is not susceptible to
hand-calculation. Finally, we discuss possible further applications of
\RSym.

We remind the reader that all new terms, notation and procedures are
listed in the Appendix. 

\subsection{Small cancellation conditions}\label{subsec:canc}

As a first example of the applicability of $\RSym$, we 
consider various small cancellation conditions,
thereby recovering the result proved in \cite[Corollary 3.3]{GerstenShort}.
Furthermore, 
in many cases \RSymSolve solves the word problem.

\begin{thm}\label{thm:smallcanc}
Let $\cQ = \langle Y \ | \ \cR \rangle$ be a group presentation for a group $G$, 
satisfying $C(p)-T(q)$ for some
$(p, q) \in \{(7, 3), (5, 4), (4, 5)\}$.
Then $\RSym$ succeeds on $\cQ$ with $\varepsilon = -1/6$, $-1/4$
and $-1/5$, respectively. If $(p, q) = (3, 7)$, then
$\RSym$ succeeds at level $2$ with $\varepsilon = -1/14$. In all of
these cases, $G$ is hyperbolic.
\end{thm}

\begin{proof}
We let $X = Y \dot{\cup} \{y^\sigma: y \in Y\}$ and set $V_P = \emptyset$, just as in
Example~\ref{ex:free_group}. Then let $\cP = \langle X \ | \ \emptyset \
| \ \cR \rangle$, so each coloured diagram over $\cP$ is also a
diagram over $\cQ$. Let $\Gamma$ be a reduced coloured
 diagram; notice that each face of 
$\Gamma$ is green, so $\Gamma \in \calD$. 
By \cite[Chapter V, Lemma 2.2]{LyndonSchupp}, the fact that $\cQ$
satisfies $C(p)-T(q)$ means that all non-boundary faces of $\Gamma$
have at least $p$ edges (and hence at least $p$ vertices), and all non-boundary vertices $v$ satisfy
$\delta(v,\Gamma) = \delta_G(v, \Gamma) \le q$.

First assume that $(p, q) \in \{(7, 3), (5, 4), (4, 5)\}$, and 
let $f$ be a non-boundary face of $\Gamma$. 
A vertex $v$ of $f$ 
that is not on $\partial(\Gamma)$ satisfies
$$\chi(v, f, \Gamma) =
\frac{2-\delta(v, \Gamma)}{2\delta(v, \Gamma)} = \frac{1}{\delta(v, \Gamma)} - \frac{1}{2}
\leq \frac{1}{q} - \frac{1}{2},$$
by Lemma~\ref{lem:vertex_curve}\,(ii). 
Since $f$ is a non-boundary face, a vertex $v$ of $f$ on $\partial(\Gamma)$
has degree at least $4$ and so 
$\chi(v, f, \Gamma) \leq -1/3$ by Lemma~\ref{lem:vert_bounds}. 
The claim that $\RSym$ succeeds, and the stated values of $\varepsilon$, now
follow from the fact that $f$ has at least
$p$ incident vertices of degree at least $q$. Since $\RSym$ succeeds
at level 1, it follows from Theorem~\ref{thm:rsymsucceeds} that $G$
is hyperbolic. 

So suppose instead that $(p,q) = (3,7)$, so that each non-boundary
face has at least $3$ incident vertices, each of degree at least $7$. If
$f$ has only three incident vertices, all boundary vertices of degree
exactly four, then $\kappa_{\Gamma}(f) = 0$, so $\RSym$ fails at level
1. We therefore apply
$\RSym$ at level $2$ to $\Gamma$, and let
 $f$ be an internal face at dual
distance at least $2$ from $\partial(\Gamma)$. 
Then 
$\delta(v, \Gamma) \ge 6$ for all
vertices $v$ of $f$ on $\partial(\Gamma)$, so Lemma~\ref{lem:vert_bounds}
tells us that $\chi(v, f, \Gamma) \leq -2/5$.
Since the non-boundary vertices $v$ of $f$ 
satisfy $\chi(v, f, \Gamma) \leq -5/14$ in this
case, and $-2/5 < -5/14$, we conclude that  $\RSym$ succeeds at level
2 with
$\varepsilon = -1/14$ as claimed.  Since $V_P = \emptyset$, there
are no $V^\sigma$-letters, 
so 
$G$ is hyperbolic  by Theorem~\ref{thm:rsymsucceeds}. 
\end{proof}

With metric small cancellation conditions, $\RSymSolve$ solves the
word problem.

\begin{thm}
Let $\mathcal{Q}$ be a group presentation satisfying $C'(1/6)$ or
$C'(1/4)-T(4)$. Then both $\RSym$ and \texttt{VerifySolver} 
succeed on $\mathcal{Q}$.
\end{thm}

\begin{proof} The success of $\RSym$ follows from
  Theorem~\ref{thm:smallcanc}.  \texttt{VerifySolver} 
  considers up to three steps from each place on each relator, with the
  vertices at each end giving curvature at most $-1/4$, and
  the intermediate vertices giving curvature at most $1/q - 1/2$,
  where $q \in \{3, 4\}$. Hence for $C'(1/6)$ a boundary face $f$ with 
$\kappa_{\Gamma}(f) > 0$ has at most three internal consolidated edges, and for
  $C'(1/4)-T(4)$ it has at most two. By Lemma~\ref{lem:boundary_face}
 these internal consolidated
  edges are contiguous, and comprise less than half of the length of
  the relator. 
\end{proof}

Our second example considers the
generalisation of small cancellation to amalgamated free products, as
described in \cite[Chapter V \S11]{LyndonSchupp}.
Let $X_1, \ldots, X_m$ be finite  groups with proper subgroups $A_i \leq X_i$,  let $A
= A_1$, and let $\psi_i : A \rightarrow A_i$ be isomorphisms.
Let $F = \langle \ast X_i \ : A = \psi_i(A_i) \rangle$ be the free product of
the $X_i$, amalgamated over the $A_i$.

A \emph{normal form} for $g \in F \setminus \{1\}$ is any expression
$y_1y_2 \cdots y_n$ such that $g =_F y_1 y_2\cdots y_n$,  each $y_i \in X_j$ for some $j$, 
successive $y_i$ come from different $X_j$, and no $y_i$ is in
$A$ unless $n = 1$. The 
length $n$, and the factors in which the $y_i$ lie, are
uniquely determined by $g$.
An element $g \in F \setminus \{1\}$
with normal form $y_1\cdots y_n$ is \emph{cyclically reduced} if
$n=1$ or $y_1$ and $y_n$ are in different factors, and
\emph{weakly cyclically reduced} if $n = 1$ or $y_ny_1 \not\in A$.
A product of normal forms $y_1 \cdots y_n x_1
\cdots x_m$ is \emph{semi-reduced}
if $y_nx_1 \not\in A$ and neither normal form is a single element of $A$.

Let $\cR$ be a set of weakly cyclically reduced elements of
$F\setminus A$. The \emph{symmetrised} set $\widehat{\cR}$ 
consists of all normal forms of all
weakly cyclically reduced
$F$-conjugates of elements of 
$\cR^{\pm}$. 
A normal form $b \in F \setminus A$ is a \emph{piece} if
there exist distinct $R_1, R_2 \in \widehat{\cR}$ such that $R_1 =_F bc_1$ and
$R_2 =_F bc_2$, where $c_1$ and $c_2$ are normal forms, 
and the products $bc_1$ and $bc_2$ are both semi-reduced.  

Notice that if $R_1 = x_1 \ldots x_n$ and $R_2 = x_1 \ldots x_k
y_{k+1} \ldots y_m$ are normal forms of  elements of $\cR$ with $m, n \ge k+2$, and
$x_{k+1}$ and $y_{k+1}$ are in the same free factor $X_i$, then $x_1
\ldots x_{k+1}$ is a piece. To see this, observe that we
 can write $y_{k+1} = x_{k+1}z$ for some
$z \in X_i$. If  $z \not\in A$ then $x_1 \ldots x_{k+1} \cdot z y_{k+2} \ldots
y_m$ is a semi-reduced product of two normal forms. If $z \in A$
then let $z_{k+2} = z y_{k+2} \not\in A$, and notice that
 the product $x_1 \ldots x_{k+1}  \cdot z_{k+2}
y_{k+3} \ldots y_m$ is a semi-reduced product of two normal forms. 

\begin{defn}\label{def:prod_canc}
A symmetrised set
$\widehat{\cR}$ satisfies $\mathcal{C}'_{FA}(\lambda)$, where $\lambda
\in \R_{>0}$,  if
\begin{enumerate}
\item[(i)] $|R| > 1/\lambda$ for all $R \in \widehat{\cR}$;
\item[(ii)] if $R\in \widehat{\cR}$ is equal in $F$ to a semi-reduced
  product $bc$, where $b$ is a piece, and 
  $c$ is a normal form, then $|b| < \lambda |R|$.
\end{enumerate}
\end{defn}

We have not attempted to optimise the value of $\varepsilon$ in the
following result. 

\begin{thm}\label{thm:amalgam}
Let $X_1, \ldots, X_m$ be finite groups, let
$F = \langle \ast X_i \ : A = \psi_i(A_i) \rangle$ be a free product with
amalgamation, and let $\cR$ be a finite set of cyclically reduced elements of
$F$ such that $\widehat{\cR}$ satisfies $\mathcal{C}_{FA}'(1/6)$.

Let $P = X_1 \cup \dot{\bigcup}_{i > 1}(X_i \setminus A_i)$,
with products defined within each $X_i$ but not across factors. 
Then $P$ is a pregroup and $\RSym$
succeeds on the presentation
$\langle (P \setminus 1)^\sigma \mid V_P \mid \cR \rangle$ with
$\varepsilon = 1/(2r)$.
\end{thm}

\begin{proof} 
The fact that $P$ is a pregroup  and $U(P) = F$  is established in
\cite[3.A.5.2]{Stallings}.
The set $\cI(\cR)$ consists of all normal forms of cyclic conjugates
of elements of $\cR$, and so is contained in $\widehat{\cR}$. The set
$\widehat{\cR}$  also contains
elements of $F$ that are weakly cyclically reduced but not cyclically
$P$-reduced, but these are not labels of faces in van Kampen
diagrams (coloured or otherwise). 

We shall use the approach of \RSymVerifyAtPlace
(Procedure~\ref{proc:RSymVerifyAtPlace}) 
 and consider each possible 
decomposition of a face $f$ in a diagram $\Gamma \in \calD$ that is labelled by 
$R \in \cI(\cR)$ into steps. We shall consider the cumulative
curvature of $f$, namely the curvature value $\psi$ stored as the last entry
of the 4-tuples in the list $L$ created by \RSymVerifyAtPlace. We
shall show that this 
is negative after at most
two steps,  and hence 
 $\kappa_{\Gamma}(f) \leq -1/2r$. Let the first step have length $l$, let
 $\chi_1$ be the curvature of the first step, and let $n = |R|$. Then the cumulative
 curvature $\psi_1$ after the first step is
$$\psi_1 = \chi_1 + (1 + \varepsilon)\frac{l}{n}  = \chi_1 +
\frac{2r+1}{2r} \cdot \frac{l}{n}.$$

Each element of  $\widehat{\cR}$ has length at least
$7$, by the assumption that $\widehat{\cR}$ satisfies
$\mathcal{C}_{FA}'(1/6)$. Therefore $r \ge n \ge 7$,
and in particular $(2r+1)/(2r) \leq 15/14$. By
Lemma~\ref{lem:step_bound},  the step curvature $\chi_1 \leq -1/6$, so if $l = 1$ then
$\psi_1 \leq -2/147 < 0$, and we are done. Hence we may assume that $l > 1$, and
in particular that the first edge is green. Let $S
\in \cI(\cR)$ be such that $f$ and a face $f_1$ labelled by  $S$
are incident with a consolidated  edge $e$ labelled $x_1x_2 \cdots
x_k$. Notice that $n = 6k + m$ for some
$m > 0$, by Condition $\mathcal{C}_{FA}'(1/6)$. 

If the step consists just of the edge $e$, so that we are in Case 3(a)
of Definition~\ref{def:one_step}, then 
$l = k$. Hence
$$\psi_1 \leq - \frac{1}{6} + \frac{2r+1}{2r} \cdot \frac{k}{6k+m},$$
and a short calculation shows that this is always negative.

So assume
that the step consists of $e$,  ending at a vertex
$v_1$, and then a red
edge (labelled $x_{k+1}$)  between $f$ and a blob $B_1$, then a vertex
$v_2$. Hence $l = k+1$, and a short calculation shows that
$\frac{2r+1}{2r}\cdot \frac{k+1}{6k+m} < 3/10$, so we may assume that
$\chi_1  > -3/10$. 
Notice that $\chi_1 = \chi(v_1, f, \Gamma) + \chi(B_1,
f, \Gamma) + \chi(v_2, f, \Gamma)$, and so 
in particular $\delta_G(v_1,\Gamma) 
= 2$  (so $f$ and $f_1$  are both adjacent to $B_1$),  the length
$|\partial(B_1)| \leq 4$, and $\delta_G(v_2, \Gamma) = 2$. In
particular,  $\chi_1 = \chi(B_1, f, \Gamma)$. 

Let $x_{k+1}$ 
lie in the free factor $X_i$. If the corresponding letter of
$S$ also lies in $X_i$, then 
$R$ contains a piece of length $k+1$.
Hence $\frac{k + 1}{6k+m} <
1/6$, so $m \ge 7$, and a short calculation shows that
$\frac{2r+1}{2r} \cdot \frac{k+1}{6k+ m} \leq \frac{2r+1}{2r} \cdot
\frac{k+1}{6k+7} < 1/6$, and so $\psi_1$ is negative. 
Hence we may assume that the corresponding letter of $S$ does not lie
in $X_i$, so $B_1$ contains two triangles, and 
$\chi_1 = \chi(B_1, f, \Gamma) = -1/4$.

We are therefore done if 
$\frac{2r+1}{2r} \cdot \frac{k+1}{6k+m} < 1/4$, so assume otherwise. 
Since $\frac{2r+1}{2r} \leq 15/14$, a short calculation shows that $k = 1$ and
$m \le 2$, so $n$ is $7$ or $8$. 
The assumption that $\widehat{\cR}$ satisfies $\mathcal{C}_{FA}'(1/6)$
now implies that 
there are no pieces on $R$ of length $2$.

Let $f_2$ be the green face incident with $B_1$ and
$v_2$. Then
 the edge shared by $f_2$ and $B_1$ has label from $X_i$. 
Since one edge label of $f_2$ at $v_2$ is from $X_i$ and no pieces have
length $2$, the faces $f_2$
and $f$ cannot be edge-incident after $v_2$. Since $\delta_G(v_2) =
2$, it follows that 
the next
step is  red, corresponding to a blob $B_2$. The face $f_2$ is 
edge-incident with $B_2$, so $B_2$
cannot be a single red triangle. Hence $|\partial(B_2)| \ge 4$. Thus
the curvature $\chi_2$ of this second step is at most $-1/4$, and so
$\chi_1 + \chi_2 \leq -1/2$. However, the sum of the two step lengths is
$3$, and since $(15/14) \cdot (3/7) < 1/2$, the cumulative curvature
after two steps is negative. 
\end{proof}

\subsection{Families of presentations}
We shall now consider some infinite families of presentations where
we can show by hand that $\RSym$ 
succeeds.
To help make our descriptions clear and concise, we shall not
work through every step of  $\RSymVerify$, but just extract
the parts that we need. 

For our first family of examples we consider the triangle groups. The following
result 
is well known, but it illustrates how we
can use \RSym to provide a straightforward proof.

\begin{prop}\label{prop:trigrp}
Let $G = \langle x,y \mid x^\ell, y^m, (xy)^n \rangle$ with
$2 \le \ell \le m \le n$ and $1/\ell + 1/m + 1/n < 1$. Then $\RSym$
and $\VerifySolverTrivInt$ both succeed on a pregroup presentation of $G$,
and so $G$ is hyperbolic. 
\end{prop}
\begin{proof} 
Let  $\cP$ be the  first 
pregroup presentation for $G$ from Example~\ref{ex:pregroup_pres}. Then $\cR^{\pm} = \{ R_1:=(xy)^n,\, 
R_2:=(x^{-1}y^{-1})^n =_{U(P)} (x_{\ell-1}y_{m-1})^n \}$. 
Let $\Gamma$ be a $\sigma$-reduced coloured diagram over $\cP$. 

Suppose first that $\ell \geq 3$ and hence $n \geq 4$ and $|R_1|=|R_2|
\ge 8$. 
Then it is not possible for two internal green faces of $\Gamma$  to share an edge, and so there are no
instantiable green places. Hence all steps are red and have length $1$.
Now each step curvature is at most $-1/6$ by Lemma~\ref{lem:step_bound}, so
each non-boundary face $f$ of $\Gamma$ satisfies $\kappa_{\Gamma}(f)
\leq 1 - 8 \cdot \frac{1}{6} = -1/3$, and so $\RSym$ succeeds with $\varepsilon = 1/3$.

We now show that $\VerifySolverTrivInt$ succeeds when $\ell \geq 3$, and hence
that $\RSymSolveTrivInt$  solves the word problem. 
Let $f$ be a boundary face of $\Gamma$ with $\kappa_{\Gamma}(f) > 0$. We must
show that at least half of $\partial(f)$ is on $\partial(\Gamma$).  If $f$
has at most one internal edge (which is red, so has length 1), then at least $2n-1$ edges of $f$ are
boundary, so assume that $f$ has at least two internal edges. Then the ends of 
at least two of these internal edges must intersect $\partial(\Gamma)$
non-trivially and, by Lemma~\ref{lem:tab_blob_bound}, the steps
corresponding to those edges have
step curvature at most $-1/4$.
Hence $\kappa_{\Gamma}(f) > 0$ implies that $f$ has at most four
internal 
edges,  which must be
contiguous by Lemma~\ref{lem:boundary_face}. Hence if $n > 4$ then 
more than half of $\partial(f)$ is on $\partial(\Gamma)$, and so 
\VerifySolver succeeds.
If $n = 4$, and  $f$ satisfies
$\kappa_{\Gamma}(f) > 0$ and has four contiguous internal edges, then 
the first or last such edge is incident with a boundary red blob. If 
$(a, b)$ is an intermult pair then $(a, b) \in D(P)$, so 
$\VerifySolverTrivInt$ succeeds.

So suppose for the remainder of the proof that $\ell=2$, and so $m \ge 3$. 
Then on $R_1$ there is a single instantiable (non-terminal) green place
$\bP_1 = (R_1(1,y,x),y,\sfG)$.
Furthermore, the consolidated edges
between two adjacent green faces in any diagram $\Gamma \in \calD$
have length $1$. So each step in a decomposition of $R_1$ has length at
most $2$.

Assume that  $n \geq 7$, and notice in particular this holds when $m = 3$. 
Then $\kappa_{\Gamma}(R_1) \leq -1/6$, and so $\RSym$ succeeds
with $\varepsilon = 1/6$. 
For $\RSymSolveTrivInt$ notice that the longest possible sequence of edges
between a boundary face $f$ with positive curvature and the interior of a
diagram is length $7$ (with label $y(xy)^3$) resulting in $\kappa_{\Gamma}(f) = -1/6$.
So  \VerifySolver fails when $n=7$, but this sequence of edges has a
boundary red
blob at each end, so $\VerifySolverTrivInt$ succeeds
by Remark~\ref{rem:triangles_dehn}. Note that $\VerifySolver$ succeeds
when $n \ge 8$. This completes the proof when $m=3$.

So suppose that $m \ge 4$ and hence $n \ge 5$. We shall show that the
stepwise curvature of steps of length $2$ is at most 
$-1/4$, and hence that $\RSym$ succeeds with $\varepsilon=1/4$. 
A step of length $2$
in a green face $f$ of $\Gamma$ consists of
two consolidated edges $e_1$ and $e_2$ 
of length $1$, labelled $x$ and $y$, for which the
adjacent faces are green and red, respectively. Let $v$ be the vertex
between $e_1$ and $e_2$, and $B$ the blob incident with $f$ at $e_2$.
Then $\chi(B, f, \Gamma) \leq -1/6$, by
Lemma~\ref{lem:tab_blob_bound}, so
 by Lemma~\ref{lem:vert_bounds} if $\delta_G(v,\Gamma) \geq 3$ then
the step curvature is at most $-1/3$.

Otherwise, $\delta(v,\Gamma) = 3$ and $\delta_G(v,\Gamma) = 2$. 
Then $B$ has two successive
edges both labelled $y^{-1}=y_{m-1}$.
If $\Area(B) > 1$, then $\chi(B, f, \Gamma) \leq -1/4$.
Otherwise $B$ is a triangle. But then the third edge of $\partial(B)$
is labelled $y_2$, which is not equal to $y$ or  $y^{-1}$ since $m \ge 4$.
Hence $B$ is a boundary red blob, and so again $\chi(B, f, \Gamma)
\leq -1/4$.

Finally, we consider the word problem for $m \ge 4$. Similarly to the previous
cases, the longest possible sequence of edges
between a boundary face with positive curvature and the interior is
length $5$, with label $y(xy)^2$, and so $\VerifySolverTrivInt$ succeeds.
\end{proof}

The next result is a more complicated application and is, as far we know, new.
(Note that there are general results due to Gromov and others that adding
a suitably high power of a non-torsion element to a presentation of a
hyperbolic group yields another hyperbolic group.)  

For this proof,  we shall modify \ComputeRSym, by adding an extra step, which
is currently not implemented and is used only for hand
calculations. After Step 4 of \ComputeRSym, we insert the following.
\begin{mylist2}
\item[Step 4+] (Optional) Each green face with curvature less than $-\varepsilon$, for
  some user-determined $\varepsilon$, gives some of its curvature to any
  adjacent non-boundary green faces $f$ for which $f$ has curvature greater
  than $-\varepsilon$.
\end{mylist2}

\begin{defn}\label{def:rsym+}
We shall refer to
the curvature distribution calculated by this modifiation
of $\RSym$ as $\RSym^+$.
\end{defn}

\begin{thm}
\label{thm:23mna}
Let $G = \langle x,y \mid x^2,y^3,(xy)^m, (xyxy^{-1})^n \rangle$.
Then there exists a pregroup presentation $\cP$ of $G$ such that the
following hold. 
\begin{enumerate}
\item[(i)] If $m \ge 13$ and $n \ge 7$, then $\RSym$ and $\VerifySolverTrivInt$
succeed on $\cP$.
\item[(ii)]  If $m \ge 7$ and $n \geq 19$, or if $m \ge 25$ and
$n \ge 4$, then $\RSym^+$ succeeds on $\cP$ at level $2$.
\end{enumerate}
Furthermore, $G$ is hyperbolic in all of these situations.
\end{thm}
\begin{proof}
We let $P = \{1, x, y, Y\}$, with products $x=x^{\sigma}$, $Y = y^\sigma$,
$y^2=Y$ and $Y^2=y$. So the only red blobs in diagrams $\Gamma$ in
$\calD$ are single red triangles with boundary $yyy$ or $YYY$, and we can take
$\cR^{\pm} =  \{ R_1 := (xy)^m,\,R_2 := (xY)^m,\, R_3:=(xyxY)^n\}$.

We start by listing the possible labels of consolidated edges between two green
faces in any diagram $\Gamma \in \calD$, but we omit those in which one of
the two faces is labelled $R_2$, since these correspond in an obvious way
to those with a face labelled $R_1$. In each case, we have specified the
words labelling the two faces, with the labels of the consolidated edge
positioned at the beginning of each of the two words.
\begin{enumerate}
\item\label{ce1}  $x$ between (faces labelled) $(xy)^m$ and $(xy)^m$;
\item\label{ce2}  $x$ between $(xy)^m$ and $(xyxY)^n$;
\item\label{ce3}  $x$ between $(xy)^m$ and $(xYxy)^n$;
\item\label{ce4}  $xy$ between $(xy)^m$ and $(Yxyx)^n$;
\item\label{ce5}  $xyx$ between $(xy)^m$ and $(xYxy)^n$;
\item\label{ce6}  $yx$ between $(yx)^m$ and $(xYxy)^n$;
\item\label{ce7}  $y$ between $(yx)^m$ and $(Yxyx)^n$;
\item\label{ce8}  $x$ between $(xyxY)^n$ and  $(xYxy)^n$.
\end{enumerate}

Since consolidated edges have length at most $3$, each step 
has length at most $4$ and so, if $m \geq 13$ and $n \geq 7$, then
there are least $7$ steps in any decomposition of a relator and hence $\RSym$
succeeds with $\varepsilon = 1/6$, by Lemma~\ref{lem:step_bound}. The longest possible sequence of edges
between a boundary face with positive curvature and the interior of $\Gamma$
is comprised of three steps together with a red triangle at the beginning;
and hence has length at most $13$. Any such sequence of edges has a red
triangle at each end of it, so if $m \ge 13$ and $n \ge 7$ then
$\VerifySolverTrivInt$ succeeds.
This proves Part (i) of the theorem. 

For Part (ii), assume that $m \ge 7$.  Let $f_1$ be a non-boundary 
face with boundary label $R_1 = (xy)^m$, in a diagram $\Gamma \in \calD$. 
We consider the possible decompositions of $R_1$ into steps.  A
step of length $k$ on $R_1$ constitutes a proportion
$k/(2m) \le k/14$ of $|R_1|$, and the step curvature is at most $-1/6$. 
So, when $k \le 2$, the step curvature is less than
its length requires on average for $\kappa_{\Gamma}(f_1) \leq 0$. 
In fact, the step curvature is less than required by a
factor of at least $7/6$.

Consider a consolidated edge labelled
$xy$, as in item \ref{ce4} of the list above. 
The edges of all red triangles are labelled $y$ or $Y$, and the letter
following $xy$ in $R_1$ is $x$, so such an edge must be followed by another
green consolidated edge. So this step consists of the consolidated edge
$xy$ only and hence has length $2$. Hence the step curvature is less than
required by a factor of at least $7/6$. This deals with items
\ref{ce1}, \ref{ce2}, \ref{ce3}, \ref{ce4}
and \ref{ce7} of the list above. 

Consider next a consolidated edge $e$ labelled $xyx$, as in item \ref{ce5}
of the list above, and let $v$ be the vertex at the end of $e$. 
There are two places that could come at $v$, 
namely $\bP_1 = (R_1(2,x,y),Y,R)$ and $\bP_2 = 
(R_1(2,x,y),x,\sfG)$. For $\bP_2$,
the step consists of $e$, and it is
easily checked that $\delta_G(v,\Gamma) \geq 4$ and hence $\chi(v,
f_1, \Gamma) \leq -1/4$. Since $1/4 > 3/14$, such a step gives less than its proportionate
contribution to $f_1$, 
by a factor of at least $7/6$.  
The same is true for $\bP_1$, except when $\delta(v, \Gamma) = 3$ and
$\delta_G(v,\Gamma) = 2$. 

Similar considerations apply to consolidated
edges labelled $yx$, as in item \ref{ce6} of the list above, so
 there are just two types of steps that give more  than their
proportionate contribution to $f_1$, namely those consisting of a consolidated
edge labelled $xyx$ or $yx$ together with a red edge, with the property that
the vertex in the middle of the step has total degree $3$. These steps have
lengths $4$ and $3$, respectively, and have  curvature $-1/6$.

Let us call these consolidated edges labelled $xyx$ or $yx$ in these
steps
 {\em bad consolidated
edges}. Then the other face $f_2$ incident with a bad consolidated edge is
labelled $R_3=(xyxY)^n$ and, since a bad consolidated edge on 
$R_3$ is immediately preceded by a red edge labelled $y$, there can be at most
$n$ bad consolidated edges on $\partial(f_2)$. (Note that the bad
consolidated edges of $f_2$ could also be incident with faces labelled $R_2 =
(xY)^m$, but the same restrictions apply.)

Now suppose that $n \ge 19$. Then, since the steps have length at most
$4$ and the step curvature is at most $-1/6$, 
a non-boundary face $f_2$ labelled $R_3$ satisfies
$\kappa_{\Gamma}(f_2) \leq 1 - n/6$. For
such faces, we can now apply Step 4+ of the algorithm to compute 
$\RSym^+$, as follows. Fix some small
$\varepsilon>0$. Then $f_2$ donates curvature
$-1/6 + (1 + \varepsilon)/n$
across each of its bad consolidated edges. Since there at most $n$ of these,
it still has curvature at most $-\varepsilon$ after making these
donations.

A face $f_1$ labelled $R_1$ (corresponding considerations apply to faces
labelled $R_2$) that is at dual distance at least $3$ from 
$\partial(\Gamma)$ receives at most $-1/6 + (1 + \varepsilon)/19$ curvature
across each bad consolidated edge in Step 4+ of the algorithm to
compute $\RSym^+$.
If  $f_1$ has $d$ bad consolidated edges, then $d \le m/2$,
so the curvature of $f_1$ before and after Step 4+ is at most
\[ 1 - \left( \frac{2m-4d}{2m} \cdot
 \frac{7}{6}\right) - \frac{d}{6} \quad\mbox{and}\quad
1 - \left( \frac{2m-4d}{2m} \cdot \frac{7}{6} \right) - \frac{d}{3}  +
\frac{d(1+ \varepsilon)}{19}.\]
It can be checked that this is negative for all $m \ge 7$ and $d \le
m/2$, with $\varepsilon$ close to $0$.

The proof in the case $n \ge 4$ and $m \ge 25$ is analogous, and is omitted.

To deduce hyperbolicity of $G$, we apply Theorem \ref {thm:rsymsucceeds}(i)
when \RSym succeeds at level 1. To apply Theorem \ref {thm:rsymsucceeds}(iii)
in the cases when \RSym succeeds only at level 2, we
need to know that neither $x$ nor $y $ is trivial in $G$. We could verify
that in each of the individual cases by describing a finite homomorphic
image of $G$ in which the images of $x$ and $y$ are nontrivial, but it is
quicker just to observe that if either $x$ or $y$ were trivial in $G$ then
$G$ would be finite of order at most $3$ and hence hyperbolic.
\end{proof}

It seems possible that by working a little harder we could slightly improve
the above result to include more pairs $(m,n)$, but we have not done this,
because we can use the 
\KBMAG\ package to test hyperbolicity in individual cases. By doing that,
and combining it with the result of Theorem \ref{thm:23mna}, we obtain

\begin{thm}
\label{thm:23mnb}
Let $G = \langle x,y \mid x^2,y^3,(xy)^m, (xyxy^{-1})^n \rangle$.
Then $G$ is infinite hyperbolic whenever
any of the following conditions hold.
\begin{itemize}
\item $m=7$ and $n \ge 13$;
\item $m=8$ and $n \ge 8$;
\item $m=9$ and $n \ge 7$;
\item $m=10$ and $n \ge 6$;
\item $m \ge 11$ and $n \ge 5$;
\item $m\ge 15$ and $n \ge 4$.
\end{itemize}
\end{thm}

In the cases $(m,n) = (7,12)$, $(8,7)$, $(9,6)$, $(10,5)$ and $(14,4)$,
the group $G$ is automatic and infinite. In each of these examples, by using
straightforward searches through the elements of $G$ of bounded length,
we were able to find a pair $g,h$ of commuting elements that project onto
a free abelian group of rank $2$ in an abelian quotient of a suitably
chosen subgroup of finite index in $G$. So these groups all contain free
abelian subgroups of rank $2$, and hence they are not hyperbolic.

These groups are finite for some smaller values of $m$ and $n$.
The final case to be settled was  $(m,n)=(13,4)$, which was proved
finite by coset enumeration.
In the cases $(m,n)=(7,10),\,(7,11),\,(8,6)$ and $(12,4)$, $G$ has been proved
to be infinite, and we conjecture that it is not automatic, and hence also
not hyperbolic, but we are unable to prove this.
See \cite{HavasHolt10} for details and
further references on the finiteness question.

\subsection{Further examples}
In an interesting combination of \RSym with some theoretical
arguments involving curvature, Chalk proves in \cite{Chalk} that the
Fibonacci groups $F(2,n)$ are hyperbolic for odd $n \ge 11$.
The hyperbolicity of $F(2,n)$ for even $n \ge 8$ was proved earlier in
\cite{HKM}, and that of $F(2,9)$ has been proved computationally, using \KBMAG.
Furthermore,  $F(2,n)$ is finite for $n=1,2,3,4,5,7$ and virtually free abelian
of rank $3$ (and so not hyperbolic) for $n=6$. So this completes the proof
that $F(2,n)$ is hyperbolic if and only if $n \ne 6$. Additionally, in
as  yet unpublished work, Chalk has used some concepts from \RSym to
prove that the Heineken group $\langle x, y, z \ | \ [x, [x, y]] = z,
[y, [y, z]] =  x, [z, [z, x]] = y \rangle$ is hyperbolic. 

Whilst in this paper we have only proved that \RSym naturally
generalises the classical small
cancellation conditions $C(p)$, $C'(1/p)$ and $T(q)$ over free groups,
and free products with amalgamation, we are confident that \RSym
naturally generalises a wide variety of other small
cancellation conditions. For example, in \cite{LyndonSchupp} there is a form of small
cancellation for groups constructed as HNN extensions, which we have
not analysed only because the construction of a pregroup describing an
HNN extension is a little technical. Metric small
cancellation has been defined  over graphs of groups, and used to prove
hyperbolicity of large families of groups \cite{Martin17}. Other
conditions for small cancellation over free groups have been introduced
by many authors: for example, Condition
$V(6)$ in \cite{Weiss07}. We think it is likely that \RSym
generalises most, or even all, 
of these, although it would be some work to check all of the details.

More speculatively, we believe that it is possible that more powerful
curvature distribution schemes than \RSym could be used to tackle a wide
range of problems regarding the hyperbolicity of finitely-presented
groups. \RSym, even with the modification we have called $\RSym^+$, is
rarely useful for 1-relator groups, but curvature
distribution schemes that allowed the same relator to be treated
differently in different contexts might well be useful. Similarly,
curvature distribution schemes that permitted the curvature to be
moved (bounded distances) across diagrams could be useful for the
Restricted Burnside Problem. We chose to present \RSym in this
paper because it can be tested in low-degree polynomial time, but if
one is willing to accept a higher degree polynomial cost, or perhaps a
cost with exponent the length of the longest relator, then
schemes could be devised which would prove the hyperbolicity of much
wider classes of finite presentations. 

\section{Implementation}\label{sec:implementation}

We have implemented \RSymVerify, for the case where $\cI(R) = R$ for
all $R \in \cR$, in the computer algebra systems
\GAP and \MAGMA, as \texttt{IsHyperbolic}. It is in the released
version of \MAGMA, and in the deposited \GAP package \texttt{Walrus}. 
The two implementations are
moderately different in their details, so we have used each of them as
a test of correctness of the other. 
We have provided methods to produce a pregroup whose universal group
is a given free product of free and finite
groups, as in Examples~\ref{ex:free_group} and \ref{ex:free_prod}. 
 The user is then able to add any additional relators.
We have also implemented $\VerifySolverTrivInt$ and  $\RSymSolveTrivInt$ in
\MAGMA. 

In this section we describe some run times, using the \MAGMA
version. The experiments were run on a MacBook Pro laptop with a 3.1GHz
processor, and all set $\varepsilon = 1/10$. We have not compared
timings with the \KBMAG package, as with the exception of the very
smallest presentations we found that \KBMAG did not appear to
terminate. 

We first ran \texttt{IsHyperbolic} on presentations of the form $\langle x, y
\mid x^2,  y^m, (xy)^n\rangle$, constructed as a quotient of the free product  $C_2
\ast C_m$, for $3 \leq m \leq 6$ and $n \in \{5, 10, 15\}$. As
expected, it succeeded for all $(m, n) \neq (3, 5)$. The time taken
was not noticeably dependent on $m$ or $n$ and was less than $0.01$ seconds for
each trial. 

We then tested presentations of the form $\langle x, y \mid x^2,
y^3, (xy)^m, [x, y]^n \rangle$, again constructed as a quotient of
$C_2 \ast C_3$, for $10 \leq m \leq 20$ and $6 \leq n \leq
15$. \texttt{IsHyperbolic} failed for $m \leq 12$ or $n = 6$, and otherwise
succeeded on all trials. Again, the time taken was not noticeably
dependent on $m$ or $n$ and was less than $0.01$ seconds for each trial.

We have also run experiments with randomly chosen relators, and the
results  appear in Table~\ref{tab:run_times}. For each, 
we 
take the average time for 20 sets of random relators with the given
parameters. 
After each run time we give the number
of times \texttt{IsHyperbolic} successfully proved that the group was
hyperbolic, with $\varepsilon = 1/10$. 

\begin{table}\caption{Run times averaged over 20 randomly-chosen examples}\label{tab:run_times}
{\small
\begin{center}
\begin{tabular}{l}
\hline
\hline
A free group of rank $2$ with 
    $m$ random relators of length $n$\\
\begin{tabular}{l|lll}
$m = 2$ & $n = 20$ & $30$ & $40$ \\
        & 0.02 (0) & 0.04 (20) & 0.07 (20) \\
\hline
$m = 3$ & $n = 25$ & $35$ & $45$\\
& 0.05 (0) & 0.10 (20) & 0.17 (20) \\
\hline
$m= 10$ & $n = 40$ & $50$ & $60$ \\
& 1.42 (11) & 2.21 (20) & 3.26 (20) \\
\hline
$m = 40$ & $n = 52$ & $62$ & $72$ \\
& 47.83 (12) & 70.68 (20) & 103.00 (20) 
\end{tabular}\\
\hline
\hline
A free group of rank $10$ with
    $m$ random relators of length $n$\\
\begin{tabular}{l|lll}
$m = 10$ & $n = 8$ & $20$ & $30$ \\
        & 0.13 (8)  & 1.02 (20)   & 2.33 (20) \\
\hline
$m = 20$ & $n = 10$ & $20$ & $30$ \\
& 1.02 (3) & 3.77 (20) & 6.97 (20)  \\
\hline
$m= 30$ & $n = 13$ & $20$ & $30$ \\
& 4.00 (19)  & 7.01 (20) & 12.31 (20) \\
\hline
$m = 50$ & $n = 15$ & $25$ & $35$ \\
& 8.82 (18)  &  20.02 (20) & 35.90 (20) \\
\end{tabular}\\
\hline
\hline
A free group of rank $100$ with
    $m$ random relators of length $n$\\
\begin{tabular}{l|lll}
$m = 30$ & $n = 4$ & $10$ & $20$ \\
        & 0.09 (14) & 0.91 (20) & 7.11 (20) \\
\hline
$m = 50$ & $n = 4$ & $10$ & $20$\\
& 0.33 (6) & 4.23 (20) & 41.78 (20) \\
\hline
$m= 70$ & $n = 5$ & $10$ & $50$ \\
& 1.49 (18) & 13.51 (20) & 132.21 (20) \\
\end{tabular}\\
\hline
\hline
$C_2 \ast C_3$ with
    $m$ random relators of length $n$\\
\begin{tabular}{l|lll}
$m = 1$\phantom{0} & $n = 96$ & $120$ & $160$ \\
        & 0.39 (1) & 0.59 (19) & 1.02 (20) \\
\hline
$m = 2$ & $n = 120$ & $160$ & $200$\\
& 2.33 (3) & 4.32 (19) & 6.74 (20) \\
\hline
$m= 5$ & $n = 200$ & $240$ & $280$ \\
& 40.60 (20) & 62.55 (20) & 83.64 (20) \\
\end{tabular}\\
\hline
\hline
$C_3 \ast C_3 \ast C_3$ with
    $m$ random relators of length $n$\\
\begin{tabular}{l|lll}
$m = 1$ & $n = 12$ & $24$ & $36$ \\
        & 0.01 (8) & 0.03 (19) & 0.06 (20) \\
\hline
$m = 2$ & $n = 20$ & $30$ & $40$\\
& 0.06 (5) & 0.12 (20) & 0.19 (20) \\
\hline
$m= 5$ & $n = 25$ & $55$ & $75$ \\
& 0.41 (1) & 1.82 (20) & 3.43 (20) \\
\end{tabular}\\
\hline
\hline
$C_3 \ast \mathrm{A}_5 \ast F_3$ with
    $m$ random relators of length $n$\\
\begin{tabular}{l|lll}
$m = 2$ & $n = 5$ & $10$ & $20$ \\
        & 1.74 (4) & 7.60 (19) & 38.36 (20) \\
\hline
$m = 3$ & $n = 12$ & $20$ & $30$\\
& 26.58 (19) & 98.67 (20)  & 302.00 (20) \\
\hline
$m= 5$ & $n = 15$ & $25$ & $35$ \\
& 184.92 (19) & 638.24 (20) & 1575.63 (20) \\
\end{tabular}\\
\hline
\hline
\end{tabular}
\end{center}
}
\end{table}

For random quotients of free groups we choose random, freely
cyclically reduced words of the given length as additional relators.
 For random quotients of free products of two groups we choose
random nontrivial group elements alternating between  the two 
factors. 
For random quotients
 of  three finite groups, we choose a factor at random (other
 than the previous factor) and then a random nontrivial  element from
 that factor. For free products with a nontrivial free
 factor we allow the free factor to be chosen twice in a row, but
 not then  to choose the inverse of the previously-chosen letter.

Let $F$ be a free group of rank $n$, and consider the quotient of $F$
by $r$ random, freely cyclically reduced relators of length $3$.  There are
$2m(4m^2 - 6m + 3) \sim (2m)^3$ 
such words of length $3$ over
$\{a_1^{\pm 1}, \ldots, a_m^{\pm 1}\}$, so define the \emph{density}
$d \in (0, 1)$ of the presentation by $r = (2m)^{3d}$. 
\.{Z}uk showed in \cite{Zuk} that
if $d < 1/2$ then the probability that $\cP$ defines an
infinite hyperbolic group tends to $1$ as $m \rightarrow \infty$, %
whilst if $d > 1/2$ then the probability that $\cP$ defines the
trivial group tends to $1$ as $m \rightarrow \infty$. 
These asymptotic results tell us what to expect when we choose $r$ random
cyclically reduced relators of length $3$ in the cases when $r/n$ is
either very small or very large, and it seemed interesting to study
the case when $n \to \infty$ with $r/n$ constant. We used our \MAGMA
implementation of \texttt{IsHyperbolic} to investigate this situation
experimentally, and also attempted to analyse it theoretically

Provided that we enforce our condition that there are no pieces of length $2$
in the presentation, the most common cause of failure of \RSym for moderate
values of $r/n$ is the possible
existence of an internal vertex of degree $3$ in a van Kampen diagram.
A simple calculation, of which we omit the details, shows that the expected
number of triples $\{a,b,c\}$ of distinct elements of $X$ which could
label the edges incident with such a vertex tends to $\lambda := 9(r/n)^{3}/2$ as
$n \to \infty$. Assuming that the number of such vertices forms a
Poisson distribution, this would imply that the probability of there being
no such triples would tend to $\exp(-\lambda)$. This estimate agrees
surprisingly well with our experiments with \RSym. When $r/n=1/2$,
for example, we have $\exp(-\lambda) \simeq 0.570$ and, the proportion of
successes over $1000$ runs of our implementation with $n=100$, $500$ and $1000$,
were $0.510$, $0.577$, and $0.569$.

If $d > 1/3$, then the
probability that two relators share a subword of length $2$, and hence
that our ``preprocessing step'' simplifies the presentation, tends to
$1$, and renders the presentation non-random.
 It is therefore unclear to us how to complete the analysis.

\section{Appendix: Glossary and list of notation}\label{sec:appendix}

\noindent {\bf List of mathematical terms}

\medskip
\begin{tabular}{c  c}
\begin{tabular}{l l}
\hline
Term & See \\
\hline
area & \ref{def:area}\\
blob curvature & \ref{alg:rsym}\\
boundary face, edge, vertex & \ref{def:van_kampen}\\
coloured area & \ref{def:c_area}\\
coloured (van Kampen) diagram & \ref{def:coloured_vkd}\\
consolidated edge & \ref{def:van_kampen}\\
curvature distribution & \ref{def:curv_dist}\\
curvature distribution scheme & \ref{def:curv_dist_scheme}\\
cyclic interleave & \ref{def:cyc_interleave}\\
cyclic interleave class & \ref{def:I(w)}\\
(cyclically) $\sigma$-reduced & \ref{def:reduced}\\
(cyclically) $P$-reduced & \ref{def:reduced}\\
decorated location & \ref{def:dec_location}\\
decorated place & \ref{def:dec_place} \\
decorated vertex graph & \ref{def:dec_vertex_graph}\\
dual distance & \ref{def:dual}\\
external face & \ref{def:van_kampen}\\
external word & \ref{def:van_kampen}\\
$\mathcal{G}$-vertex & \ref{def:vertex_graph}\\
green degree & \ref{def:coloured_vkd}\\
green place & \ref{def:place}\\
green-rich & \ref{def:valid}\\
half-edge & \ref{def:half-edge}\\
intermult, intermult pair & \ref{def:intermult}\\
internal face & \ref{def:van_kampen}\\
interleave & \ref{def:interleave}\\
interleave set & \ref{def:int_set}\\
intermediate place & \ref{def:one_step}\\
length of step & \ref{def:step}\\
\hline
\end{tabular} &

\begin{tabular}{l l}
\hline
Term & See \\
\hline

location  & \ref{def:location} \\
loop-minimal & \ref{def:loop_min}\\
one-step reachable & \ref{def:one_step}\\
$P$-reduced, semi-$P$-reduced & \ref{def:P_reduced}\\
place, potential place, 
 & \ref{def:place} \\
plane graph & \ref{def:delta_Ge}\\
post-interleave set & \ref{def:dec_location}\\
pregroup & \ref{def:pregroup}\\
pregroup presentation & \ref{def:pregroup_pres}\\
pregroup Dehn function & \ref{def:alternative_dehn}\\
pre-interleave set & \ref{def:dec_location}\\
$\cR$-letter & \ref{def:cr}\\
red blob & \ref{def:blob}\\
red degree & \ref{def:coloured_vkd}\\
red place & \ref{def:place}\\
$\sigma$-reduced, semi-$\sigma$-reduced & \ref{def:sigma_reduced}\\
simply connected red blob & \ref{def:blob}\\
single rewrite & \ref{def:interleave}\\
standard group presentation & \ref{def:grp_pres}\\
step & \ref{def:step}\\
stepwise curvature & \ref{def:step}\\
subdiagram & \ref{def:subdiagram}\\
succeeds with constant $\varepsilon$ & \ref{def:succeed}\\
succeeds at level $d$ & \ref{def:succeed}\\
terminal place & \ref{def:terminal}\\
universal group & \ref{def:up}\\
$V^\sigma$-letter & \ref{def:loop_min}\\
verifies a solver & \ref{def:verify}\\
vertex graph & \ref{def:vertex_graph} \\
\hline
\end{tabular}
\end{tabular}

\bigskip

\begin{tabular}{c cc c}
\begin{tabular}{l l}
{\bf List of notation} &\\
\hline
Symbol & See \\
\hline
$\Area(\Gamma)$ & \ref{def:area}\\
$\beta(B)$ & \ref{alg:rsym}\\
$\chi(v, f, \Gamma), \chi(B, f, \Gamma)$ & \ref{def:rsym} \\
$\CArea(\Gamma)$ & \ref{def:c_area}\\
$\partial$ & \ref{def:van_kampen}\\
$\calD$ & \ref{def:calD}\\
$D(P)$ & \ref{def:pregroup}\\
$\De(n)$ & \ref{def:alternative_dehn}\\
$\delta_G(v, \Gamma)$, $\delta_R(v, \Gamma)$ & \ref{def:coloured_vkd}\\
$\delta_G(e,\Gamma)$ & \ref{def:delta_Ge}\\
$\epsilon_i$ & \ref{def:edge_curve}\\
$F(X^\sigma)$ &  \ref{def:up}\\
$\mathcal{G}$ & \ref{def:vertex_graph}\\
$\mathcal{I}(\mathcal{P})$, $\mathcal{I}(\mathcal{R})$ & \ref{def:ip}\\
$\mathcal{I}(w)$ & \ref{def:I(w)}\\
$\mathcal{I}(a, b)$ & \ref{def:int_set}\\
$\kappa_\Gamma$ & \ref{alg:rsym}\\
$\OneStep(\bP)$ & \ref{alg:one_step}\\
$\mathcal{P} = \langle X^\sigma \mid V_P \mid  \mathcal{R}\rangle$ &
                                                                     \ref{def:pregroup_pres}\\
$\mathcal{P}_G$ & \ref{def:grp_pres}\\
$\PD(n)$ & \ref{def:alternative_dehn}\\
$\mathrm{Post}(R(i))$ & \ref{def:dec_location}\\
$\mathrm{Pre}(R(i))$ & \ref{def:dec_location}\\
$R(i, a, b)$ & \ref{def:location}\\
$\RSym$ & \ref{def:rsym}\\
\hline
\end{tabular}
&  & & 
\begin{tabular}{l l}
\\
\hline
Symbol & See \\
\hline
$\RSym+$ & \ref{def:rsym+}\\
$U(P)$ & \ref{def:up}\\
$\mathcal{V}$ & \ref{def:dec_vertex_graph}\\
$V_P$ & \ref{def:up}\\
$X^\sigma$  & \ref{def:up}\\
$\approx$ & \ref{def:interleave}\\
$\approx^c$ & \ref{def:cyc_interleave}\\
$[ab]$ & \ref{def:pregroup} \\ 
\hline
\\
\\
{\bf Procedures} \\
\hline
Name & See \\
\hline
\Blob & \ref{alg:blob}\\
\FindEdges & \ref{alg:find_edges}\\
\texttt{ComputeOneStep} & \ref{alg:one_step}\\
\ComputeRSym  & \ref{alg:rsym}\\
\RSymIntVerify & \ref{proc:RSymIntVerify}\\
\RSymSolve & \ref{alg:RSymSolve}\\
\RSymSolveTrivInt & \ref{rem:triangles_dehn}\\
\RSymVerify & \ref{proc:RSymVerify} \\
\RSymVerifyAtPlace & \ref{proc:RSymVerifyAtPlace}\\
\VerifySolver & \ref{proc:VerifySolver}\\
\VerifySolverTrivInt & \ref{rem:triangles_dehn} \\
\VerifySolverAtPlace & \ref{proc:VerifySolverAtPlace}\\
\Vertex & \ref{alg:vertex}\\
\hline
\end{tabular}
\end{tabular}

\paragraph{Acknowledgements} This work was supported by
EPSRC grant number EP/I03582X/1. We would like to thank P.E. Holmes for
many useful conversations about the fundamental ideas underpinning
this project. We would also like to thank Simon Jurina for a very
careful reading of a final draft of this article.

\paragraph{Addresses}
Derek Holt, Mathematics Institute, University of Warwick, Coventry,
CB4 7AL, UK. \texttt{D.F.Holt@warwick.ac.uk}
\\
Steve Linton, 
School of Computer Science, University of St Andrews, Fife KY16
9SX, UK. \texttt{steve.linton@st-andrews.ac.uk}\\ 
Max Neunh\"offer, Im Bendchen 35a, 50169 Kerpen, Germany. \texttt{max@9hoeffer.de}\\
Richard Parker, 70 York St, Cambridge CB1 2PY,
UK. \texttt{richpark7920@gmail.com} \\
Markus Pfeiffer, 2 Orchid Cottage, Wester Balrymonth, St Andrews, Fife
KY16
8NN, UK.  \texttt{markus.pfeiffer@st-andrews.ac.uk}\\
Colva M. Roney-Dougal,   Mathematical Institute, University of St
Andrews, Fife KY16 9SS, UK. \texttt{colva.roney-dougal@st-andrews.ac.uk}

\end{document}